\pdfoutput=1
\RequirePackage{ifpdf}
\ifpdf 
\documentclass[pdftex]{sigma}
\else
\documentclass{sigma}
\fi

\numberwithin{equation}{section}

\newtheorem{Theorem}{Theorem}[section]
\newtheorem{Corollary}[Theorem]{Corollary}

\newtheorem{Proposition}[Theorem]{Proposition}
 { \theoremstyle{definition}
\newtheorem{Definition}[Theorem]{Definition}

\newtheorem{Example}[Theorem]{Example}
\newtheorem{Remark}[Theorem]{Remark} }

\newcommand{\abss}[1]{\left\vert #1 \right\vert}
\newcommand{\rcap}{\tikz[scale = 0.3] \draw[->] (0,0) arc(180:0:1);}

\newcommand{\lcap}{\tikz[scale = 0.3] \draw[<-] (0,0) arc(180:0:1);}
\newcommand{\lcup}{\tikz[scale = 0.3] \draw[<-] (0,0) arc(-180:0:1);}

\usepackage{tikz}
\usetikzlibrary{arrows,decorations.pathmorphing,positioning}
\usepackage{mathrsfs}
\usepackage{bm}

\newcommand{\Kaufhoriz}{\begin{tikzpicture}[scale = 0.5, baseline = -3pt]
\draw[dashed] (0,0) circle(1);
\draw[thick] (45:1) arc (-45:-135:1);
\draw[thick] (-45:1) arc (45:135:1);
\end{tikzpicture}}
\newcommand{\Kaufvert}{\begin{tikzpicture}[scale = 0.5, baseline = -3pt]
\draw[dashed] (0,0) circle(1);
\draw[thick] (45:1) arc (135:225:1);
\draw[thick] (135:1) arc (45:-45:1);
\end{tikzpicture}}
\newcommand{\Kauftresse}{\begin{tikzpicture}[scale = 0.5, baseline = -3pt]
\draw[dashed] (0,0) circle(1);
\draw[thick] (-45:1) -- (135:1) node[fill, color= white, pos = 0.5, circle, scale = 0.45]{};
\draw[thick] (45:1) -- (-135:1);
\end{tikzpicture}}
\newcommand{\Kaufcirc}{\begin{tikzpicture}[scale = 0.5, baseline = -3pt]
\draw[dashed] (0,0) circle(1);
\draw[thick] (0,0) circle(0.7);
\end{tikzpicture}}
\newcommand{\Kaufvide}{\begin{tikzpicture}[scale = 0.5, baseline = -3pt]
\draw[dashed] (0,0) circle(1);
\end{tikzpicture}}

\newcommand{\unortwist}{\begin{tikzpicture}[scale = 0.35, baseline = 5pt]
\draw[thick] (1,0.1) .. controls (1,2) and (0,2) .. (0,1) node[fill, color= white, pos = 0.21, circle, scale = 0.55]{};
\draw[thick] (1,1.9) .. controls (1,0) and (0,0) .. (0,1);
\end{tikzpicture}}

\newcommand{\halftwist}{\begin{tikzpicture}[yscale = 0.3, xscale = 0.4, baseline = 2pt]
\draw[thick] (1,0) .. controls (1,1) and (0,1) .. (0,2) node[fill, color= white, pos = 0.5, circle, scale = 0.5]{};
\fill[gray!20] (0,0) .. controls (0,1) and (1,1) ..(1,2)-- (0,2) .. controls (0,1) and (1,1) .. (1,0) -- (0,0);
\draw (0,0) .. controls (0,1) and (1,1) .. (1,2);
\end{tikzpicture}\ }

\newcommand{\St}{\operatorname{St}}
\newcommand{\halft}{{\operatorname{ht}}}

\newcommand{\Free}{\operatorname{Free}}
\newcommand{\colim}{\operatorname{colim}}
\newcommand{\Hom}{\operatorname{Hom}}

\newcommand{\End}{\operatorname{End}}
\newcommand{\undHom}{\underline{\operatorname{Hom}}}
\newcommand{\undEnd}{\underline{\operatorname{End}}}

\newcommand{\bmrhd}{\bm{\triangleright}}
\newcommand{\bmotimes}{\bm{\otimes}}
\newcommand{\Oqcomfin}{\Oq\text{-}{\rm comod}^{\rm fin}}
\newcommand{\Oqcom}{\Oq\text{-}{\rm comod}}
\newcommand{\Oq}{\mathcal{O}_{q^2}({\rm SL}_2)}
\newcommand{\Uq}{\mathcal{U}_{q^2}(\mathfrak{sl}_2)}

\newcommand{\Uqmod}{\Uq\text{-}{\rm mod}}

\tikzset{leftymissy/.style={path picture={\draw[black](path picture bounding box.north west) -- (path picture bounding box.north east) -- (path picture bounding box.south east) -- (path picture bounding box.south west);}}}
\tikzset{rightymissy/.style={path picture={\draw[black](path picture bounding box.south east) -- (path picture bounding box.south west) -- (path picture bounding box.north west) -- (path picture bounding box.north east);}}}
\tikzset{topydowny/.style={path picture={\draw[black](path picture bounding box.north west) -- (path picture bounding box.north east)
(path picture bounding box.south east) -- (path picture bounding box.south west);}}}

\begin{document}
\allowdisplaybreaks

\newcommand{\arXivNumber}{2104.13848}

\renewcommand{\PaperNumber}{042}

\FirstPageHeading

\ShortArticleName{Relating Stated Skein Algebras and Internal Skein Algebras}

\ArticleName{Relating Stated Skein Algebras\\ and Internal Skein Algebras}

\Author{Benjamin HA\"IOUN}
\AuthorNameForHeading{B.~Ha\"ioun}
\Address{Institut de Math\'ematiques de Toulouse, France}
\Email{\href{mailto:benjamin.haioun@ens-lyon.fr}{benjamin.haioun@ens-lyon.fr}}
\URLaddress{\url{https://perso.math.univ-toulouse.fr/bhaioun/}}

\ArticleDates{Received October 07, 2021, in final form May 25, 2022; Published online June 11, 2022}

\Abstract{We give an explicit correspondence between stated skein algebras, which are defined via explicit relations on stated tangles in [Costantino~F., L\^e~T.T.Q., arXiv:1907.11400], and internal skein algebras, which are defined as internal endomorphism algebras in free cocompletions of skein categories in [Ben-Zvi~D., Brochier A., Jordan D., \textit{J.~Topol.} \textbf{11} (2018), 874--917, arXiv:1501.04652] or in [Gunningham~S., Jordan~D., Safronov~P., arXiv:1908.05233]. Stated skein algebras are defined on surfaces with multiple boundary edges and we generalise internal skein algebras in this context. Now, one needs to distinguish between left and right boundary edges, and we explain this phenomenon on stated skein algebras using a half-twist. We prove excision properties of multi-edges internal skein algebras using excision properties of skein categories, and agreeing with excision properties of stated skein algebras when $\mathcal{V} = \mathcal{U}_{q^2}(\mathfrak{sl}_2)\text{-}{\rm mod}^{\rm fin}$. Our proofs are mostly based on skein theory and we do not require the reader to be familiar with the formalism of higher categories.}


\Classification{57K16; 18M15}

\section{Introduction}

The goal of this paper is to draw a comprehensive link between stated skein algebras of \cite{Cos} and internal skein algebras of \cite{Gunningham} and \cite{BBJ}, including their structures and properties. The interest of such a link is to benefit from the nice features of both sides. Stated skein algebras are defined very explicitly, and have numerous applications. Internal skein algebras are much more theoretical, they are defined for all ribbon categories of coefficients, and their basic (in particular, excision) properties derive formally.

\subsection*{Skein categories} Skein algebras, introduced by Przytycki and Turaev, are a generalisation of the Kauffman bracket polynomial to surfaces. The Kauffman bracket $\langle-\rangle$ is the quotient map from the vector space freely generated by isotopy classes of framed links in $\mathbb{R}^3$ to the quotient of this vector space by the local Kauffman-bracket relations:
\[ \Kauftresse \ = q \Kaufvert\ +q^{-1} \Kaufhoriz \qquad \text{and} \qquad \Kaufcirc \ = \big({-}q^2-q^{-2}\big) \Kaufvide\ . \]
Every link in $\mathbb{R}^3$ can be reduced in an essentially unique way to the empty link via some Kauffman-bracket relations, thus $\langle-\rangle$ takes values in the ground field $k$. The skein algebra of an oriented surface~$\Sigma$ is the quotient of the vector space freely generated by isotopy classes of framed links in $\Sigma\times (0,1)$ by the Kauffman-bracket relations. Its algebra structure is given by stacking in the $(0,1)$-coordinate.

It is interesting to extend these constructions to form a category, so one can cut and glue pieces of links together. One considers the category of framed tangles instead of the vector space of framed links. Its objects are finite sequences of framed points (in $\mathbb{R}^2$, or in $\Sigma$) and its morphisms are isotopy classes of framed tangles linking them (in $\mathbb{R}^2\times [0,1]$, or in $\Sigma\times [0,1]$). Then the skein algebra is the quotient of the endomorphism algebra of the empty set by the Kauffman-bracket relations.

Skein categories, introduced in \cite{Walker} and \cite{JohnsonFreyd} (see also \cite{Cooke}), are a generalisation of the Reshe\-ti\-khin--Turaev construction to surfaces, as well as a categorical extension of skein algebras. Given a $k$-linear ribbon category~$\mathcal{V}$, in \cite{RT} (see also~\cite{Turaev}) one constructs a functor ${\rm RT}$ from the category ${\rm Tan}_\mathcal{V}^{\rm fr}$ of framed oriented coloured tangles in $\mathbb{R}^2\times [0,1]$ to~$\mathcal{V}$. If one adds coupons to framed tangles, this functor is surjective on morphisms. One can quotient by the relations which are true in~$\mathcal{V}$ after ${\rm RT}$, which gives a category equivalent to~$\mathcal{V}$. The skein category ${\rm Sk}_\mathcal{V}(\Sigma)$ of an oriented surface $\Sigma$ with coefficients in~$\mathcal{V}$ is the quotient of the category of framed tangles with coupons in $\Sigma\times [0,1]$ by the local relations which are true in~$\mathcal{V}$ after ${\rm RT}$ on a little disk on $\Sigma$. We are particularly interested in the case where~$\mathcal{V}$ is the category of type~1 finite-dimensional modules over the ribbon Hopf alge\-bra~$\Uq$ for gene\-ric~$q$, which is equivalent to the category of finite-dimensional right comodules over the coribbon Hopf algebra $\Oq$. In this case every coupon is a linear combination of tangles so there is no need for coupons and the local relations which are true in~$\mathcal{V}$ after ${\rm RT}$ are exactly the Kauffman-bracket relations. Then the endomorphism algebra of the empty set is exactly the skein algebra.

This categorical extension comes with much more information and structure than skein algebras. A~boundary edge $c$ of a surface $\Sigma$ provides ${\rm Sk}_\mathcal{V}(\Sigma)$ with a structure of ${\rm Sk}_\mathcal{V}(c\times(0,1))$-module category, and the skein category of a gluing of two surfaces along a boundary edge is equivalent to the Tambara relative tensor product of the two module categories. We will recall these definitions and constructions in Section~\ref{SectSkCat}. The skein category construction moreover satisfies nice functoriality properties with respect to embeddings and their isotopies.

One would like to reduce skein categories to an algebra without losing these extra information and structure. This is the point of both stated skein algebras and internal skein algebras, which happen to be isomorphic.

\subsection*{Stated skein algebras} Without going through this categorical extension/algebraic reduction story, one would like to have nice formulas to compute skein algebras of a gluing of surfaces. Skein algebras are not suited for this kind of operation because one cannot cut a link, and most links on a gluing cannot be described by some links on each part. Therefore, it is natural to consider an extension of skein algebras that allows links to be cut, namely considering tangles, with endpoints above the boundary of the surface (in $\partial \Sigma \times (0,1)$), ready to be glued to other tangles. Instead of oriented surfaces one now works with marked surfaces that have boundary intervals above which one allows tangles to end. Now any tangle $\alpha$ on a gluing of two surfaces $\Sigma_1 \cup_c \Sigma_2$ can be cut on the gluing arc $c$ to give a pair of tangles $(\alpha \cap \Sigma_1, \alpha \cap \Sigma_2)$, one on each surface. This procedure induces a morphism from the vector space of tangles in the gluing to the tensor product of the vector spaces of tangles in both surfaces. This cutting however seems to highly depend on the choice of the representative $\alpha$ in its isotopy class, thus one has to quotient the vector space of tangles by some boundary relations to make this well-defined on the level of isotopy classes. Stated skein algebras were introduced in \cite{BonahonWong} and some variants in \cite{Muller}. They were later refined in \cite{TTQLe} and further studied in \cite{Cos}. The authors define the stated skein algebra $\mathscr{S}(\mathfrak{S})$ of a marked surface~$\mathfrak{S}$ as the algebra of unoriented framed tangles in $\mathfrak{S} \times (0,1)$ with states $+$ or $-$ on the endpoints, modulo the Kauffman-bracket relations and two additional boundary relations. It makes the cutting construction well-defined, namely one has an algebra morphism $\rho_c\colon \mathscr{S}(\Sigma_1 \cup_c \Sigma_2) \to \mathscr{S}(\Sigma_1) \otimes \mathscr{S}(\Sigma_2)$ in the situation above, see Theorem~\ref{splitmorph}.

For a boundary edge $c$, we want an action on the boundary similar to the ${\rm Sk}_\mathcal{V}(c\times(0,1))$-module category structure on ${\rm Sk}_\mathcal{V}(\mathfrak{S})$. This is done by cutting out a bigon
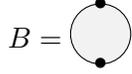
\[
B= \begin{tikzpicture}[scale = 0.8, baseline = -5pt]
\fill [gray!10] (0,0) circle (0.5);
\draw (0,0) circle (0.5);
\node at (0,-0.5){\small $\bullet$};
\node at (0,0.5){\small $\bullet$};
\end{tikzpicture}
\]
 from the surface near this boundary edge. The stated skein algebra of the bigon $\mathscr{S}(B)$ has a structure of Hopf algebra given by the cutting construction, and actually of coribbon Hopf algebra by very geometric means. It is isomorphic to the Hopf algebra $\Oq$. The stated skein algebra $\mathscr{S}(\mathfrak{S})$ is then naturally a comodule over $\mathscr{S}(B)$. It is either a right or a left $\mathscr{S}(B)$-comodule depending on whether one sees the boundary edge at the right or at the left.

For a gluing of two surfaces, stated skein algebras satisfy very simple excision properties expressed by their $\mathscr{S}(B)$-comodule structures.

\subsection*{Internal skein algebras} Sticking to the skein category, one could ask if it can be completely described by an algebra, namely such that ${\rm Sk}_\mathcal{V}(\Sigma)$ is equivalent to a category of modules over this algebra. When one has a boundary edge on a connected surface $\Sigma$, one has a ${\rm Sk}_\mathcal{V}\big(\mathbb{R}^2\big)$-action $\rhd$ on ${\rm Sk}_\mathcal{V}(\Sigma)$ by this boundary edge and it is essentially surjective on objects. The internal $\Hom$ object with respect to this action captures the idea of being described by a single object. Namely, one asks whether there exists an object $A_\Sigma$ such that $\Hom_{{\rm Sk}_\mathcal{V}(\Sigma)}(V\rhd\varnothing,\varnothing) \simeq \Hom_{\mathcal{V}}(V, A_\Sigma)$ naturally in $V \in \mathcal{V}$. Here $V$ is seen as an object of $\mathcal{V}$ in the right hand side, and as an object of ${\rm Sk}_\mathcal{V}(\mathbb{R}^2)$, a single point coloured by $V$, in the left hand side. Actually by rigidity of $\mathcal{V}$ this implies that $\Hom_{{\rm Sk}_\mathcal{V}(\Sigma)}(V\rhd\varnothing,W \rhd \varnothing) \simeq \Hom_{\mathcal{V}}(V,W \otimes A_\Sigma)$ hence that all morphisms in ${\rm Sk}_\mathcal{V}(\Sigma)$ are described by morphisms in $\mathcal{V}$ involving this object $A_\Sigma$. It turns out that this object does exist, though not necessarily as an object of $\mathcal{V}$ but rather of its ``free cocompletion''~$\mathcal{E}$, where one formally adds all colimits. One now has an equivalence of categories between the free cocompletion of ${\rm Sk}_\mathcal{V}(\Sigma)$ and the category of right modules over $A_\Sigma$ in $\mathcal{E}$, see Theorem~\ref{thASigmaMods}. In the case where $\mathcal{V}$ is the semisimple category $\Oqcomfin$, in the free cocompletion one merely needs to add infinite direct sums. It means that $A_\Sigma$ is a possibly infinite-dimensional $\Oq$-comodule, and indeed stated skein algebras usually are. The structure of internal $\Hom$ object comes with its own composition and endows $A_\Sigma$ with an algebra structure.

Note that we use the term internal skein algebra from \cite{Gunningham} though we use the definition as internal endomorphism algebra from \cite{BBJ}, referred there as moduli algebras. This name refers to Alekseev--Grosse--Schomerus moduli algebras which, though isomorphic, have a very different definition, and the first denomination seemed more appropriate to our context. The notion of a free cocompletion is only defined up to canonical equivalence, but in~\cite{Gunningham} the authors chose to work over a standard choice, the category of presheaves $[\mathcal{V}^{\rm op},{\rm Vect}]$. Similarly, the internal endomorphism algebra of an object is only defined up to unique isomorphism, but in~\cite{Gunningham} the authors chose a standard choice, namely the one of Proposition~\ref{propIntHomFree}. In the case of $\mathcal{V} = \Oqcomfin$, there is a canonical equivalence
\[
R\colon\  \begin{array}{ccc}
 \Oqcom &\tilde{\to}& [\mathcal{V}^{\rm op},{\rm Vect}],
 \\
 X &\mapsto& \Hom_{\Oqcom}(-,X),
\end{array}
\]
and we will compare the stated skein algebras which are objects of $\Oqcom$ to internal skein algebras which are objects of $[\mathcal{V}^{\rm op},{\rm Vect}]$ using this equivalence. Actually, because internal endomorphism algebras from \cite{BBJ} are only defined up to unique isomorphism, and in any free cocompletion, we~can say that the stated skein algebra \emph{is} the internal endomorphism algebra.

Note also that in \cite{Gunningham} and \cite{BBJ} one considers a right ${\rm Sk}_\mathcal{V}\big(\mathbb{R}^2\big)$-action $\lhd$ on ${\rm Sk}_\mathcal{V}(\Sigma)$ which induces a~braided opposite algebra structure compared to the left action. We will rather use the left action as the algebra structure then coincides with the one on stated skein algebras (namely $\alpha . \beta$ is $\alpha$ above $\beta$).

\subsection*{Relating stated skein algebras and internal skein algebras} When $\mathcal{V} = \Oqcomfin$, we show that these two algebras are isomorphic as $\Oq$-comodule-algebras. This result was already known as folklore and stated in a slightly weaker form (only as algebras in $\rm Vect$) in \cite[Theorem~4.4]{LeYu}, \cite[Theorem~9.1]{LeSikora} and \cite[Remark 2.21]{Gunningham}. Note that in their context the isomorphism a priori cannot be improved to an $\Oq$-comodule-algebra isomorphism because of the product inversion coming from seeing the boundary edge at the right. On the other hand, both algebras are known to be isomorphic to the Alekseev--Grosse--Schomerus moduli algebras, which are deformation quantizations of the representation variety of a punctured surface in the direction of the Fock--Rosly Poisson structure. In \cite[Theorem~5.3]{Faitg} and \cite{Korinman} the authors provide an $\Oq$-comodule algebra isomorphism between these AGS algebras and stated skein algebras. In \cite[Section~7]{BBJ} the authors provide an isomorphism between AGS algebras and internal skein algebras. Note though that one does not see the braided opposite algebra structure through these isomorphisms, and it might be that one deforms the representation variety in the direction of the opposite Poisson structure. Section~\ref{SectRelation} gives a more direct proof of the isomorphism using only skein theory.
\begin{Theorem}[Theorem~\ref{thmRelation}]
For $\mathcal{V} = \Oqcomfin$ and $\Sigma$ a compact oriented surface with a~boundary edge $e$ seen at the right for stated skein algebras, and at the left for internal skein algebras, one has an isomorphism of $\Oq$-comodule-algebras $A_\Sigma \simeq \mathscr{S}(\Sigma)$.
\end{Theorem}
We explicitly give the natural isomorphisms $\St_W\colon\Hom_{{\rm Sk}_\mathcal{V}(\Sigma)}(W\rhd\varnothing,\varnothing) \tilde{\to} \Hom_{\mathcal{V}}(W,\mathscr{S}(\Sigma))$ exhi\-bi\-ting $\mathscr{S}(\Sigma)$ as the internal endomorphism algebra of the empty set in ${\rm Sk}_\mathcal{V}(\Sigma)$, without relying on excision properties. We only give $\St_W$ for $W\in {\rm Sk}_\mathcal{V}(\mathbb{R}^2)$ a well-ordered set of points coloured by the standard corepresentation $V$ of basis $v_+$,~$v_-$~-- equivalently, an object of the Temperley--Lieb full subcategory ${\rm TL}$~-- which is sufficient as these objects generate the whole category under colimits. This natural isomorphism is the expected one. Any morphism in the left hand side can be described by a~(linear combination of) tangle $\alpha$ with $n$ boundary points coloured by $V$. Then $\St_{V^{\otimes n}}(\alpha)\colon V^{\otimes n} \to \mathscr{S}(\Sigma)$ maps a simple tensor $v_{\varepsilon_1}\otimes\cdots \otimes v_{\varepsilon_n}$ to the tangle~$\alpha$ with states $\varepsilon_1,\dots,\varepsilon_n$ from top to bottom.

\subsection*{Multi-edges} Internal skein algebras are usually defined on connected surfaces with a single boundary edge, and we extend the definition to marked surfaces $\mathfrak{S}$ with multiple boundary edges, so multiple actions. It is simply the internal endomorphism object of the empty set in ${\rm Sk}_\mathcal{V}(\mathfrak{S})$ with respect to the ${\rm Sk}_\mathcal{V}\big(\mathbb{R}^2\big)^{\otimes n}$-module structure. The multi-edge internal skein algebra $A_\mathfrak{S}$ thus obtained is an algebra object in $\Free(\mathcal{V}^{\otimes n})\simeq \mathcal{E}^{\boxtimes n}$. Moreover, we now allow boundary edges to be seen either as left or as right edges. The corresponding left or right ${\rm Sk}_\mathcal{V}\big(\mathbb{R}^2\big)$-actions differ by rotating the picture by 180 degrees, which is homotopic to the identity by the half twist \halftwist. We give an explicit way to compare internal skein algebras obtained from right and left edges using the half twist, which induces a braided opposite algebra structure:

\begin{Proposition}[Proposition~\ref{propASigRightBrOpp}]
Let $\Sigma$ be a surface with a boundary edge $e$ and $A_\Sigma$ the internal skein algebra of $\Sigma$ with $e$ seen at the left. Pre-composition with the half twist gives a~natural isomorphism~$\sigma^R$ exhibiting $A_\Sigma$ as the internal skein algebra with $e$ seen at the right. The right algebra structure $m^R\colon A_\Sigma \otimes A_\Sigma \to A_\Sigma$ differs from the left one $m\colon A_\Sigma \otimes A_\Sigma \to A_\Sigma$ by a~braiding: $m^R = m \circ c_{A_\Sigma \otimes A_\Sigma}$.
Namely, the ``right'' internal skein algebras introduced in~{\rm \cite{Gunningham}} and~{\rm \cite{BBJ}} are the braided opposites of the ``left'' ones introduced in Section~$\ref{SectASigma}$.
\end{Proposition}

The skein category can still be rebuilt as the category of modules over the internal skein algebra as soon as there is at least one boundary edge per connected component, i.e., when all objects ``come from the action'':
\begin{Theorem}[Theorem~\ref{thmReducedSkModASig}]
The free cocompletion of the full subcategory of ${\rm Sk}_\mathcal{V}(\mathfrak{S})$ generated by points near a boundary edge is equivalent to the category of right $A_\mathfrak{S}$-modules in $\mathcal{E}^{\boxtimes n}$, where~$\mathcal{E}^{\boxtimes n}$ has opposite monoidal structure on coordinates associated with right edges.
\end{Theorem}

Having multiple edges now, one can state excision for a gluing of two surfaces. This follows from the excision properties of skein categories.
\begin{Theorem}[Theorem~\ref{thmExcisionAFS}]
Let $\mathfrak{S}_1$ be a marked surface with $n_1+k$ boundary edges with a~sequence of $k$ right boundary edges $\vec{c}_1$ numbered $\vec{k}_1$ and $\mathfrak{S}_2$ a marked surface with $n_2+k$ boundary edges with a sequence of $k$ left boundary edges $\vec{c}_2$ numbered $\vec{k}_2$. One has two thick embeddings $\mathfrak{S}_1 \hookleftarrow C \hookrightarrow \mathfrak{S}_2$ where $C = \sqcup^k (0,1)$. Let $\mathfrak{S} = \mathfrak{S}_1 \cup_{\vec{c}_1=\vec{c}_2}\mathfrak{S}_2$ be their collar gluing, one has an isomorphism $A_\mathfrak{S} \simeq (A_{\mathfrak{S}_1}, A_{\mathfrak{S}_2})^{{\rm inv}_{\vec{k}_1,\vec{k}_2}}$ in $\mathcal{E}^{\boxtimes n_1+n_2}$ between the internal skein algebra of the gluing and the invariants of the tensor product of the internal skein algebras of the two surfaces.
\end{Theorem}

Hence one has a good candidate for the algebraic reduction of the skein category.

When $\mathcal{V} = \Oqcomfin$ and all edges are seen at the left, one has an equivalence of categories $\Free(\mathcal{V}^{\otimes n}) \simeq \Oq^{\otimes n}\text{-}$comod and multi-edge internal skein algebras are isomorphic to stated skein algebras as $\Oq^{\otimes n}$-comodule algebras by the same arguments as in Section~\ref{SectRelation}. We explain more explicitly what happens when one considers right edges in this context, and use a geometric half-coribbon structure on $\mathscr{S}(B)$ which is needed to fully describe the relation with stated skein algebras.
\begin{Theorem}[Theorem~\ref{thmRelationMultiAndRight}]
Let $\mathfrak{S}$ be a marked surface with $n$ boundary edges labelled either as left or as right edges. There is an isomorphism of $\big(\Oq^{\otimes k},\Oq^{\otimes n-k}\big)$-bicomodule algebras $A_{\mathfrak{S}}^{L_{k+1},\dots,L_n} \simeq \mathscr{S}(\mathfrak{S})$, between $\mathscr{S}(\mathfrak{S})$ seen as a left $\Oq$-comodule on right edges, and $A_\mathfrak{S}$ with the comodule structures of the right edges switched at the left using the antipode.
\end{Theorem}

We check that the excision properties expressed under this isomorphism are indeed the same.

\section{Preliminaries}\label{SectPreliminaries}
We recall basic facts about the quantum groups $\Oq$ and $\Uq$, and the needed constructions and properties of skein categories in~\cite{Cooke} and of stated skein algebras in \cite{Cos}. We finish with a slight modification of the construction of internal skein algebras of \cite{Gunningham}. Let $k$ be either $\mathbb{Q}\big(q^{\frac{1}{2}}\big)$ or $\mathbb{C}$ with $q^{\frac{1}{2}} \in \mathbb{C}^\times$ generic, i.e., not a root of unity. We work in the symmetric monoidal (2,1)-category ${\rm Cat}_k$ of small $k$-linear categories, $k$-linear functors and natural isomorphisms, where the tensor product $\mathcal{C} \otimes \mathcal{D}$ has objects $\operatorname{Ob}(\mathcal{C}) \times \operatorname{Ob}(\mathcal{D})$ and morphisms $\Hom_{\mathcal{C} \otimes \mathcal{D}}((c,d),(c',d')) := \Hom_\mathcal{C}(c,c') \otimes_k \Hom_\mathcal{D}(d,d')$.

\subsection[The coribbon Hopf algebra O\_\{q\textasciicircum{}2\}(SL\_2)]
{The coribbon Hopf algebra $\boldsymbol{\Oq}$}\label{SectRibCat}

We recall the definition of the coribbon Hopf algebra $\Oq$, its categories of right comodules and their links with left $\Uq$-modules.
\begin{Definition}
The coribbon Hopf algebra $\Oq$ is the free $k$-algebra with
(one should read the matrix equations component-wise, and the tensor product of matrices should be computed as a usual matrix product with tensor products of coefficients instead of products)
\begin{alignat*}{3}
& \text{generators}\colon&& a,\ b,\ c,\ d,&
\\
& \text{relations}\colon&& ca = q^2ac,\quad db = q^2bd,\quad ba = q^2ab,\quad dc = q^2cd,&
\\
&&& bc = cb,\quad ad - q^{-2}bc = 1\quad \text{and}\quad da - q^2cb = 1,&
\\
&\text{coproduct}\colon&&\Delta\begin{pmatrix} a & b \\[-.5ex] c & d \end{pmatrix}
= \begin{pmatrix} a & b \\[-.5ex] c & d \end{pmatrix} \otimes \begin{pmatrix} a & b \\[-.5ex] c & d \end{pmatrix}\!,&
\\
& \text{counit}\colon&& \varepsilon \begin{pmatrix} a & b \\[-.5ex] c & d \end{pmatrix}
= \begin{pmatrix} 1& 0 \\[-.5ex] 0 & 1\end{pmatrix}\!,&
\\
& \text{antipode}\colon&& S \begin{pmatrix} a & b \\[-.5ex] c & d \end{pmatrix}
= \begin{pmatrix} d & -q^2b \\[-.5ex] -q^{-2}c & a \end{pmatrix}\!,&
\\
&\text{co-}R\text{-matrix}\colon&& R \begin{pmatrix}
a \otimes a & a \otimes b & b \otimes a & b \otimes b
\\
a \otimes c & a \otimes d & b \otimes c & b\otimes d\\ c\otimes a & c\otimes b & d \otimes a & d\otimes b\\ c\otimes c & b\otimes d & d\otimes c & d\otimes d \end{pmatrix}
=\begin{pmatrix} q& 0 & 0 & 0\\ 0 & q^{-1} & q-q^{-3} & 0\\ 0 & 0 & q^{-1}&0\\ 0 & 0 & 0 & q \end{pmatrix}\!,&
\\
& \text{coribbon functional}\colon \quad && \theta \begin{pmatrix} a & b \\[-.5ex] c & d \end{pmatrix}
= \begin{pmatrix} -q^3 & 0 \\[-.5ex] 0 & -q^3 \end{pmatrix}\!.&
\end{alignat*}
\end{Definition}

\begin{Remark}
We took the coribbon functional so that the twist on a comodule $V$ is given by $\theta_V\colon V \overset{\Delta_V}{\to} V \otimes H \overset{{\rm Id}_V \otimes \theta}{\to} V$. In the literature (see \cite{Wakui} for the coribbon case, or \cite{Kassel} for the ribbon case), the coribbon functional is defined to be $\theta^{-1}$ in our definition.
Moreover, we use here an unusual coribbon functional, which arises naturally on the stated skein algebra of the bigon in Section~\ref{SectStatedSkein}. It is studied in \cite{Tingley} where the author proves that it gives precisely the Kauffman-bracket relation under the Reshetikhin--Turaev functor, whereas the usual coribbon functional gives relations which differ by a sign, and give the Jones polynomial after writhe renormalisation.\looseness=-1
\end{Remark}

\begin{Definition}\label{defOqcomods}
The quantum plane $\mathbb{A}_{q^2}^2$ is the free $k$-algebra generated by $x$ and $y$ modulo the relation $yx = q^2xy$. It is a right $\Oq$-comodule algebra with $\Delta \left(\begin{smallmatrix} x&y \end{smallmatrix}\right) = \left(\begin{smallmatrix} x&y \end{smallmatrix}\right) \otimes \left( \begin{smallmatrix} a & b \\ c & d \end{smallmatrix}\right)$ on generators.

Its subspace of homogeneous polynomials of degree $n$ forms a sub-comodule which we denote by~$V_n$.
In particular the generators span the comodule $V_1$ which we abbreviate $V$ and call the standard co-representation of $\Oq$. It is more conventional to call its basis $v_+ = x$ and $v_- = y$. This comodule is self dual. The left dual $V^*$ of $V$ has basis $v_+^*$ and $v_-^*$, and one has an isomorphism of $\Oq\text{-comodules}$
\begin{gather*}
\varphi \colon \
\begin{cases}
V \to V^*,
\\
v_+ \mapsto -q^{\frac{5}{2}} v_-^*,
\\[.5ex]
v_- \mapsto q^{\frac{1}{2}}v_+^*.
\end{cases}
\end{gather*}
\end{Definition}

{\sloppy\begin{Proposition}[{\cite[Section~4.2.1]{Klimyk}}]\label{propOqSemiSimple}
The $\Oq$-comodules $V_n,\ n \in \mathbb{N}$, are all the simple $\Oq$-comodules. Moreover, the categories $\Oq\text{-}{\rm comod}$ and $\Oq\text{-}{\rm comod}^{\rm fin}$ are semi-simple.
\end{Proposition}}

\begin{Definition}
The Hopf algebra $\Uq$ is the free $k$-algebra with:
\begin{alignat*}{3}
& \text{generators}\colon\quad  &&E,\,F,\,K,&
\\
&\text{relations}\colon && KE=q^4EK, \quad KF=q^{-4}FK,\quad  EF-FE = \frac{K-K^{-1}}{q^2-q^{-2}},&
\\
&\text{coproduct}\colon &&\Delta (K) = K \otimes K,\quad \Delta (E) = 1\otimes E+E\otimes K,\quad \Delta (F) = K^{-1}\otimes F + F\otimes 1,&
\\
& \text{counit}\colon && \varepsilon(K)=1,\quad\varepsilon(E)=\varepsilon(F)=0,&
\\
& \text{antipode}\colon && S(K) = K^{-1},\quad S(E) = -EK^{-1},\quad S(F) =-KF.&
\end{alignat*}
\end{Definition}

\begin{Definition}
The left $\Uq$-module $V_{\pm,n}$ is the vector space of dimension $n+1$ with action:
\begin{gather*}
 K=\pm \setlength{\arraycolsep}{1.5pt}\begin{pmatrix}
 q^{2n} & 0& \cdots & 0 \\ 0& q^{2(n-2)} &\ddots &\vdots \\ \vdots &\ddots &\ddots&0 \\ 0 &\cdots&0& q^{-2n}\end{pmatrix}\!,\qquad
 E=\setlength{\arraycolsep}{2.5pt}\begin{pmatrix}
 0 & [n]_{q^2}& & 0 \\ & \ddots &\ddots & \\ \vdots & &\ddots&[1]_{q^2} \\ 0 & \cdots & & 0\end{pmatrix}\!,\\
 F=\setlength{\arraycolsep}{2.5pt}\pm \begin{pmatrix}
 0 & & \cdots & 0 \\ {[1]_{q^2}} & \ddots & & \vdots\\ & \ddots&\ddots& \\ 0 & & [n]_{q^2}& 0\end{pmatrix}\!,
 \end{gather*}
 where $ [n]_{q^2} = \frac{q^{2n}-q^{-2n}}{q^2-q^{-2}}$.

The modules of the form $V_{+,n}$ are called of type 1, and the others are discarded here. The mo\-du\-le~$V_{+,1}$ is denoted by $V$ and called the standard representation of $\Uq$.
A module $W$ is called locally finite (of type 1) if $\forall w \in W$, $\Uq \cdot w$ is finite-dimensional (and isomorphic to some $V_{+,n}$).
\end{Definition}

\begin{Proposition}[{\cite[Theorem~VII.2.2]{Kassel}}]
The $\Uq$-modules $V_{\pm,n},\ n \in \mathbb{N}$ are all the locally finite simple $\Uq$-modules. Moreover, the categories $\Uq\text{-}{\rm mod}^{\rm lf}$ and $\Uq\text{-}{\rm mod}^{\rm fin}$, of respectively locally finite and finite-dimensional $\Uq$-modules of type $1$, are semi-simple.
\end{Proposition}

\begin{Definition}
We can define a dual pairing, namely a non-degenerate bilinear form
\[
\langle \cdot,\cdot\rangle \colon \  \Uq \otimes \Oq \to k
\] satisfying $\langle x, y.y'\rangle = \langle x_{(1)},y\rangle \langle x_{(2)},y'\rangle$, $\langle x.x', y\rangle = \langle x,y_{(1)}\rangle \langle x',y_{(2)}\rangle$, $\langle x,1\rangle = \varepsilon(x)$, $\langle 1,y\rangle = \varepsilon(y)$ and $\langle S(x),y\rangle = \langle x,S(y)\rangle$. It is given on generators by
\begin{gather*}
 \left\langle K,\begin{pmatrix}a&b\\c&d \end{pmatrix} \right\rangle =
 \begin{pmatrix}q^2&0\\0&q^{-2} \end{pmatrix}\!, \qquad \!
 \left\langle E,
 \begin{pmatrix}a&b\\c&d \end{pmatrix} \right\rangle =
 \begin{pmatrix}0&1\\0&0 \end{pmatrix}\!,\qquad \!
 \left\langle F,\begin{pmatrix} a&b\\c&d \end{pmatrix}\right\rangle =
 \begin{pmatrix}0&0\\1&0 \end{pmatrix}\!.
 \end{gather*}
 A right $\Oq$-comodule structure on some vector space $W$ induces a left $\Uq$-module structure by $x\cdot w = w_{(1)}. \langle x,w_{(2)} \rangle$, $x \in \Uq$, $w\in W$.
\end{Definition}

\begin{Proposition}[{\cite[equation~(3.3), p.~126]{Abe}, \cite[Theorem~7.9]{Takeuchi}}]\label{propOqcomodUqmod}
This correspondence induces equivalences of categories $\Oq\text{-}{\rm comod}^{\rm fin}\simeq \Uq\text{-}{\rm mod}^{\rm fin}$ and $\Oq\text{-}{\rm comod} \simeq \Uq\text{-}{\rm mod}^{\rm lf}$.

The simple comodules $V_n$ are mapped on the simple modules $V_{+,n}$.
\end{Proposition}

Note that this correspondence also preserves the monoidal (and actually ribbon, though using an unusual twist on $\Uqmod$) structure. In the following, we will mostly adopt the point of view of comodules over $\Oq$, because there are fewer conditions to ask, the braided and ribbon structure of $\Oq$ is algebraic and not ``topological'', and it has a geometric description using stated skein algebras.

\subsection{Skein categories}\label{SectSkCat}
Tangle invariants obtained from a $k$-linear ribbon category $\mathcal{V}$ are defined in \cite[Section~I.2]{Turaev}. They can be extended to give tangle invariants on any oriented surface $\Sigma$ by local relations, and produce the skein category of the surface ${\rm Sk}_\mathcal{V}(\Sigma)$, see \cite{Walker}. We will use and recall their boundary structure and excision properties, see \cite{Cooke}.

\begin{Definition}
Let $[n]$ denote the set of $n$ points $\{1,\ldots,n\} \times \{0\}$ in $\mathbb{R}^2$, equipped with blackboard framing (coming out of the page when we draw ribbon graphs as below) and ${\rm Ribbon}_\mathcal{V}$ the category whose objects are framed oriented $\mathcal{V}$-coloured points of $\mathbb{R}^2$ of the form $[n]$, and morphisms ribbon graphs (i.e., $\mathcal{V}$-coloured oriented framed tangles with coupons coloured by morphisms) between them in $\mathbb{R}^2\times[0,1]$.
\end{Definition}
$$
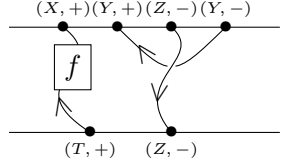
\begin{tikzpicture}[baseline = 15pt, yscale = 0.7, xscale = 1.2]
\draw (0,2) -- (3,2) node[pos = 0.2]{\small $\bullet$} node[pos = 0.2, above]{\tiny $(X,+)$} node[pos = 0.4]{\small $\bullet$} node[pos = 0.4, above]{\tiny $(Y,+)$} node[pos = 0.6]{\small $\bullet$} node[pos = 0.6, above]{\tiny $(Z,-)$} node[pos = 0.8]{\small $\bullet$} node[pos = 0.8, above]{\tiny $(Y,-)$};
\draw (0,0) -- (3,0) node[pos = 0.3]{\small $\bullet$} node[pos = 0.3, below]{\tiny $(T,+)$} node[pos = 0.6]{\small $\bullet$} node[pos = 0.6, below]{\tiny $(Z,-)$};
\draw (0.9,0) .. controls (0,1) and (1.2,1.2) .. (0.6,2) node[pos=0.2, sloped]{$<$} node [pos = 0.6, rectangle, draw, fill=white]{$f$};
\draw (1.2,2) .. controls (1.8,1) .. (2.4,2) node[fill, color= white, midway, circle, scale = 0.3]{} node[pos = 0.2, sloped]{$<$};
\draw (1.8,0) .. controls (1.3,1) and (2.2,1.4) .. (1.8,2) node[pos = 0.3, sloped]{$<$};
\end{tikzpicture}
$$

\begin{Theorem}[Reshetikhin--Turaev]\label{thmturaev} Given $\mathcal{V}$ a strict ribbon category, there exists a unique strictly monoidal functor ${\rm RT}\colon {\rm Ribbon}_\mathcal{V} \to \mathcal{V}$, called the Reshetikhin--Turaev functor, mapping positive single points to their colour, preserving ribbon structures and mapping coupons to their colour.
\end{Theorem}

\begin{Remark}\label{rmkEqCatSkR2VV}
For skein categories, we would prefer to consider all points with all possible framings instead of only those of the form $[n]$ as in ${\rm Ribbon}_\mathcal{V}$. We denote this bigger category by ${\rm Ribbon}_\mathcal{V}\big(\mathbb{R}^2\big)$. The two categories are equivalent as the inclusion of the first in the second is obviously fully faithful and essentially surjective. A quasi-inverse is however not so natural to define. One has to choose an isomorphism from each object of ${\rm Ribbon}_\mathcal{V}\big(\mathbb{R}^2\big)$ to one of the form~$[n]$. This can be done for example by giving lexicographical order on $\mathbb{R}^2$ and pushing all points in good position while preserving this order, then turning the framing clockwise until it is vertical.

Working with ${\rm Ribbon}_\mathcal{V}\big(\mathbb{R}^2\big)$, and allowing non-strict ribbon categories $\mathcal{V}$, the above theorem still holds, but the functor ${\rm RT}$ is only essentially unique. We now assume a choice of quasi-inverse as above has been made, so that the functor ${\rm RT}$ is defined on ${\rm Ribbon}_\mathcal{V}\big(\mathbb{R}^2\big)$.
\end{Remark}

The above functor ${\rm RT}$ describes a way to obtain invariants of coloured tangles, and actually coloured ribbon graphs, from a ribbon category $\mathcal{V}$. We want to study this invariant, and say that two ribbon graphs are identified if they give the same invariant, namely the same morphism in~$\mathcal{V}$ after evaluation under~${\rm RT}$. Skein categories generalize this construction for coloured ribbon graphs on a thickened surface. The idea is to take the relations between ribbon graph which are true locally, on an embedded cube.

\begin{Definition}
Let $\Sigma$ be a compact oriented surface possibly with boundary and $\mathcal{V}$ a ribbon category. The $k$-linear category ${\rm Ribbon}_\mathcal{V}(\Sigma)$ has objects finite sets of $\mathcal{V}$-coloured oriented framed points of $\Sigma$ and morphisms $k$-linear combinations of isotopy classes of ribbon graphs.

The skein category ${\rm Sk}_\mathcal{V}(\Sigma)$ with coefficients in $\mathcal{V}$ is the quotient of ${\rm Ribbon}_\mathcal{V}(\Sigma)$ by the local relation $\sum \lambda_i F_i = 0$ if there exists an orientation preserving embedding of a cube $\phi\colon [0,1]^3 \hookrightarrow \Sigma \times [0,1]$ such that all of the $F_i$'s coincide outside this little cube, intersect $\phi\big(\partial [0,1]^3\big)$ on either the top or the bottom face, transversally, and give the zero morphism in $\mathcal{V}$ after evaluation of the functor ${\rm RT}$ on this little cube, namely $\sum \lambda_i \mathop{\rm RT}\big(\phi^{-1}(F_i\vert_{{\rm im}\ \phi})\big) = 0$.

Put differently, let $F$ be a ribbon graph and $\phi$ a little cube of $\Sigma \times (0,1)$ such that $G := F\vert_{{\rm im}\ \phi}$ is a (possibly complicated) ribbon graph. Then $\mathop{\rm RT}(G)$ is a morphism in $\mathcal{V}$, and we allow ourselves to replace $G$ with a single coupon coloured by this morphism.
\end{Definition}

\begin{Remark}\label{rmkInclVtoSkMonoid}
When $\Sigma = \mathbb{R}^2$, since ${\rm RT}$ is full, via coupons, it induces an equivalence of categories ${\rm Sk}_\mathcal{V}\big(\mathbb{R}^2\big)\to \mathcal{V}$. Its quasi-inverse is given by the inclusion of the full subcategory $\mathcal{V} \subseteq {\rm Sk}_\mathcal{V}\big(\mathbb{R}^2\big)$ mapping an object $V \in \mathcal{V}$ to the framed point [1] coloured by $V$, and a morphism to a coupon coloured by this morphism.

Note that this inclusion is monoidal but not strictly monoidal, namely one has an isomorphism from the two framed points [2] respectively coloured by $V$ and $W$ to the framed point [1] coloured by $V \otimes W$, this isomorphism is the coupon ${\rm Id}_{V \otimes W}$ with two incoming strands coloured respectively by $V$ and $W$ and one outgoing strand coloured by $V \otimes W$.
$$
\begin{tikzpicture}[yscale = 0.7, baseline = 0pt]
\draw (0,0) -- (1,0) node[pos = 0.25,below]{\small $V$} node[pos = 0.33]{\small $\bullet$} node[pos = 0.67]{\small $\bullet$} node[pos = 0.75,below]{\small $W$};
\draw (0,2) -- (1,2) node[pos = 0.5]{\small $\bullet$} node[midway,above]{\small $V \otimes W$};
\draw (0.33,0) -- ++(0,1) node[pos = 0.4, sloped]{\footnotesize $>$};
\draw (0.67,0) -- ++(0,1) node[pos = 0.4, sloped]{\footnotesize $>$};
\draw (0.5,1) -- ++(0,1) node[pos = 0.7, sloped]{\footnotesize $>$};
\node[draw, rectangle, fill=white] at (0.5,1) {\small ${\rm Id}_{V \otimes W}$};
\end{tikzpicture}
$$
\end{Remark}

\begin{Remark}[{\cite[Remark 1.7]{Cooke}}]\label{Skmonoidal}
For a general surface $\Sigma$, the categories defined above are not monoidal because there is no notion of horizontal juxtaposition, which we used in $\mathbb{R}^2$. However, if $A = C \times [0,1]$ for a 1-manifold $C$, the category ${\rm Sk}_\mathcal{V}(A)$ is monoidal with tensor product induced by $A \sqcup A \overset{[0,\frac{1}{3}]\sqcup[\frac{2}{3},1]}\hookrightarrow A$.
\end{Remark}

\begin{Remark}[{\cite[Remarks 1.6 and 1.18]{Cooke}}]\label{rmkFunctEmbIsotOfSkCat} An orientation-preserving smooth embedding $f\colon\Sigma_1 \allowbreak\to \Sigma_2$ induces a functor ${\rm Sk}_\mathcal{V}(f)\colon{\rm Sk}_\mathcal{V}(\Sigma_1)\to {\rm Sk}_\mathcal{V}(\Sigma_2)$. It maps an object $s$, which is a bunch of coloured points in $\Sigma_1$, to $f(s)$, and a ribbon graph $T \subseteq \Sigma_1 \times [0,1]$ to $(f\times {\rm Id})(T)$. This defines a symmetric monoidal functor ${\rm Sk}_\mathcal{V}\colon ({\rm Mfld}_2^{\rm or},\sqcup) \to ({\rm Cat}_k,\otimes_k)$.

An isotopy of smooth embeddings $\lambda \colon \Sigma_1 \times [0,1] \to \Sigma_2$ between $f=\lambda_0$ and $g=\lambda_1$ induces a~natural isomorphism ${\rm rib}_\lambda\colon {\rm Sk}_\mathcal{V}(f) \Rightarrow {\rm Sk}_\mathcal{V}(g)$, where ${\rm rib}_{\lambda,s}\colon f(s) \to g(s)$ is the braid in $\Sigma_2 \times [0,1]$ drawn by $\{(\lambda_t(s),t), t \in [0,1]\}$. Homotopic isotopies give isotopic ribbon graphs, and this extends to a symmetric monoidal $\infty$-functor ${\rm Sk}_\mathcal{V}\colon ({\rm Mfld}_2^{\rm or},\sqcup) \to ({\rm Cat}_k,\otimes_k)$. It is shown in~\cite{Cooke} that this functor is the factorisation homology with coefficients in $\mathcal{V}$.
\end{Remark}

\begin{Proposition}[{\cite[Section~1.3]{Cooke}}]\label{Skmodcat}
For $C\subset \partial \Sigma$ a boundary component $($actually, a thick left boundary component, i.e., equipped with an embedding $C \times [0,2) \hookrightarrow \Sigma)$ the category ${\rm Sk}_\mathcal{V}(\Sigma)$ inherits a~structure of left ${\rm Sk}_\mathcal{V}(A)$-module category, where $A = C\times [0,1]$. Namely, one has a~functor $\rhd\colon {\rm Sk}_\mathcal{V}(A) \otimes {\rm Sk}_\mathcal{V}(\Sigma) \to {\rm Sk}_\mathcal{V}(\Sigma)$ compatible with the monoidal structure of ${\rm Sk}_\mathcal{V}(A)$. It is given by pushing points or ribbon graphs in $A$ inside $\Sigma$.
Similarly, a thick right boundary component $C\times (-1,1]\hookrightarrow \Sigma$ induces a right ${\rm Sk}_\mathcal{V}(A)$-module structure on ${\rm Sk}_\mathcal{V}(\Sigma)$.
\end{Proposition}
Skein categories satisfy a form of excision, namely the skein category of a gluing is obtained as a~relative tensor product of the skein categories of the initial surfaces.

\begin{Theorem}[see {\cite[Theorem~1.22]{Cooke}} for details]\label{thExcSkCat}
Let $\Sigma_1$ and $\Sigma_2$ be two surfaces and $C$ a thick right boundary component of $\Sigma_1$ and a thick left boundary component of $\Sigma_2$. Let $\Sigma_1\cup_{A}\Sigma_2$ be the collar gluing of the surfaces along the two thick embeddings. The skein category ${\rm Sk}_\mathcal{V}(\Sigma_1\cup_{A} \Sigma_2)$ is the Tambara relative tensor product of the right ${\rm Sk}_\mathcal{V}({A})$-module ${\rm Sk}_\mathcal{V}(\Sigma_1)$ and the left ${\rm Sk}_\mathcal{V}({A})$-module ${\rm Sk}_\mathcal{V}(\Sigma_2)$.

Namely, the canonical functor ${\rm Sk}_\mathcal{V}(\Sigma_1) \otimes {\rm Sk}_\mathcal{V}(\Sigma_2) \to {\rm Sk}_\mathcal{V}(\Sigma_1\cup_{A} \Sigma_2)$ induces an equivalence of categories between $k$-linear functors out of ${\rm Sk}_\mathcal{V}(\Sigma_1\cup_{A} \Sigma_2)$ and ${\rm Sk}_\mathcal{V}({A})$-balanced functors out of ${\rm Sk}_\mathcal{V}(\Sigma_1) \otimes {\rm Sk}_\mathcal{V}(\Sigma_2)$ $($i.e., equipped with a natural isomorphism $\iota\colon (-\lhd-,-)\tilde{\Rightarrow}(-,-\rhd-)$ between the actions on~$\Sigma_1$ and on $\Sigma_2)$.
\end{Theorem}

\begin{proof}[Idea of proof]
One has an explicit description of the Tambara relative tensor product of two module categories, by formally adding a natural isomorphism $\iota$ to the morphisms of the two categories.

$$
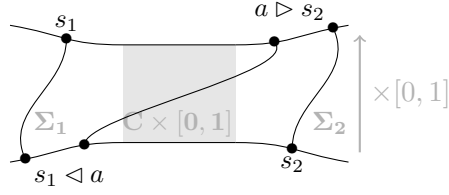
\begin{tikzpicture}[baseline = 15pt, xscale = 1.5, yscale = 1.3]
\fill[gray!20] (4.5,0.5) -- (3.5,0.5)-- (3.5,1.5) -- (4.5,1.5)--cycle;
\node[gray!60] at (4.1,0.7){\small $\mathbf{\Sigma_1 \quad\quad C \times [0,1]\quad \quad\quad \Sigma_2}$};
\draw[gray!60, ->, thick] (5.6,0.4)-- ++(0,1.2) node[midway, right]{$\times [0,1]$};
\draw (4.5,0.5) -- (3.5,0.5)	.. controls (3,0.5 ) and (3,0.4 ) .. (2.5,0.3) coordinate[pos = 0.9] (EEM) coordinate[pos = 0.3] (EMA) ;
\draw (4.5,0.5)	.. controls (5,0.5 ) and (5,0.4 ) .. (5.5,0.3)coordinate[pos = 0.5] (EN);
\draw (4.5,1.5) -- (3.5,1.5)	.. controls (3,1.5 ) and (3,1.6 ) .. (2.5,1.7) coordinate[pos = 0.5] (EM) ;
\draw (5.5,1.7)	.. controls (5,1.6 ) and (5,1.5 ) .. (4.5,1.5)coordinate[pos = 0.7] (ENA) coordinate[pos = 0.1] (EEN);
\draw (EM) node{\small $\bullet$} node[above = -1pt]{\small $s_1$} .. controls ++(0,-0.5) and ++(-0.2,0.5) .. (EEM) node{\small $\bullet$} node[below right = -1pt]{\small $s_1\lhd a$};
\draw (EN) node{\small $\bullet$} node[below = -1pt]{\small $s_2$} .. controls ++(0,0.5) and ++(0.2,-0.5) .. (EEN) node{\small $\bullet$} node[above left = -1pt]{\small $a \rhd s_2$} node[pos = 0.3,right]{};
\draw (EMA) node{\small $\bullet$} .. controls ++(0,0.3) and ++(0.25,-0.3) .. (ENA) node[pos = 0.25,above]{} node{\small $\bullet$} node[pos = 0.75,below]{};
\end{tikzpicture}
$$
It~is equivalent to ${\rm Sk}_\mathcal{V}(\Sigma_1\cup_{A} \Sigma_2)$ by sending, for $s_1$, $a$, $s_2$ some sets of coloured points respectively in~$\Sigma_1$, $A$, $\Sigma_2$, the balancing isomorphism $\iota_{s_1,a,s_2} \colon (s_1\lhd a,s_2) \to (s_1,a \rhd s_2)$ to the morphism depicted hereby.
\end{proof}

\begin{Corollary}\label{corrMorphSkRecoll}
Let $s_1\in {\rm Sk}_\mathcal{V}(\Sigma_1)$ and $s_2\in {\rm Sk}_\mathcal{V}(\Sigma_2)$. Then any morphism
\[
\alpha\in\Hom_{{\rm Sk}_\mathcal{V}(\Sigma_1\cup_A\Sigma_2)}((s_1,s_2),\varnothing)
\]
can be decomposed in a $($linear combination of$)$ pair$($s$)$
\[
\alpha_1 \in \Hom_{{\rm Sk}_\mathcal{V}(\Sigma_1)}(s_1,\varnothing\lhd a),\qquad
\alpha_2\in\Hom_{{\rm Sk}_\mathcal{V}(\Sigma_2)}(a\rhd s_2,\varnothing)
\]
for some $a \in {\rm Sk}_\mathcal{V}(A)$, with $\alpha = ({\rm Id}_\varnothing , \alpha_2) \circ \iota_{\varnothing,a,s_2} \circ (\alpha_1 , {\rm Id}_{s_2})$.

This decomposition is unique up to balancing, namely if $\alpha_2$ can be written $\beta_2 \circ (\gamma\rhd {\rm Id}_{s_2})$, with $\beta_2\in\Hom_{{\rm Sk}_\mathcal{V}(\Sigma_2)}(b\rhd s_2,\varnothing)$ and $\gamma\in\Hom_{{\rm Sk}_\mathcal{V}(A)}(a,b)$, for some $b\in {\rm Sk}_\mathcal{V}(A)$, then $(\alpha_1,\beta_2\circ (\gamma\rhd {\rm Id}_{s_2})) \sim (({\rm Id}_\varnothing\lhd\gamma) \circ \alpha_1, \beta_2)$.
\end{Corollary}

\begin{proof}
On a drawing one wants to decompose $\alpha$ as
$$
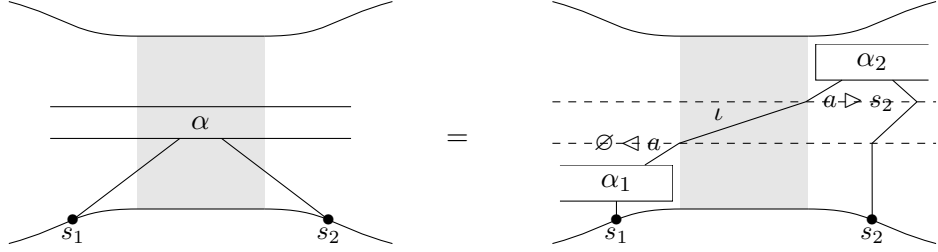
\begin{tikzpicture}[baseline = 2cm, xscale = 1.7, yscale = 2.3]
\fill[gray!20] (4.5,0.5) -- (3.5,0.5)-- (3.5,1.5) -- (4.5,1.5)--cycle;
\draw (4.5,0.5) -- (3.5,0.5)	.. controls (3,0.5 ) and (3,0.4 ) .. (2.5,0.3) coordinate[pos = 0.5] (EM);
\draw (4.5,0.5)	.. controls (5,0.5 ) and (5,0.4 ) .. (5.5,0.3)coordinate[pos = 0.5] (ENbel);
\draw (4.5,1.5) -- (3.5,1.5)	.. controls (3,1.5 ) and (3,1.6 ) .. (2.5,1.7);
\draw (5.5,1.7)	.. controls (5,1.6 ) and (5,1.5 ) .. (4.5,1.5);
\node[minimum width=4cm, topydowny, outer sep = 0pt] (a) at (4,1){$\alpha$};
\draw (EM) node{\small $\bullet$} node[below = -1pt]{\small $s_1$} -- (a);
\draw (ENbel) node{\small $\bullet$} node[below = -1pt]{\small $s_2$} -- (a);
\end{tikzpicture} \hspace{.6cm}=\hspace{1cm}
\begin{tikzpicture}[baseline = 2cm, xscale = 1.7, yscale = 2.3]
\fill[gray!20] (4.5,0.5) -- (3.5,0.5)-- (3.5,1.5) -- (4.5,1.5)--cycle;
\draw (4.5,0.5) -- (3.5,0.5)	.. controls (3,0.5 ) and (3,0.4 ) .. (2.5,0.3) coordinate[pos = 0.5] (EM);
\draw (4.5,0.5)	.. controls (5,0.5 ) and (5,0.4 ) .. (5.5,0.3)coordinate[pos = 0.5] (ENbel);
\draw (4.5,1.5) -- (3.5,1.5)	.. controls (3,1.5 ) and (3,1.6 ) .. (2.5,1.7);
\draw (5.5,1.7)	.. controls (5,1.6 ) and (5,1.5 ) .. (4.5,1.5);
\draw[dashed] (2.5,0.88) -- (5.5,0.88) coordinate[pos = 0.33] (LA) coordinate[pos = 0.833] (ENmidbel);
\draw[dashed] (2.5,1.12) -- (5.5,1.12) coordinate[pos = 0.66] (RA) coordinate[pos = 0.95] (ENmidtop);
\node[minimum width=1.5cm, leftymissy, outer sep = 0pt] (a1) at (3,0.65){$\alpha_1$};
\node[minimum width=1.5cm, rightymissy, outer sep = 0pt] (a2) at (5,1.35){$\alpha_2$};
\draw (EM) node{\small $\bullet$} node[below = -1pt]{\small $s_1$} -- (a1);
\draw (ENbel) node{\small $\bullet$} node[below = -1pt]{\small $s_2$} -- (ENmidbel) -- (ENmidtop) node[left = 5pt]{\small $a\rhd s_2$} -- (a2);
\draw (a1)-- (LA) node[left = 3pt]{\small $\varnothing \lhd a$} -- (RA) node[pos = 0.4, above left = -2pt]{\small $\iota$} -- (a2);
\end{tikzpicture}
$$
which is easily done by pushing the ribbon graph happening in the middle region inside say $N$ leaving only straight lines (namely, $\iota$'s) behind. The relation $(\alpha_1,\beta_2\circ (\gamma\rhd {\rm Id}_{s_2})) \sim (({\rm Id}_\varnothing\lhd\gamma) \circ \alpha_1, \beta_2)$ is true by sliding $\gamma$ along the straight lines of $\iota$, and this is the only relation by the above theorem.
\end{proof}

Note that the asymmetry in this description is purely artificial, and one could have chosen a~cup or a~cap instead of a~slanted line to link the left and right actions.
\begin{Example}
We are particularly interested in the case $\mathcal{V} = \Oq\text{-}{\rm comod}^{\rm fin}$, for which the relations the Reshetikhin--Turaev functor imposes (between framed tangles) are precisely the Kauffman-bracket relations.
\end{Example}

\begin{Remark}\label{rmkUnorTanOq}
In any skein category, the identity coupon ${\rm Id}_{X^*}\colon X^* \to X^*$ with entry a downward oriented $X$-coloured ribbon and output an upward oriented $X^*$-coloured ribbon depicted below gives an identification $(X,-) \simeq (X^*,+)$.
$$
\begin{tikzpicture}[yscale = 0.7, baseline = 0pt]
\draw (0,0) -- (1,0) node[midway,below]{\small $(X,-)$};
\draw (0,2) -- (1,2) node[midway,above]{\small $(X^*,+)$};
\node[draw, rectangle] (I) at (0.5,1) {\small ${\rm Id}_{X^*}$};
\draw[gray!40, line width = 3pt] (0.5,0) -- (I) node[midway, sloped, black]{$<$} -- (0.5,2) node[midway, sloped, black]{$>$};
\end{tikzpicture}
$$
In the case $\mathcal{V} = \Oq\text{-}{\rm comod}^{\rm fin}$ and $X = V$ is the standard corepresentation, the object~$V^*$ is already isomorphic to $V$ in $\mathcal{V}$ by $\varphi$ from Definition~\ref{defOqcomods}. Thus one gets an identification $(V,+) \simeq (V,-)$ and sliding the coupon coloured by it along a strand changes the orientation of the strand. Consequently, one can switch signs of points and orientations of strands.
\end{Remark}

 In the following, we stop mentioning them and talk about unoriented framed tangles. The Reshetikin--Turaev functor is still well-defined on unoriented framed tangles, see \cite[Theorem~4.2]{Tingley}. An unoriented tangle gives a ribbon graph by choosing an arbitrary orientation, colouring every strand by $V$ and replacing $V^*$'s imposed on the boundary points by $V$'s using $\varphi$ or $\varphi^{-1}$. Concretely, for the unoriented cap $\cap$ for example, one can orient it either left or right, so one has to check that $\mathop{\rm RT}(\lcap) \circ (\varphi \otimes {\rm Id}_V) = \mathop{\rm RT}(\rcap) \circ ({\rm Id}_V \otimes \varphi)$. Note that this would not hold with the usual coribbon element in $\Oq$.

{\sloppy\begin{Theorem}[{\cite[Theorem~12.3.10]{CP}} or {\cite[Theorem~3.3.4]{Carter}} for endomorphisms in $\mathbb{R}^2$]\label{thmSkAsTangles}
Let~$n.V$ and~$m.V$ be two objects of ${\rm Sk}_\mathcal{V}(\Sigma)$ given respectively by $n$ and $m$ points coloured by $V$. Then any morphism in $\Hom_{{\rm Sk}_\mathcal{V}(\Sigma)}(n.V,m.V)$ can be represented by a linear combination of unoriented framed tangles $($without coupons$)$. Moreover, two linear combinations of unoriented framed tangles represent the same morphism in $\Hom_{{\rm Sk}_\mathcal{V}(\Sigma)}(n.V,m.V)$ if and only if one can get from one to the other by a~sequence of isotopies and Kauffman-bracket relations.
\end{Theorem}}

\begin{Corollary}
The algebra $\Hom_{{\rm Sk}_\mathcal{V}(\mathbb{R}^2)}(n.V,n.V)$ is isomorphic to the Temperley--Lieb algebra~${\rm TL}_n$, defined in {\rm \cite[Section~XII.3]{Turaev}}. The full subcategory of ${\rm Sk}_\mathcal{V}\big(\mathbb{R}^2\big)$ of objects of the form~$[n]$ with every point coloured by $V$ is equivalent to the category ${\rm TL}$, defined in {\rm\cite[Section~XII.2]{Turaev}}. In the following, we will call this full subcategory ${\rm TL}$.
\end{Corollary}

The category ${\rm TL}$ is a ribbon full subcategory of $\mathcal{V}$, and the above theorem proves that for any surface $\Sigma$ the category ${\rm Sk}_{{\rm TL}}(\Sigma)$ is a full subcategory of ${\rm Sk}_\mathcal{V}(\Sigma)$.

\subsection{Stated skein algebras}\label{SectStatedSkein} Stated skein algebras generalise skein algebras for tangles on a marked surface with boundary, see \cite{BonahonWong,TTQLe,Muller}. These tangles can be cut and stated skein algebras satisfy nice excision properties. We will recall here the approach of \cite{Cos}. Stated skein algebras can be defined over integral rings of coefficients, but we will work over a field $k$ being either $\mathbb{Q}\big(q^{\frac{1}{2}}\big)$ or $\mathbb{C}$ with $q^{\frac{1}{2}} \in \mathbb{C}^\times$ generic, as our proofs only hold in this context.

\begin{Definition}
Let $\mathfrak{S}$ be an oriented surface. The skein algebra $\mathring{\mathscr{S}}(\mathfrak{S})$ is the $k$-vector space generated by isotopy classes of unoriented framed links in $\mathfrak{S}\times (0,1)$ modulo the skein relations:
\[ \Kauftresse \ = q \Kaufvert\ +q^{-1} \Kaufhoriz \qquad \text{and} \qquad \Kaufcirc \ = \big({-}q^2-q^{-2}\big) \Kaufvide \]
in a little embedded cube $\phi\colon \mathbb{D}^3 \hookrightarrow \mathfrak{S}\times (0,1)$.

It is an algebra with product given by vertical superposition $\mathfrak{S}\times (0,1) \sqcup \mathfrak{S}\times (0,1) \overset{\left(\frac{1}{2},1\right)\sqcup\left(0,\frac{1}{2}\right)}\longrightarrow \mathfrak{S} \times (0,1)$ and unit the empty link.
\end{Definition}

Note that by Theorem~\ref{thmSkAsTangles}, one has an algebra isomorphism
\[
\mathring{\mathscr{S}}(\mathfrak{S}) \simeq \End_{{\rm Sk}_{\Oqcomfin}(\Sigma)}(\varnothing).
\]

\begin{Definition}
A marked surface is a compact oriented surface with boundary $\overline{\mathfrak{S}}$ with a finite set $P\subseteq \partial\overline{\mathfrak{S}}$ of boundary points, called marked points. We write $\mathfrak{S} = \overline{\mathfrak{S}} \smallsetminus P$ and call this the marked surface. We write $\partial_P\overline{\mathfrak{S}}$ the boundary components of $\overline{\mathfrak{S}}$ that contain a point of $P$ and $\partial \mathfrak{S} := \partial_P\overline{\mathfrak{S}} \smallsetminus P$. Namely we only consider boundary components of $\mathfrak{S}$ that contain a marked point, which we remove in $\mathfrak{S}$, so all components of $\partial \mathfrak{S}$ are intervals. The circular boundary components are discarded and give punctures in $\mathfrak{S}$. The boundary structure of stated skein algebras will depend on a way to cut out a bigon out of a~boundary edge, and that of internal skein algebras on a way to insert one from a~boundary edge. To avoid choices, we suppose that marked surfaces come equipped with a thickening of their boundary edges inside the surface.

A stated tangle $\alpha$ on $\mathfrak{S}$ is an unoriented, framed, compact, properly embedded 1-submanifold of $\mathfrak{S} \times (0,1)$ whose boundary $\partial \alpha \subseteq \partial\mathfrak{S} \times (0,1)$ has positively vertical framing and comes equipped with a~state ${\rm st}\colon \partial \alpha \to \{+,-\}$. We call height the $(0,1)$-coordinate of a point, and require that all boundary points of $\alpha$ lying over a same component of $\partial\mathfrak{S}$ have distinct heights. An isotopy of stated tangles is an isotopy with values in stated tangles, in particular preserving the height order over a same boundary component. A stated tangle in $\mathfrak{S} \times (0,1)$ can always be represented as a diagram with blackboard framing in $\mathfrak{S}$, with its under/over crossing information, such that the height order of boundary points corresponds to a given orientation on the boundary edges, see \cite[Section~3.5]{BonahonWong}.
\end{Definition}

\begin{Definition}[{\cite[Section~2.5]{Cos}}]
The stated skein algebra $\mathscr{S}(\mathfrak{S})$ of a marked surface $\mathfrak{S}$ is the $k$-vector space generated by isotopy classes of stated tangles on $\mathfrak{S}$ modulo usual skein relations:
\[
\Kauftresse \ = q \Kaufvert\ + q^{-1} \Kaufhoriz\,, \qquad
\Kaufcirc \ = \big({-}q^2-q^{-2}\big)\, \Kaufvide\
\]
and the boundary skein relations:
\[
\begin{tikzpicture}[scale = 0.5, baseline = -3pt]
\draw[dashed] (70:1) arc(70:290:1);
\draw (-70:1) -- (70:1) node[pos = 0.85,sloped]{\small $>$};
\draw[thick] (70:1)++ (0,-0.6) node[right]{\small $+$} arc(90:270:0.8 and 0.4) node[right]{\small $-$};
\end{tikzpicture}
= q^{-\frac{1}{2}} \begin{tikzpicture}[scale = 0.5, baseline = -3pt]
\draw[dashed] (70:1) arc(70:290:1);
\draw (-70:1) -- (70:1) node[pos = 0.85,sloped]{\small $>$};
\end{tikzpicture}\!, \qquad
\begin{tikzpicture}[scale = 0.5, baseline = -3pt]
\draw[dashed] (70:1) arc(70:290:1);
\draw (-70:1) -- (70:1) node[pos = 0.85,sloped]{\small $>$};
\draw[thick] (70:1)++ (0,-0.6) node[right]{\small $+$} arc(90:270:0.8 and 0.4) node[right]{\small $+$};
\end{tikzpicture} = \begin{tikzpicture}[scale = 0.5, baseline = -3pt]
\draw[dashed] (70:1) arc(70:290:1);
\draw (-70:1) -- (70:1) node[pos = 0.85,sloped]{\small $>$};
\draw[thick] (70:1)++ (0,-0.6) node[right]{\small $-$} arc(90:270:0.8 and 0.4) node[right]{\small $-$};
\end{tikzpicture} = 0,
\qquad
\begin{tikzpicture}[scale = 0.5, baseline = -3pt]
\draw[dashed] (70:1) arc(70:290:1);
\draw (-70:1) -- (70:1) node[pos = 0.85,sloped]{\small $>$};
\draw[thick] (70:1)++ (0,-0.6) node[right]{\small $-$} -- ++(-1.3,0);
\draw[thick] (-70:1)++ (0,0.6) node[right]{\small $+$} -- ++(-1.3,0);
\end{tikzpicture}
=q^2 \begin{tikzpicture}[scale = 0.5, baseline = -3pt]
\draw[dashed] (70:1) arc(70:290:1);
\draw (-70:1) -- (70:1) node[pos = 0.85,sloped]{\small $>$};
\draw[thick] (70:1)++ (0,-0.6) node[right]{\small $+$} -- ++(-1.3,0);
\draw[thick] (-70:1)++ (0,0.6) node[right]{\small $-$} -- ++(-1.3,0);
\end{tikzpicture}
+q^{-\frac{1}{2}} \begin{tikzpicture}[scale = 0.5, baseline = -3pt]
\draw[dashed] (70:1) arc(70:290:1);
\draw (-70:1) -- (70:1) node[pos = 0.85,sloped]{\small $>$};
\draw[thick] (70:1)++ (-1.3,-0.6) arc(90:-90:0.8 and 0.4);
\end{tikzpicture}\!,
\]
where the arrows on the boundary edges represent the relative height order of the two points.
It is an algebra with product given by vertical superposition and unit the empty link.

We denote by $C^\mu_\nu$ the coefficient such that
\[
\begin{tikzpicture}[scale = 0.5, baseline = -3pt]
\draw[dashed] (70:1) arc(70:290:1);
\draw (-70:1) -- (70:1) node[pos = 0.85,sloped]{\small $>$};
\draw[thick] (70:1)++ (0,-0.6) node[right]{\small $\mu$} arc(90:270:0.8 and 0.4) node[right]{\small $\nu$};
\end{tikzpicture} = C^\mu_\nu \begin{tikzpicture}[scale = 0.5, baseline = -3pt]
\draw[dashed] (70:1) arc(70:290:1);
\draw (-70:1) -- (70:1) node[pos = 0.85,sloped]{\small $>$};
\end{tikzpicture},
\]
namely $C_+^+ = C_-^- = 0$, $C^+_-=q^{-\frac{1}{2}}$ and one can compute $C^-_+ = -q^{-\frac{5}{2}}$, see \cite[Lemma~2.3(13)]{Cos}. We~also write $C(\nu) := C_\nu^{-\nu}$.
\end{Definition}

\begin{Proposition}[{\cite[Lemmas 2.3 and 2.4]{TTQLe}}] \label{propLeftStSkRel}
These relations express equivalently with the boundary at the left, namely:
\[
\begin{tikzpicture}[scale = 0.5, baseline = -3pt, rotate = 180]
\draw[dashed] (70:1) arc(70:290:1);
\draw (-70:1) -- (70:1) node[pos = 0.15,sloped]{\small $>$};
\draw[thick] (70:1)++ (0,-0.6) node[left]{\small $\nu$} arc(90:270:0.8 and 0.4) node[left]{\small $\mu$};
\end{tikzpicture} = {}^{\mu}_{\nu}\reflectbox{$C$} \begin{tikzpicture}[scale = 0.5, baseline = -3pt, rotate = 180]
\draw[dashed] (70:1) arc(70:290:1);
\draw (-70:1) -- (70:1) node[pos = 0.15,sloped]{\small $>$};
\end{tikzpicture},
\]
where ${}^+_+ \reflectbox{$C$} = {}^-_-\reflectbox{$C$}=0$, ${}^+_-\reflectbox{$C$}= -q^{\frac{5}{2}}$ and ${}^-_+\reflectbox{$C$}= q^{\frac{1}{2}}$, and
\[
\begin{tikzpicture}[scale = 0.5, baseline = -3pt, rotate = 180]
\draw[dashed] (70:1) arc(70:290:1);
\draw (-70:1) -- (70:1) node[pos = 0.15,sloped]{\small $>$};
\draw[thick] (70:1)++ (-1.3,-0.6) arc(90:-90:0.8 and 0.4);
\end{tikzpicture} = q^{-\frac{1}{2}} \begin{tikzpicture}[scale = 0.5, baseline = -3pt, rotate = 180]
\draw[dashed] (70:1) arc(70:290:1);
\draw (-70:1) -- (70:1) node[pos = 0.15,sloped]{\small $>$};
\draw[thick] (70:1)++ (0,-0.6) node[left]{\small $-$} -- ++(-1.3,0);
\draw[thick] (-70:1)++ (0,0.6) node[left]{\small $+$} -- ++(-1.3,0);
\end{tikzpicture} -q^{-\frac{5}{2}} \begin{tikzpicture}[scale = 0.5, baseline = -3pt, rotate = 180]
\draw[dashed] (70:1) arc(70:290:1);
\draw (-70:1) -- (70:1) node[pos = 0.15,sloped]{\small $>$};
\draw[thick] (70:1)++ (0,-0.6) node[left]{\small $+$} -- ++(-1.3,0);
\draw[thick] (-70:1)++ (0,0.6) node[left]{\small $-$} -- ++(-1.3,0);
\end{tikzpicture}.
\]
\end{Proposition}

\begin{Remark}
It is easy to check that $\mathscr{S}(\mathfrak{S} \sqcup \mathfrak{S}') \simeq \mathscr{S}(\mathfrak{S}) \otimes_k \mathscr{S}(\mathfrak{S}')$ since all relations happen in a~connected disk.
\end{Remark}

\begin{Definition} [see {\cite[Section~3.1]{TTQLe}} for details]
Let $\mathfrak{S}$ be a marked surface and $c$ an ideal arc on~$\mathfrak{S}$, joining two marked points in $\overline{\mathfrak{S}}$. Denote by ${\rm Cut}_c(\mathfrak{S})$ the marked surface obtained by cutting $\mathfrak{S}$ along~$c$.

Given a stated tangle $\alpha$ on $\mathfrak{S}$, one can cut it along $c$ and get a tangle ${\rm Cut}_c(\alpha)$ on ${\rm Cut}_c(\mathfrak{S})$. This tangle has new boundary points, two lifts per points of $\alpha \cap c$, and we may give any state to these points. The obtained stated tangle is called a lift of $\alpha$ if the two lifts of a point of $\alpha \cap c$ have same state.
\end{Definition}

This definition is not perfectly innocent and it seems that one could have chosen different state-matching patterns for lifts, see Remark \ref{rmkChoiceCutting}.

\begin{Theorem}[{\cite[Theorem~3.1]{TTQLe}}]\label{splitmorph} Let $\mathfrak{S}$ be a marked surface and $c$ an ideal arc. The map $\rho_c\colon \mathscr{S}(\mathfrak{S}) \to \mathscr{S}({\rm Cut}_c(\mathfrak{S}))$, $\alpha \mapsto \sum_{\rm lifts} \tilde{\alpha}$ is well-defined $($it only depends on the isotopy class of~$\alpha)$ and is an injective algebra morphism. Moreover, the splitting morphisms $\rho_c$ and $\rho_{c'}$ associated to two disjoint ideal arcs $c$ and $c'$ commute.
\end{Theorem}

\begin{Example}
The bigon $B$ is the marked surface $(\mathbb{D},\{\pm i\})$, the disk with two marked points. The algebra $\mathscr{S}(B)$ is generated as an algebra by the
\[
{}_\mu\beta_\nu = \begin{tikzpicture}[scale = 0.8, baseline = -5pt]
\fill [gray!10] (0,0) circle (0.5);
\draw (0,0) circle (0.5);
\node at (0,-0.5){\small $\bullet$};
\node at (0,0.5){\small $\bullet$};
\draw [thick] (-0.5,0) node[left]{$\mu$} -- (0.5,0) node[right]{$\nu$};
\end{tikzpicture}\!\!, \qquad
\mu,\nu \in \{\pm\},
\]
and has basis the
\[_{\vec{\mu}}\beta_{\vec{\nu}} = \begin{tikzpicture}[scale = 0.8, baseline = -5pt]
\fill [gray!10] (0,0) circle (0.5);
\draw (0,0) circle (0.5);
\node[rotate = 45] at (135:0.5) {\small $>$};
\node[rotate = -45] at (45:0.5) {\small $<$};
\node at (0,-0.5){\small $\bullet$};
\node at (0,0.5){\small $\bullet$};
\draw [thick] (-0.5,0) node[rotate = 90, above]{\footnotesize $\cdots$} -- (0.5,0) node[rotate = 90, below]{\footnotesize $\cdots$};
\draw [thick] (-0.45,-0.2) node[below left = -2pt]{\small $\mu_n$} -- (0.45,-0.2) node[below right = -2pt]{\small $\nu_n$};
\draw [thick] (-0.45,0.2) node[above left = -2pt]{\small $\mu_1$} -- (0.45,0.2) node[above right = -2pt]{\small $\nu_1$};
\end{tikzpicture}\!,
\]
where $\vec{\mu} = (\mu_1,\ldots,\mu_n)$ and $\vec{\nu} = (\nu_1,\ldots,\nu_n)$ are decreasing sequences of signs.

It is a bialgebra with coproduct given by cutting along the ``unique'' arc joining the two marked points \[
\begin{tikzpicture}[scale = 0.8, baseline = -5pt]
\fill [gray!10] (0,0) circle (0.5);
\draw (0,0) circle (0.5);
\node at (0,-0.5){\small $\bullet$};
\node at (0,0.5){\small $\bullet$};
\draw [dashed] (0,-0.5) -- (0,0.5) node[midway, right = -2pt]{$c$};
\end{tikzpicture}, \qquad
\Delta=\rho_c\colon\ \mathscr{S}(B) \to \mathscr{S}(B \sqcup B) \simeq \mathscr{S}(B) \otimes \mathscr{S}(B).
\]
Coassociativity comes from the second part of Theorem~\ref{splitmorph}. The counit $\varepsilon\colon \mathscr{S}(B) \to k$ is defined on the basis by $\varepsilon(_{\vec{\mu}}\beta_{\vec{\nu}}) = \delta_{\vec{\mu},\vec{\nu}}$.
It is a Hopf algebra with antipode
\[
S\Bigg(\begin{tikzpicture}[scale = 1, baseline = -5pt]
\fill [gray!10] (0,0) circle (0.5);
\draw (0,0) circle (0.5);
\node[rotate = 45] at (135:0.5) {\small $>$};
\node[rotate = -45] at (45:0.5) {\small $<$};
\node at (0,-0.5){\small $\bullet$};
\node at (0,0.5){\small $\bullet$};
\draw [thick] (-0.5,0) node[rotate = 90, above]{\footnotesize $\cdots$} -- (0.5,0) node[rotate = 90, below]{\footnotesize $\cdots$};
\draw [thick] (-0.45,-0.2) node[below left = -2pt]{\small $\mu_m$} -- (0.45,-0.2) node[below right = -2pt]{\small $\nu_n$};
\draw [thick] (-0.45,0.2) node[above left = -2pt]{\small $\mu_1$} -- (0.45,0.2) node[above right = -2pt]{\small $\nu_1$};
\node[circle, fill = white, draw = black, inner sep = 1.5pt] at (0,0) {$\beta$};
\end{tikzpicture}\!\!\Bigg)=
\begin{tikzpicture}[scale = 1, baseline = -5pt]
\fill [gray!10] (0,0) circle (0.5);
\draw (0,0) circle (0.5);
\node[rotate = 45] at (135:0.5) {\small $>$};
\node[rotate = -45] at (45:0.5) {\small $<$};
\node at (0,-0.5){\small $\bullet$};
\node at (0,0.5){\small $\bullet$};
\draw [thick] (-0.5,0) node[rotate = 90, above]{\footnotesize $\cdots$} -- (0.5,0) node[rotate = 90, below]{\footnotesize $\cdots$};
\draw [thick] (-0.45,-0.2) node[below left = -2pt]{\small $-\nu_1$} -- (0.45,-0.2) node[below right = -2pt]{\small $-\mu_1$};
\draw [thick] (-0.45,0.2) node[above left = -2pt]{\small $-\nu_n$} -- (0.45,0.2) node[above right = -2pt]{\small $-\mu_m$};
\node[circle, rotate = 180, fill = white, draw = black, inner sep = 1.5pt] at (0,0) {$\beta$};
\end{tikzpicture}
 .\ \frac{C(\vec{\nu})}{C(\vec{\mu})},\qquad
\text{where}\qquad C(\vec{\nu}):= \prod_{i=1}^n C(\nu_i).
 \]
It is coquasitriangular with co-$R$-matrix
\[
R(\alpha \otimes \beta) =\varepsilon
\raisebox{0pt}{$\left(\rule{-3pt}{22pt}\right.$}\!
\begin{tikzpicture}[scale = 0.7, baseline = 0pt, rotate = 180]
\draw (0,-1) node{\small $\bullet$} arc(-90:90:1) node[near start, sloped]{\small $>$} node{\small $\bullet$};
\draw (0,-1) arc(-90:-270:1) node[near start, sloped]{\small $<$};
\draw[thick] (-0.9,-0.45) -- (-0.4,-0.45) .. controls (0.3,-0.45) and (0.3,0.3).. (0.9,0.3);
\draw[thick] (-0.94,-0.3) -- (-0.4,-0.3) .. controls (0.3,-0.3) and (0.3,0.45).. (0.94,0.45);
\draw[line width = 3pt, white] (-0.94,0.3) -- (-0.4,0.3) .. controls (0.3,0.3) and (0.3,-0.45).. (0.94,-0.45);
\draw[thick] (-0.94,0.3) -- (-0.4,0.3) .. controls (0.3,0.3) and (0.3,-0.45).. (0.94,-0.45);
\draw[line width = 3pt, white] (-0.9,0.45) -- (-0.4,0.45) .. controls (0.3,0.45) and (0.3,-0.3).. (0.9,-0.3);
\draw[thick] (-0.9,0.45) -- (-0.4,0.45) .. controls (0.3,0.45) and (0.3,-0.3).. (0.9,-0.3);
\node[circle, fill=white, draw = black, scale = 0.8, inner sep = 1.5pt] (X) at (-0.5,0.37){$\beta$};
\node[circle, fill=white, draw = black, scale = 1, inner sep = 1.5pt] (X) at (-0.5,-0.37){$\alpha$};
\end{tikzpicture}\!
\raisebox{0pt}{$\left.\rule{-3pt}{22pt}\right)$}\!,
\]
see \cite[Theorem~3.5]{Cos}.
It is coribbon with coribbon functional
\[
\theta(\alpha) = \varepsilon
\raisebox{0pt}{$\left(\rule{-3pt}{22pt}\right.$}\!
\begin{tikzpicture}[scale = 0.7, baseline = 0pt]
\draw (0,-1) node{\small $\bullet$} arc(-90:90:1) node[near end, sloped]{\small $<$} node{\small $\bullet$};
\draw (0,-1) arc(-90:-270:1) node[near end, sloped]{\small $>$};
\node[circle, draw = black, scale = 1, inner sep = 1.5pt] (X) at (-0.3,-0.3){$\alpha$};
\draw[thick, double] (-0.94,-0.3) -- (X) -- (0.2,-0.3) .. controls (0.5,-0.3) and (0.8,0.4).. (0.5,0.4);
\draw[line width = 4pt, white] (0.5,0.4).. controls (0.2,0.4) and (0.5,-0.3) .. (0.8,-0.3) -- (0.94,-0.3);
\draw[thick, double] (0.5,0.4).. controls (0.2,0.4) and (0.5,-0.3) .. (0.8,-0.3) -- (0.94,-0.3);
\end{tikzpicture}\!
\raisebox{0pt}{$\left.\rule{-3pt}{22pt}\right)$}\!.
\]
\end{Example}

\begin{Proposition}[{\cite[Theorem~4.1]{TTQLe}, \cite[Section~2.2]{KorQue}}, {\cite[Theorem~3.4]{Cos}}]
One has an isomorphism of coribbon Hopf algebras $\mathscr{S}(B)\simeq \Oq$ given on the generators by $_+\beta_+\mapsto a$, ${}_-\beta_{-}\mapsto d$,  ${}_+\beta_{-}\mapsto b$ and ${}_-\beta_{+}\mapsto c$.
\end{Proposition}
Stated skein algebras provide great examples of (possibly infinite-dimensional) $\Oq$ or $\mathscr{S}(B)$-comodules.

Given a marked surface $\mathfrak{S}$ and a boundary edge $e$ of $\mathfrak{S}$, we can consider an ideal arc $c$ going along~$e$ but inside $\mathring{\mathfrak{S}}$. The piece between $e$ and $c$ is a bigon, and the rest is homeomorphic to the original surface. The splitting morphism along $c$, $\Delta = \rho_c\colon \mathscr{S}(\mathfrak{S}) \to \mathscr{S}(\mathfrak{S}\sqcup B)\simeq \mathscr{S}(\mathfrak{S})\otimes \mathscr{S}(B)$ endows $\mathscr{S}(\mathfrak{S})$ with right $\Oq$-comodule structure.
$$
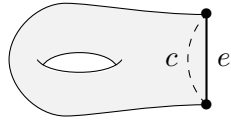
\begin{tikzpicture}[xscale = 0.75, yscale = 0.75,baseline = 10pt]
\fill[gray!10] (3.5,0.2) coordinate (B)	.. controls (2.5,0.2) and (1.5,0 ) .. (1,0)	.. controls (0.5,0 ) and (0,0.5 ) .. (0,1)	.. controls (0,1.5 ) and (0.5,2 ) .. (1,2)	.. controls (1.5,2 ) and (2.5,1.8) .. (3.5,1.8) coordinate (H) -- cycle ;
\fill[white] (0.6,0.9) .. controls (0.7,0.7) and (1.8,0.7) .. (1.9,0.9) .. controls (1.7,1.2) and (0.8,1.2) .. (0.6,0.9);
\draw (0.5,1) .. controls (0.7,0.7) and (1.8,0.7) .. (2,1);
\draw (0.6,0.9) .. controls (0.8,1.2) and (1.7,1.2) .. (1.9,0.9);
\draw (3.5,0.2) node{\small $\bullet$}	.. controls (2.5,0.2) and (1.5,0 ) .. (1,0)	.. controls (0.5,0 ) and (0,0.5 ) .. (0,1)	.. controls (0,1.5 ) and (0.5,2 ) .. (1,2)	.. controls (1.5,2 ) and (2.5,1.8) .. (3.5,1.8) node{\small $\bullet$};
\draw[thick] (B) -- (H) node[midway,right]{$e$};
\path[dashed, bend left = 150pt, in = 135, out = 45] (B) edge node[midway,left]{$c$} (H);
\end{tikzpicture}
$$
It is compatible with its algebra structure, namely $\mathscr{S}(\mathfrak{S})$ is an $\Oq$-comodule-algebra, see \cite[Proposition~4.1]{Cos}.

\begin{Definition}
Let $\operatorname{rot}\colon B \to B$ be the homeomorphism of marked surfaces given by the planar 180 degree rotation. It induces $\operatorname{rot}_*\colon \mathscr{S}(B) \to \mathscr{S}(B)$, with
\[
\operatorname{rot}_*\Bigg(
\begin{tikzpicture}[scale = 1, baseline = -5pt]
\fill [gray!10] (0,0) circle (0.5);
\draw (0,0) circle (0.5);
\node[rotate = 45] at (135:0.5) {\small $>$};
\node[rotate = -45] at (45:0.5) {\small $<$};
\node at (0,-0.5){\small $\bullet$};
\node at (0,0.5){\small $\bullet$};
\draw [thick] (-0.5,0) node[rotate = 90, above]{\footnotesize $\cdots$} -- (0.5,0) node[rotate = 90, below]{\footnotesize $\cdots$};
\draw [thick] (-0.45,-0.2) node[below left = -2pt]{\small $\mu_m$} -- (0.45,-0.2) node[below right = -2pt]{\small $\nu_n$};
\draw [thick] (-0.45,0.2) node[above left = -2pt]{\small $\mu_1$} -- (0.45,0.2) node[above right = -2pt]{\small $\nu_1$};
\node[circle, fill = white, draw = black, inner sep = 1.5pt] at (0,0) {$\beta$};
\end{tikzpicture}
\Bigg) =
\begin{tikzpicture}[scale = 1, baseline = -5pt]
\fill [gray!10] (0,0) circle (0.5);
\draw (0,0) circle (0.5);
\node[rotate = 135] at (-135:0.5) {\small $<$};
\node[rotate = -135] at (-45:0.5) {\small $>$};
\node at (0,-0.5){\small $\bullet$};
\node at (0,0.5){\small $\bullet$};
\draw [thick] (-0.5,0) node[rotate = 90, above]{\footnotesize $\cdots$} -- (0.5,0) node[rotate = 90, below]{\footnotesize $\cdots$};
\draw [thick] (-0.45,-0.2) node[below left = -2pt]{\small $\nu_1$} -- (0.45,-0.2) node[below right = -2pt]{\small $\mu_1$};
\draw [thick] (-0.45,0.2) node[above left = -2pt]{\small $\nu_n$} -- (0.45,0.2) node[above right = -2pt]{\small $\mu_m$};
\node[circle, rotate = 180, fill = white, draw = black, inner sep = 1.5pt] at (0,0) {$\beta$};
\end{tikzpicture}
\]
which is an algebra isomorphism. It reverses the coproduct, namely
\[
\Delta\circ \operatorname{rot}_*(\beta) =
 \begin{tikzpicture}[scale = 1, baseline = -5pt]
\fill [gray!10] (0,0) circle (0.5);
\draw (0,0) circle (0.5);
\node[rotate = 135] at (-135:0.5) {\small $<$};
\node[rotate = -135] at (-45:0.5) {\small $>$};
\node at (0,-0.5){\small $\bullet$};
\node at (0,0.5){\small $\bullet$};
\draw [thick] (-0.5,0) -- (0.5,0);
\draw [thick] (-0.45,-0.2) -- (0.45,-0.2) ;
\draw [thick] (-0.45,0.2) -- (0.45,0.2);
\node[circle, rotate = 180, fill = white, draw = black, inner sep = 1pt] at (0,0) {\footnotesize $\beta_{(2)}$};
\end{tikzpicture} \otimes \begin{tikzpicture}[scale = 1, baseline = -5pt]
\fill [gray!10] (0,0) circle (0.5);
\draw (0,0) circle (0.5);
\node[rotate = 135] at (-135:0.5) {\small $<$};
\node[rotate = -135] at (-45:0.5) {\small $>$};
\node at (0,-0.5){\small $\bullet$};
\node at (0,0.5){\small $\bullet$};
\draw [thick] (-0.5,0) -- (0.5,0);
\draw [thick] (-0.45,-0.2) -- (0.45,-0.2) ;
\draw [thick] (-0.45,0.2) -- (0.45,0.2);
\node[circle, rotate = 180, fill = white, draw = black, inner sep = 1pt] at (0,0) {\footnotesize $\beta_{(1)}$};
\end{tikzpicture} = (\operatorname{rot}_*\otimes \operatorname{rot}_*) \circ \Delta^{\rm op} (\beta),
\]
and preserves the counit, because $\varepsilon(\operatorname{rot}_*(_{\vec{\mu}}\beta_{\vec{\nu}})) = \varepsilon(_{\vec{\nu}}\beta_{\vec{\mu}}) = \delta_{\vec{\nu},\vec{\mu}} = \delta_{\vec{\mu},\vec{\nu}}$.
On $\Oq$, it is given by $r\left(\begin{smallmatrix}a&b\\c&d\end{smallmatrix}\right) = \left( \begin{smallmatrix}a&c\\b&d\end{smallmatrix}\right)$.
\end{Definition}

\begin{Remark}If one sees the edge $e$ at the left instead of the right of the surface, one gets a~structure of left $\Oq$-comodule. One can easily get from one to another by rotating the whole picture. Namely, the left coaction $\Delta_{\rm l}$ is obtained from the right coaction $\Delta_{\rm r}$ by rotating the bigon by 180 degrees. In \cite[Proposition~4.1]{Cos} one gets $\Delta_{\rm l} = \operatorname{fl} \circ ({\rm Id}_{\mathscr{S}(\mathfrak{S})} \otimes \operatorname{rot}_*) \circ \Delta_{\rm r}$, where $\operatorname{fl}$ denotes the flip of tensors.
$$
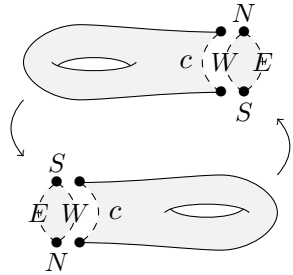
\begin{tikzpicture}[xscale = 0.75, yscale = 0.5,baseline = 10pt]
\fill[gray!10] (3.5,0.2) coordinate (B)	.. controls (2.5,0.2) and (1.5,0 ) .. (1,0)	.. controls (0.5,0 ) and (0,0.5 ) .. (0,1)	.. controls (0,1.5 ) and (0.5,2 ) .. (1,2)	.. controls (1.5,2 ) and (2.5,1.8) .. (3.5,1.8) coordinate (H) to[bend right = 150pt, in = -135, out = -45] (B) ;
\fill[white] (0.6,0.9) .. controls (0.7,0.7) and (1.8,0.7) .. (1.9,0.9) .. controls (1.7,1.2) and (0.8,1.2) .. (0.6,0.9);
\draw (0.5,1) .. controls (0.7,0.7) and (1.8,0.7) .. (2,1);
\draw (0.6,0.9) .. controls (0.8,1.2) and (1.7,1.2) .. (1.9,0.9);
\draw (3.5,0.2)node{\small $\bullet$}	.. controls (2.5,0.2) and (1.5,0 ) .. (1,0)	.. controls (0.5,0 ) and (0,0.5 ) .. (0,1)	.. controls (0,1.5 ) and (0.5,2 ) .. (1,2)	.. controls (1.5,2 ) and (2.5,1.8) .. (3.5,1.8) node{\small $\bullet$};
\path[dashed, bend left = 150pt, in = 135, out = 45] (B) edge node[midway,left]{$c$} (H);
\node [right= 8pt of B, inner sep = 0pt] (B') {};
\node [right= 8pt of H, inner sep = 0pt] (H') {};
\fill[gray!10] (B'.center)to[bend left = 150pt, in = 135, out = 45] (H'.center) to[bend left = 150pt, in = 135, out = 45] (B');
\draw[dashed] (B'.center) to[bend left = 150pt, in = 135, out = 45] node[pos = 0]{\small $\bullet$} node[pos = 1]{\small $\bullet$} node[midway]{\small $W$} (H'.center);
\draw[dashed] (H'.center) to[bend left = 150pt, in = 135, out = 45] node[midway]{\small $E$} node[pos = 1,below]{\small $S$} node[pos = 0,above]{\small $N$} (B'.center);
\draw[->] (0,0) to[bend right] ++(0,-1.5);
\begin{scope}[xshift = 4.5cm, yshift = -2cm, rotate = 180]
\fill[gray!10] (3.5,0.2) coordinate (B)	.. controls (2.5,0.2) and (1.5,0 ) .. (1,0)	.. controls (0.5,0 ) and (0,0.5 ) .. (0,1)	.. controls (0,1.5 ) and (0.5,2 ) .. (1,2)	.. controls (1.5,2 ) and (2.5,1.8) .. (3.5,1.8) coordinate (H) to[bend right = 150pt, in = -135, out = -45] (B) ;
\fill[white] (0.6,0.9) .. controls (0.7,0.7) and (1.8,0.7) .. (1.9,0.9) .. controls (1.7,1.2) and (0.8,1.2) .. (0.6,0.9);
\draw (0.5,1) .. controls (0.7,0.7) and (1.8,0.7) .. (2,1);
\draw (0.6,0.9) .. controls (0.8,1.2) and (1.7,1.2) .. (1.9,0.9);
\draw (3.5,0.2)node{\small $\bullet$}	.. controls (2.5,0.2) and (1.5,0 ) .. (1,0)	.. controls (0.5,0 ) and (0,0.5 ) .. (0,1)	.. controls (0,1.5 ) and (0.5,2 ) .. (1,2)	.. controls (1.5,2 ) and (2.5,1.8) .. (3.5,1.8) node{\small $\bullet$};
\path[dashed, bend left = 150pt, in = 135, out = 45] (B) edge node[midway,right]{$c$} (H);
\node [left= 8pt of B, inner sep = 0pt] (B') {};
\node [left= 8pt of H, inner sep = 0pt] (H') {};
\fill[gray!10] (B'.center)to[bend left = 150pt, in = 135, out = 45] (H'.center) to[bend left = 150pt, in = 135, out = 45] (B');
\draw[dashed] (B'.center) to[bend left = 150pt, in = 135, out = 45] node[pos = 0]{\small $\bullet$} node[pos = 1]{\small $\bullet$} node[midway]{\small $W$} (H'.center);
\draw[dashed] (H'.center) to[bend left = 150pt, in = 135, out = 45] node[midway]{\small $E$} node[pos = 1,above]{\small $S$} node[pos = 0,below]{\small $N$} (B'.center);
\draw[->] (0,0) to[bend right] ++(0,-1.5);
\end{scope}
\end{tikzpicture}
$$
Actually, one gets such a structure for each boundary edge of $\mathfrak{S}$, and if $\mathfrak{S}$ has $n$ right boundary edges and $m$ left, $\mathscr{S}(\mathfrak{S})$ is an $\big(\Oq^{\otimes n},\Oq^{\otimes m}\big)$-bicomodule algebra.
\end{Remark}

\begin{Remark}\label{rmkFormActOnSS}
The structure forms on $\mathscr{S}(B)$, such as the co-$R$-matrix or the coribbon functional, are often defined using the counit on some transformation of the tangle. This has a direct interpretation on how this form then acts on comodules.

Let $\varphi\colon\mathscr{S}(B) \to k$ be given by some
\[
\varphi\raisebox{-3pt}{$\left(\rule{0pt}{22pt}\right.$}\!\!
\begin{tikzpicture}[scale = 0.7, baseline = 0pt]
\draw (0,-1) node{\small $\bullet$} arc(-90:90:1) node[near end, sloped]{\small $<$} node{\small $\bullet$};
\draw (0,-1) arc(-90:-270:1) node[near end, sloped]{\small $>$};
\draw[thick] (-0.94,-0.3)node[below left]{\small $\varepsilon_n$} -- (0.94,-0.3)node[below right]{\small $\eta_m$};
\draw[thick] (-1,0) node[yshift=2pt,left]{\footnotesize $\vdots$} -- (1,0)node[yshift=2pt,right]{\footnotesize $\vdots$};
\draw[thick] (-0.94,0.3) node[above left]{\small $\varepsilon_1$} -- (0.94,0.3)node[above right]{\small $\eta_1$};
\node[circle, fill=white, draw = black, scale = 1.5, inner sep = 1.5pt] (X) at (0,0){$\alpha$};
\end{tikzpicture}\!\!
\raisebox{-3pt}{$\left.\rule{0pt}{22pt}\right)$}
 = \varepsilon
\raisebox{-3pt}{$\left(\rule{0pt}{32pt}\right.$}\!\!\!
\begin{tikzpicture}[scale = 0.7, baseline = 0pt]
\draw (0,-1) node{\small $\bullet$} arc(-90:90:1) node[near end, sloped]{\small $<$} node{\small $\bullet$};
\draw (0,-1) arc(-90:-270:1) node[near end, sloped]{\small $>$};
\draw[thick] (-0.94,-0.3)node[below left]{\small $\varepsilon_n$} -- (0.94,-0.3);
\draw[thick] (-1,0) node[yshift=2pt,left]{\footnotesize $\vdots$} -- (1,0)node[xshift = 5pt, yshift = 0pt, rotate = -90]{\small $s_m(\eta_1 \dots \eta_m)$};
\draw[thick] (-0.94,0.3) node[above left]{\small $\varepsilon_1$} -- (0.94,0.3);
\node[circle, minimum width=0.55cm, fill=white, draw = black, scale = 1, inner sep = 1.5pt] (X) at (-0.45,0){$\alpha$};
\node[circle, minimum width=0.5cm,fill=white, draw = black, scale = 1, inner sep = 1.5pt] at (0.45,0){\small $T_m$};
\end{tikzpicture}\!
\raisebox{-3pt}{$\left.\rule{0pt}{32pt}\right)$},
\]
where $T_m,\ m \in \mathbb{N},$ is a family of tangles with $m$ right and $m$ left boundary points and $s_m\colon \{\pm\}^m\to {\rm vect}_ k(\{\pm\}^m)$. Here, we allowed $s_m$ to have values in formal linear combination of $m$-tuples of states because in the definition of the co-$R$-matrix for example one needs coefficients depending on the states. What we mean by a tangle $\alpha$ with state a formal linear combination of states $\sum_i \lambda_i \vec{\eta}_i$ is the linear sum of stated tangles $\alpha_{\sum_i \lambda_i \vec{\eta}_i} := \sum_i \lambda_i \alpha_{\vec{\eta}_i}$.
Note that the $T_m$'s and the $s_m$'s should satisfy extra conditions for this to be well-defined on $\mathscr{S}(B)$. Then for a marked surface $\mathfrak{S}$ with a right edge $e$, the map $\Phi_{\mathscr{S}(\mathfrak{S})}\colon \mathscr{S}(\mathfrak{S}) \overset{\Delta}{\to} \mathscr{S}(\mathfrak{S}) \otimes \mathscr{S}(B) \overset{{\rm Id} \otimes \varphi}{\longrightarrow} \mathscr{S}(\mathfrak{S})$ is given by
\[
\begin{array}{ll}
\begin{tikzpicture}[scale = 0.7, baseline = 0pt]
\draw (0,-1) node{\small $\bullet$} arc(-90:90:1) node[near end, sloped]{\small $<$} node{\small $\bullet$};
\draw (0,-1) -- (-1,-1);
\draw (0,1) -- (-1,1);
\draw[thick] (-1,-0.3) -- (0.94,-0.3)node[below right]{\small $\eta_m$};
\draw[thick] (-1,0) -- (1,0)node[yshift=2pt,right]{\footnotesize $\vdots$};
\draw[thick] (-1,0.3) -- (0.94,0.3)node[above right]{\small $\eta_1$};
\node[fill=white, minimum width = 1cm, minimum height = 0.7cm, leftymissy, inner sep = 1.5pt] (X) at (-0.5,0){$\alpha$};
\end{tikzpicture}
&\overset{\Delta}{\longmapsto} \sum_{\vec{\nu}}
\raisebox{-3pt}{$\left(\rule{0pt}{26pt}\right.$}
\begin{tikzpicture}[scale = 0.7, baseline = 0pt]
\draw (0,-1) node{\small $\bullet$} arc(-90:90:1) node[near end, sloped]{\small $<$} node{\small $\bullet$};
\draw (0,-1) -- (-1,-1);
\draw (0,1) -- (-1,1);
\draw[thick] (-1,-0.3) -- (0.94,-0.3)node[below right]{\small $\nu_m$};
\draw[thick] (-1,0) -- (1,0)node[yshift=2pt,right]{\footnotesize $\vdots$};
\draw[thick] (-1,0.3) -- (0.94,0.3)node[above right]{\small $\nu_1$};
\node[fill=white, minimum width = 1cm, minimum height = 0.7cm, leftymissy, inner sep = 1.5pt] (X) at (-0.5,0){$\alpha$};
\end{tikzpicture} \otimes
\begin{tikzpicture}[scale = 0.7, baseline = 0pt]
\draw (0,-1) node{\small $\bullet$} arc(-90:90:1) node[near end, sloped]{\small $<$} node{\small $\bullet$};
\draw (0,-1) arc(-90:-270:1) node[near end, sloped]{\small $>$};
\draw[thick] (-0.94,-0.3)node[below left]{\small $\nu_m$} -- (0.94,-0.3)node[below right]{\small $\eta_m$};
\draw[thick] (-1,0) node[yshift=2pt,left]{\footnotesize $\vdots$} -- (1,0)node[yshift=2pt,right]{\footnotesize $\vdots$};
\draw[thick] (-0.94,0.3) node[above left]{\small $\nu_1$} -- (0.94,0.3)node[above right]{\small $\eta_1$};
\end{tikzpicture}\!\!\!
\raisebox{-3pt}{$\left.\rule{0pt}{26pt}\right)$}
\\
& \overset{{\rm Id}\otimes \varphi}{\longmapsto} ({\rm Id} \otimes \varepsilon)
\raisebox{-3pt}{$\left(\rule{0pt}{32pt}\right.$}
\begin{tikzpicture}[scale = 0.7, baseline = 0pt]
\draw (0,-1) node{\small $\bullet$} arc(-90:90:1) node[near end, sloped]{\small $<$} node{\small $\bullet$};
\draw (0,-1) -- (-1,-1);
\draw (0,1) -- (-1,1);
\draw[thick] (-1,-0.3) -- (0.94,-0.3)node[below right]{\small $\nu_m$};
\draw[thick] (-1,0) -- (1,0)node[yshift=2pt,right]{\footnotesize $\vdots$};
\draw[thick] (-1,0.3) -- (0.94,0.3)node[above right]{\small $\nu_1$};
\node[fill=white, minimum width = 1cm, minimum height = 0.7cm, leftymissy, inner sep = 1.5pt] (X) at (-0.5,0){$\alpha$};
\end{tikzpicture} \otimes
\begin{tikzpicture}[scale = 0.7, baseline = 0pt]
\draw (0,-1) node{\small $\bullet$} arc(-90:90:1) node[near end, sloped]{\small $<$} node{\small $\bullet$};
\draw (0,-1) arc(-90:-270:1) node[near end, sloped]{\small $>$};
\draw[thick] (-0.94,-0.3)node[below left]{\small $\nu_m$} -- (0.94,-0.3);
\draw[thick] (-1,0) node[yshift=2pt,left]{\footnotesize $\vdots$} -- (1,0)node[xshift = 5pt, yshift = 0pt, rotate = -90]{\small $s_m(\eta_1 \dots \eta_m)$};
\draw[thick] (-0.94,0.3) node[above left]{\small $\nu_1$} -- (0.94,0.3);
\node[circle, minimum width=0.5cm,fill=white, draw = black, scale = 1, inner sep = 1.5pt] at (0,0){\small $T_m$};
\end{tikzpicture}
\raisebox{-3pt}{$\left.\rule{0pt}{32pt}\right)$}
\\
& \overset{\rm splitting}{\underset{\rm well\ def}{=}} ({\rm Id} \otimes \varepsilon)
\raisebox{-3pt}{$\left(\rule{0pt}{32pt}\right.$}
\begin{tikzpicture}[scale = 0.7, baseline = 0pt]
\draw (0,-1) node{\small $\bullet$} arc(-90:90:1) node[near end, sloped]{\small $<$} node{\small $\bullet$};
\draw (0,-1) -- (-1,-1);
\draw (0,1) -- (-1,1);
\draw[thick] (-1,-0.3) -- (0.94,-0.3)node[below right]{\small $\nu_m$};
\draw[thick] (-1,0) -- (1,0)node[yshift=2pt,right]{\footnotesize $\vdots$};
\draw[thick] (-1,0.3) -- (0.94,0.3)node[above right]{\small $\nu_1$};
\node[fill=white, minimum width = 0.7cm, minimum height = 0.7cm, leftymissy, inner sep = 1.5pt] (X) at (-0.65,0){$\alpha$};
\node[circle, minimum width=0.5cm,fill=white, draw = black, scale = 1, inner sep = 1.5pt] at (0.4,0){\small $T_m$};
\end{tikzpicture} \otimes
\begin{tikzpicture}[scale = 0.7, baseline = 0pt]
\draw (0,-1) node{\small $\bullet$} arc(-90:90:1) node[near end, sloped]{\small $<$} node{\small $\bullet$};
\draw (0,-1) arc(-90:-270:1) node[near end, sloped]{\small $>$};
\draw[thick] (-0.94,-0.3)node[below left]{\small $\nu_m$} -- (0.94,-0.3);
\draw[thick] (-1,0) node[yshift=2pt,left]{\footnotesize $\vdots$} -- (1,0)node[xshift = 5pt, yshift = 0pt, rotate = -90]{\small $s_m(\eta_1 \dots \eta_m)$};
\draw[thick] (-0.94,0.3) node[above left]{\small $\nu_1$} -- (0.94,0.3);
\end{tikzpicture}
\raisebox{-3pt}{$\left.\rule{0pt}{32pt}\right)$}
\overset{\rm counit}{=}
\begin{tikzpicture}[scale = 0.7, baseline = 0pt]
\draw (0,-1) node{\small $\bullet$} arc(-90:90:1) node[near end, sloped]{\small $<$} node{\small $\bullet$};
\draw (0,-1) -- (-1,-1);
\draw (0,1) -- (-1,1);
\draw[thick] (-1,-0.3) -- (0.94,-0.3);
\draw[thick] (-1,0) -- (1,0)node[xshift = 5pt, yshift = 0pt, rotate = -90]{\small $s_m(\eta_1 \dots \eta_m)$};
\draw[thick] (-1,0.3) -- (0.94,0.3);
\node[fill=white, minimum width = 0.7cm, minimum height = 0.7cm, leftymissy, inner sep = 1.5pt] (X) at (-0.65,0){$\alpha$};
\node[circle, minimum width=0.5cm,fill=white, draw = black, scale = 1, inner sep = 1.5pt] at (0.4,0){\small $T_m$};
\end{tikzpicture}.
\end{array}
\]
The co-$R$-matrix isn't exactly of this form, but the same kind of computation applies. Remember that the braiding on $\mathscr{S}(\mathfrak{S}) \otimes \mathscr{S}(\mathfrak{S})$ is defined using the coaction on each $\mathscr{S}(\mathfrak{S})$, the co-$R$-matrix on the $\mathscr{S}(B)^{\otimes 2}$ part thus obtained, and then flipping the two factors, namely $c_{\mathscr{S}(\mathfrak{S}),\mathscr{S}(\mathfrak{S})} = \operatorname{fl} \circ R_{24}\circ \big(\Delta_{\mathscr{S}(\mathfrak{S})}\otimes \Delta_{\mathscr{S}(\mathfrak{S})}\big)$. The braided opposite product on $\mathscr{S}(\mathfrak{S})$ is defined as $m^{\rm bop} := m \circ c_{\mathscr{S}(\mathfrak{S}),\mathscr{S}(\mathfrak{S})}$ and it has a nice geometric depiction, namely:
\[
\begin{array}{ll}
m^{\rm bop}(\alpha \otimes \beta) & = m \circ \operatorname{fl}
\raisebox{-3pt}{$\left(\rule{0pt}{26pt}\right.$}
\begin{tikzpicture}[scale = 0.7, baseline = 0pt]
\draw (0,-1) node{\small $\bullet$} arc(-90:90:1) node[near end, sloped]{\small $<$} node{\small $\bullet$};
\draw (0,-1) -- (-1,-1);
\draw (0,1) -- (-1,1);
\draw[thick] (-1,-0.3) -- (0.94,-0.3)node[below right]{\small $\nu_n$};
\draw[thick] (-1,0) -- (1,0)node[yshift=2pt,right]{\footnotesize $\vdots$};
\draw[thick] (-1,0.3) -- (0.94,0.3)node[above right]{\small $\nu_1$};
\node[fill=white, minimum width = 1cm, minimum height = 0.7cm, leftymissy, inner sep = 1.5pt] (X) at (-0.5,0){$\alpha_{(1)}$};
\end{tikzpicture} \otimes \begin{tikzpicture}[scale = 0.7, baseline = 0pt]
\draw (0,-1) node{\small $\bullet$} arc(-90:90:1) node[near end, sloped]{\small $<$} node{\small $\bullet$};
\draw (0,-1) -- (-1,-1);
\draw (0,1) -- (-1,1);
\draw[thick] (-1,-0.3) -- (0.94,-0.3)node[below right]{\small $\mu_m$};
\draw[thick] (-1,0) -- (1,0)node[yshift=2pt,right]{\footnotesize $\vdots$};
\draw[thick] (-1,0.3) -- (0.94,0.3)node[above right]{\small $\mu_1$};
\node[fill=white, minimum width = 1cm, minimum height = 0.7cm, leftymissy, inner sep = 1.5pt] (X) at (-0.5,0){$\beta_{(1)}$};
\end{tikzpicture}
\!.\
\varepsilon \raisebox{-3pt}{$\left(\rule{0pt}{26pt}\right.$}\!\!
\begin{tikzpicture}[scale = 0.7, baseline = 0pt, rotate = 180]
\draw (0,-1) node{\small $\bullet$} arc(-90:90:1) node[near start, sloped]{\small $>$} node{\small $\bullet$};
\draw (0,-1) arc(-90:-270:1) node[near start, sloped]{\small $<$};
\draw[thick] (-0.9,-0.45) -- (-0.4,-0.45) .. controls (0.3,-0.45) and (0.3,0.3).. (0.9,0.3);
\draw[thick] (-0.94,-0.3) -- (-0.4,-0.3) .. controls (0.3,-0.3) and (0.3,0.45).. (0.94,0.45)node[left]{\small $\vec{\nu}$};
\draw[line width = 3pt, white] (-0.94,0.3) -- (-0.4,0.3) .. controls (0.3,0.3) and (0.3,-0.45).. (0.94,-0.45);
\draw[thick] (-0.94,0.3) -- (-0.4,0.3) .. controls (0.3,0.3) and (0.3,-0.45).. (0.94,-0.45)node[left]{\small $\vec{\mu}$};
\draw[line width = 3pt, white] (-0.9,0.45) -- (-0.4,0.45) .. controls (0.3,0.45) and (0.3,-0.3).. (0.9,-0.3);
\draw[thick] (-0.9,0.45) -- (-0.4,0.45) .. controls (0.3,0.45) and (0.3,-0.3).. (0.9,-0.3);
\node[circle, fill=white, draw = black, scale = 0.6, inner sep = 1.5pt] (X) at (-0.5,0.37){$\beta_{(2)}$};
\node[circle, fill=white, draw = black, scale = 0.65, inner sep = 1.5pt] (X) at (-0.5,-0.37){$\alpha_{(2)}$};
\end{tikzpicture}\!\!
\raisebox{-3pt}{$\left.\rule{0pt}{26pt}\right)$}\!\raisebox{-3pt}{$\left.\rule{0pt}{26pt}\right)$}
\\
&=
 \begin{tikzpicture}[scale = 0.7, baseline = 0pt]
\draw (0,-1) node{\small $\bullet$} arc(-90:90:1) node[near end, sloped]{\small $<$} node{\small $\bullet$};
\draw (0,-1) -- (-1,-1);
\draw (0,1) -- (-1,1);
\draw[thick, double] (-1,-0.3) -- (0.94,-0.3)node[right=-2pt]{\small $\vec{\nu}$};
\draw[thick, double] (-1,0.3) -- (0.94,0.3)node[right=-2pt]{\small $\vec{\mu}$};
\node[scale = 0.8, fill=white, minimum width = 1cm, minimum height = 0.5cm, leftymissy, inner sep = 1.5pt] (X) at (-0.5,0.35){\small $\beta_{(1)}$};
\node[scale = 0.8, fill=white, minimum width = 1cm, minimum height = 0.5cm, leftymissy, inner sep = 1.5pt] (X) at (-0.5,-0.35){\small $\alpha_{(1)}$};
\end{tikzpicture}
.\
\varepsilon \raisebox{-3pt}{$\left(\rule{0pt}{26pt}\right.$}\!\!
\begin{tikzpicture}[scale = 0.7, baseline = 0pt, rotate = 180]
\draw (0,-1) node{\small $\bullet$} arc(-90:90:1) node[near start, sloped]{\small $>$} node{\small $\bullet$};
\draw (0,-1) arc(-90:-270:1) node[near start, sloped]{\small $<$};
\draw[thick] (-0.9,-0.45) -- (-0.4,-0.45) .. controls (0.3,-0.45) and (0.3,0.3).. (0.9,0.3);
\draw[thick] (-0.94,-0.3) -- (-0.4,-0.3) .. controls (0.3,-0.3) and (0.3,0.45).. (0.94,0.45)node[left]{\small $\vec{\nu}$};
\draw[line width = 3pt, white] (-0.94,0.3) -- (-0.4,0.3) .. controls (0.3,0.3) and (0.3,-0.45).. (0.94,-0.45);
\draw[thick] (-0.94,0.3) -- (-0.4,0.3) .. controls (0.3,0.3) and (0.3,-0.45).. (0.94,-0.45)node[left]{\small $\vec{\mu}$};
\draw[line width = 3pt, white] (-0.9,0.45) -- (-0.4,0.45) .. controls (0.3,0.45) and (0.3,-0.3).. (0.9,-0.3);
\draw[thick] (-0.9,0.45) -- (-0.4,0.45) .. controls (0.3,0.45) and (0.3,-0.3).. (0.9,-0.3);
\node[circle, fill=white, draw = black, scale = 0.6, inner sep = 1.5pt] (X) at (-0.5,0.37){$\beta_{(2)}$};
\node[circle, fill=white, draw = black, scale = 0.65, inner sep = 1.5pt] (X) at (-0.5,-0.37){$\alpha_{(2)}$};
\end{tikzpicture}\!\!\raisebox{-3pt}{$\left.\rule{0pt}{26pt}\right)$}
=
\begin{tikzpicture}[scale = 0.7, baseline = 0pt]
\draw (0,-1) node{\small $\bullet$} arc(-90:90:1) node[near end, sloped]{\small $<$} node{\small $\bullet$};
\draw (0,-1) -- (-1,-1);
\draw (0,1) -- (-1,1);
\draw[thick] (-0.9,-0.45) -- (-0.4,-0.45) .. controls (0.3,-0.45) and (0.3,0.3).. (0.9,0.3);
\draw[thick] (-0.94,-0.3) -- (-0.4,-0.3) .. controls (0.3,-0.3) and (0.3,0.45).. (0.94,0.45);
\draw[line width = 3pt, white] (-0.94,0.3) -- (-0.4,0.3) .. controls (0.3,0.3) and (0.3,-0.45).. (0.94,-0.45);
\draw[thick] (-0.94,0.3) -- (-0.4,0.3) .. controls (0.3,0.3) and (0.3,-0.45).. (0.94,-0.45);
\draw[line width = 3pt, white] (-0.9,0.45) -- (-0.4,0.45) .. controls (0.3,0.45) and (0.3,-0.3).. (0.9,-0.3);
\draw[thick] (-0.9,0.45) -- (-0.4,0.45) .. controls (0.3,0.45) and (0.3,-0.3).. (0.9,-0.3);
\node[scale = 0.8, fill=white, minimum width = 0.8cm, minimum height = 0.5cm, leftymissy, inner sep = 1.5pt] (X) at (-0.6,0.35){\small $\beta$};
\node[scale = 0.8, fill=white, minimum width = 0.8cm, minimum height = 0.5cm, leftymissy, inner sep = 1.5pt] (X) at (-0.6,-0.35){\small $\alpha$};
\end{tikzpicture}\!.
\end{array}
\]
\end{Remark}

Stated skein algebras satisfy a form of excision, namely the stated skein algebra of a gluing is obtained as a relative tensor product of the stated skein algebras of the two initial surfaces.
Let $\mathfrak{S}'$ be a marked surface and $c_1$ and $c_2$ respectively a right and a left boundary edges of $\mathfrak{S}'$. Then $\mathscr{S}(\mathfrak{S}')$ has a structure of $(\Oq,\Oq)$-bicomodule. We consider $\mathfrak{S} = \mathfrak{S}'/_{c_1=c_2}$ the marked surface obtained by gluing $c_1$ to $c_2$, and $c$ the ideal arc formed by $c_1=c_2$ in $\mathfrak{S}$. We have $\mathfrak{S}' = {\rm Cut}_c(\mathfrak{S})$, and Theorem~\ref{splitmorph} gives an injective algebra morphism $\rho_c\colon\mathscr{S}(\mathfrak{S}) \to \mathscr{S}(\mathfrak{S}')$.
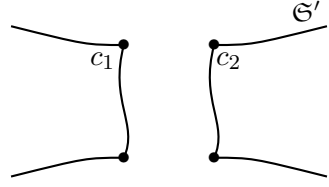
$$
\begin{tikzpicture}
\draw [thick] (0,2) .. controls (1,1.75) .. (1.5,1.75);
\draw [thick] (0,0) .. controls (1,0.25) .. (1.5,0.25);
\draw [thick] (1.5, 1.75) node{\small $\bullet$} .. controls (1.3,1) and (1.7,0.7) .. (1.5,0.25) node{\small $\bullet$} node[pos = 0.1,left = -2pt]{$c_1$};
\begin{scope}[xshift = 1.2cm]
\draw [thick] (3,2) .. controls (2,1.75).. (1.5,1.75) node[pos = 0.1, above]{\small $\mathfrak{S}'$};
\draw [thick] (3,0) .. controls (2,0.25).. (1.5,0.25);
\draw [thick] (1.5, 1.75) node{\small $\bullet$} .. controls (1.3,1) and (1.7,0.7) .. (1.5,0.25) node{\small $\bullet$} node[pos = 0.1,right = -2pt]{$c_2$};
\end{scope}
\end{tikzpicture}
$$

\begin{Definition}
Let $H$ be a Hopf algebra and $M$ an $(H,H)$-bicomodule, with coproducts denoted by $\Delta_1\colon M \to M \otimes H$ and $\Delta_2\colon M \to H \otimes M$. The 0-th Hochschild cohomology of $M$ is the subalgebra of $M$ defined as $HH^0(M) := \{ x \in M\ /\ \Delta_1 (x) = \operatorname{fl}\circ\Delta_2 (x)\}$.
\end{Definition}
\begin{Theorem}[{\cite[Section~2.3]{KorQue}, see also \cite[Theorem~4.8]{Cos}}] \label{thmExcOfSS}
Let $\mathfrak{S} = \mathfrak{S}'/_{c_1 = c_2}$ be a gluing. The splitting morphism $\rho_c\colon \mathscr{S}(\mathfrak{S}) \to \mathscr{S}(\mathfrak{S}')$ maps isomorphically $\mathscr{S}(\mathfrak{S})$ on $HH^0(\mathscr{S}(\mathfrak{S}'))$.
\end{Theorem}
\begin{Remark}
If $\mathfrak{S}_1$ and $\mathfrak{S}_2$ are two marked surfaces, $c_1$ is a right boundary edge of $\mathfrak{S}_1$, $c_2$ a left boundary edge of $\mathfrak{S}_2$ and $\mathfrak{S}' = \mathfrak{S}_1 \sqcup \mathfrak{S}_2$. Then $\mathscr{S}(\mathfrak{S}') \simeq \mathscr{S}(\mathfrak{S}_1)\otimes_k \mathscr{S}(\mathfrak{S}_2)$, note that this is the tensor product as vector spaces and not as comodules, and the 0-th Hochschild cohomology of $\mathscr{S}(\mathfrak{S}')$ corresponds to the cotensor product of $\mathscr{S}(\mathfrak{S}_1)$ and $\mathscr{S}(\mathfrak{S}_2)$ over $\mathscr{S}(B)$, namely $\mathscr{S}(\mathfrak{S}_1) \square_H \mathscr{S}(\mathfrak{S}_2) := \{ x \in \mathscr{S}(\mathfrak{S}_1)\otimes \mathscr{S}(\mathfrak{S}_2)\ /\ \Delta_1 \otimes {\rm Id}_2 (x) = {\rm Id}_1 \otimes \Delta_2 (x)\}$.
\end{Remark}

\subsection{Internal skein algebras}\label{SectASigma}
When a surface has a boundary edge, one can push a little disk inside the surface from this boundary edge. This process induces an action of the skein category of the disk on the skein category of the surface. Namely, ${\rm Sk}_\mathcal{V}(\Sigma)$ is a ${\rm Sk}_\mathcal{V}\big(\mathbb{R}^2\big)$-module category, see \cite[Sections~3.2]{Cooke}. In this situation there is a notion of internal $\Hom$ objects that encode entirely in ${\rm Sk}_\mathcal{V}\big(\mathbb{R}^2\big)$ the behavior of objects of ${\rm Sk}_\mathcal{V}\big(\mathbb{R}^2\big)$ seen in ${\rm Sk}_\mathcal{V}(\Sigma)$ by the action. Namely for fixed $V,W \in {\rm Sk}_\mathcal{V}(\Sigma)$ one has $\Hom_{{\rm Sk}_\mathcal{V}(\Sigma)}(X \rhd V,W) \simeq \Hom_{{\rm Sk}_\mathcal{V}(\mathbb{R}^2)}(X, \underline{\Hom}(V,W))$ naturally in $X$, see \cite[Section~7.9]{EGNO}. However, such an object does not always exist in ${\rm Sk}_\mathcal{V}\big(\mathbb{R}^2\big)$, and actually lives in its free cocompletion. The internal skein algebra $A_\Sigma$ of the surface is the internal endomorphism algebra of the empty set $\underline{\Hom}(\varnothing,\varnothing)$, see~\cite{Gunningham} or~\cite{BBJ} together with \cite{Cooke}. This means one can understand ribbon graphs in ${\rm Sk}_\mathcal{V}(\Sigma)$ with boundary points on the bottom and near the boundary edge as morphisms in (the free cocompletion of) ${\rm Sk}_\mathcal{V}\big(\mathbb{R}^2\big)$ with target $A_\Sigma$.

Note that in \cite{Cooke}, \cite{BBJ} and \cite{Gunningham} one uses a right ${\rm Sk}_\mathcal{V}\big(\mathbb{R}^2\big)$-action and we use a left ${\rm Sk}_\mathcal{V}\big(\mathbb{R}^2\big)$-action here. See Section~\ref{SectMultiEdges} for more details.

\begin{Definition}\label{algStrOnIntHom}
Let $\mathcal{A}$ be a $k$-linear monoidal category and $\mathcal{M}$ a left $\mathcal{A}$-module category. Let~$M_1$ and~$M_2$ be two objects of~$\mathcal{M}$. The internal $\Hom$ of $M_1$ and~$M_2$ with respect to the $\mathcal{A}$-module structure is an object $\undHom(M_1,M_2)\in \mathcal{A}$
equipped with a natural isomorphism $\eta\colon \Hom_\mathcal{M}(-\rhd M_1,M_2) \tilde{\to} \Hom_\mathcal{A}(-,\undHom(M_1,M_2))$. It is unique up to canonical isomorphism when it exists.
When all involved internal $\Hom$ objects exist, one can define: the eva\-lua\-tion map ${\rm ev}_{M_1,M_2}\colon \undHom(M_1,M_2) \rhd M_1\to M_2$ is the image under $\eta^{-1}$ of ${\rm Id}_{\undHom(M_1,M_2)}$ and the composition map $c\colon \undHom(M_2,M_3) \otimes \undHom(M_1,M_2) \to \undHom(M_1,M_3)$ is the image under~$\eta$ of the morphism ${\rm ev}_{M_2,M_3} \circ ({\rm Id}_{\undHom(M_2,M_3)} \rhd {\rm ev}_{M_1,M_2})\colon (\undHom(M_2,M_3) \otimes \undHom(M_1,M_2))\rhd M_1 \to \undHom(M_2,M_3) \rhd$ $M_2\to M_3$.
In~particular, an internal endomorphism $\undEnd(M) := \undHom(M,M)$ is an algebra object in~$\mathcal{A}$, with unit the morphism $1_\mathcal{A} \to \undEnd(M)$ associated with ${\rm Id}_M$.
\end{Definition}

The functor $\Hom_\mathcal{M}(-\rhd M_1, M_2)\colon\mathcal{A}^{\rm op} \to {\rm Vect}_k$ cannot always be represented in $\mathcal{A}$. However, it is always an object of its free cocompletion.
\begin{Definition}
The 2-category ${\rm Cocomp}_k$ has objects essentially small $k$-linear cocomplete categories and morphisms $k$-linear cocontinuous functors between them, i.e., functors preserving colimits and natural transformations between functors.
It is symmetric monoidal with the Kelly tensor product $\boxtimes$ for cocomplete categories, which is characterized by
\[
\Hom_{{\rm Cocomp}_k}(\mathcal{A} \boxtimes \mathcal{B},\mathcal{C}) \simeq \Hom_{{\rm Cocomp}_k}(\mathcal{A},\Hom_{{\rm Cocomp}_k}(\mathcal{B},\mathcal{C})) \simeq {\rm Cocont}(\mathcal{A},\mathcal{B};\mathcal{C}),
\]
see \cite[Section~6.5]{Kelly2} and \cite[Theorem~2.45]{Ramos}.

The free cocompletion of a $k$-linear category $\mathcal{C}$ is a cocomplete category $\Free(\mathcal{C}) \in {\rm Cocomp}_k$ together with a functor $i\colon \mathcal{C} \to \Free(\mathcal{C})$ which is initial among functors to a $k$-linear cocomplete category. Namely, $i_*\colon \Hom_{{\rm Cat}_k}(\mathcal{C},\mathcal{D}) \to \Hom_{{\rm Cocomp}_k}(\Free(\mathcal{C}),\mathcal{D})$ is an equivalence of categories for any $\mathcal{D}\in {\rm Cocomp}_k$. The free cocompletion $\Free(\mathcal{C})$ is unique up to essentially unique equivalence.
\end{Definition}

\begin{Remark}\label{rmkFreeIsPresheaves}
A standard choice for the free cocompletion is the presheaf category $[\mathcal{C}^{\rm op},{\rm Vect}_k]$, in which $\mathcal{C}$ embeds by the Yoneda embedding. Moreover for any other choice of free cocompletion~$\hat{\mathcal{C}}$, the canonical equivalence is given by
\[
R\colon\ \begin{cases}
\hat{\mathcal{C}}  \to  [\mathcal{C}^{\rm op},{\rm Vect}_k],
\\
X \mapsto  \Hom_{\hat{\mathcal{C}}}(i(-),X).\end{cases}
\]
See \cite[Section~2.2]{Dugger} for more details. From now on, we assume that we made this standard choice and set $\Free(\mathcal{C}) := [\mathcal{C}^{\rm op},{\rm Vect}_k]$.
Note that every object $C$ of $\mathcal{C}$ seen as an object of $\Free(\mathcal{C})$ is compact projective, i.e., the functor $\Hom_{\Free(\mathcal{C})}(C,-)$ is cocontinuous.
\end{Remark}

It is then easy to verify that the free cocompletion $\Free\colon {\rm Cat}_k \to {\rm Cocomp}_k$ is a symmetric monoidal functor. Hence if $\mathcal{A} \in {\rm Cat}_k$ is a monoidal $k$-linear category and $\mathcal{M}$ an $\mathcal{A}$-module category, after free cocompletions $\Free(\mathcal{A})$ is again monoidal and $\Free(\mathcal{M})$ is a $\Free(\mathcal{A})$-module category.

\begin{Proposition}\label{propIntHomFree}
Let $\mathcal{A} \in {\rm Cat}_k$ be a monoidal $k$-linear category, $\mathcal{M}$ an $\mathcal{A}$-module category and~$M_1$, $M_2$ two objects of $\mathcal{M}$. The presheaf $F= \Hom_\mathcal{M}(-\rhd M_1,M_2) \in \Free(\mathcal{A})$ is the internal $\Hom$ object of $M_1$ and $M_2$ $($seen as objects of $\Free(\mathcal{M}))$ with respect to the $\Free(\mathcal{A})$-module structure.
\end{Proposition}

\begin{proof}
For the ``small'' objects $A \in \mathcal{A} \hookrightarrow \Free(\mathcal{A})$, the isomorphism $\Hom_{\Free(\mathcal{A})}(A,F)\simeq F(A) := \Hom_{\mathcal{M}}(A\rhd M_1,M_2)$ is given by the Yoneda lemma.
Now any object $X\in\Free(\mathcal{A})$ is obtained as a~colimit of such small objects, $X = \colim_i A_i$, $A_i \in \mathcal{A}$, by the co-Yoneda lemma. Then it is straight\-forward to check that
\begin{gather*}
\Hom_{\Free(\mathcal{A})}(X,F) \simeq \lim_i \Hom_{\Free(\mathcal{A})}(A_i,F) \simeq \lim_i \Hom_{\Free(\mathcal{M})}(A_i \rhd M_1,M_2)
\\ \hphantom{\Hom_{\Free(\mathcal{A})}(X,F)}
{}\simeq \Hom_{\Free(\mathcal{M})}(\colim_i (A_i \rhd M_1),M_2) \simeq \Hom_{\Free(\mathcal{M})}(X \rhd M_1,M_2).
\end{gather*}
Here we kept the notation $\rhd$ for its essentially unique cocontinuous extention to free cocompletions.
\end{proof}

This means that when one works with free cocompletions of a module structure on some ``small'' categories, the internal $\Hom$ objects of the free cocompletions are completely described by what happens on the ``small'' subcategories.
Using the same argument one can rephrase the definition of the internal $\Hom$ object of $M_1$ and~$M_2$ in a free cocompletion $\hat{\mathcal{A}}$ of $\mathcal{A}$ as an object $X\in \hat{\mathcal{A}}$ together with an isomorphism of presheaves $\Hom_\mathcal{M}(-\rhd M_1,M_2) \tilde{\to} R(X)$, where $R$ is the equivalence from Remark~\ref{rmkFreeIsPresheaves}.

\begin{Proposition}\label{propOqFree}
The inclusion $\Oq\text{-}{\rm comod}^{\rm fin} \hookrightarrow \Oq\text{-}{\rm comod}$ is a free cocompletion.
\end{Proposition}
\begin{proof}
At generic $q$, the category $\Oqcom$ is semisimple and its simples are finite-dimen\-sio\-nal. Hence any $\Oq$-comodule is a direct sum of these, and any colimit in \linebreak $\Oqcomfin$ is simply a direct sum.
\end{proof}

Note that the monoidal structure on $\Oqcom$ is the free cocompletion of the monoidal structure on $\Oqcomfin$ as it extends it and commutes with direct sums in both factors.
In the preceding proposition, one could replace $\Oqcomfin$ with its full subcategory ${\rm TL}$, as every $\Oq$-comodule is a quotient of direct sums of objects of ${\rm TL}$. Actually, the inclusion ${\rm Sk}_{{\rm TL}}(\Sigma) \hookrightarrow \Free\big({\rm Sk}_{\Oqcomfin}(\Sigma)\big)$ is also a free cocompletion by \cite[Theorems~2.28 and~3.27]{Cooke}, using factorisation homology.

We now apply the internal $\Hom$ object constructions to the case of skein categories. We choose a~$k$-linear ribbon category $\mathcal{V}$ and choose a free cocompletion $\mathcal{E}$.

\begin{Definition}\label{defLeftActSkR2}
We consider an oriented surface with boundary $\Sigma$, with a ``red'' arc chosen on its boundary, seen at the left. This arc can be thickened in the surface, which gives a thick left embedding $(0,1) \to\partial \Sigma$. In particular one has an embedding of surfaces $(0,1)\times[0,1]\sqcup\Sigma \hookrightarrow \Sigma$ depicted hereby. By Proposition~\ref{Skmodcat}, this gives a structure of left ${\rm Sk}_\mathcal{V}((0,1)\times[0,1]) \simeq {\rm Sk}_\mathcal{V}\big(\mathbb{R}^2\big)$-module category on~${\rm Sk}_\mathcal{V}(\Sigma)$. We denote the action functor by \mbox{$\rhd\colon{\rm Sk}_\mathcal{V}\big(\mathbb{R}^2\big) \otimes {\rm Sk}_\mathcal{V}(\Sigma) \to {\rm Sk}_\mathcal{V}(\Sigma)$}.
\end{Definition}
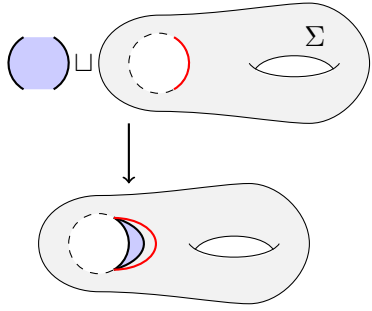
$$
\begin{tikzpicture}[yscale = 0.8, xscale = -0.8]
\fill[gray!10] (4.5,1) .. controls (4.5,0.5) and (4,0.2) .. (3.5,0.2)	.. controls (2.5,0.2) and (1.5,0 ) .. (1,0)	.. controls (0.5,0 ) and (0,0.5 ) .. (0,1)	.. controls (0,1.5 ) and (0.5,2 ) .. (1,2)	.. controls (1.5,2 ) and (2.5,1.8) .. (3.5,1.8) .. controls (4,1.8) and (4.5,1.5) .. (4.5,1) -- cycle ;
\fill[white] (0.6,0.9) .. controls (0.7,0.7) and (1.8,0.7) .. (1.9,0.9) .. controls (1.7,1.2) and (0.8,1.2) .. (0.6,0.9);
\fill[white] (3.5,1) circle(0.5);
\draw[dashed] (3.5,1) circle(0.5);
\draw[thick, red] (240:0.5)++(3.5,1) arc(240:120:0.5);
\draw (0.5,1) .. controls (0.7,0.7) and (1.8,0.7) .. (2,1);
\draw (0.6,0.9) .. controls (0.8,1.2) and (1.7,1.2) .. (1.9,0.9) node[midway, above right]{$\Sigma$};
\draw (4.5,1) .. controls (4.5,0.5) and (4,0.2) .. (3.5,0.2)	.. controls (2.5,0.2) and (1.5,0 ) .. (1,0)	.. controls (0.5,0 ) and (0,0.5 ) .. (0,1)	.. controls (0,1.5 ) and (0.5,2 ) .. (1,2)	.. controls (1.5,2 ) and (2.5,1.8) .. (3.5,1.8) .. controls (4,1.8) and (4.5,1.5) .. (4.5,1);
\node (sq) at (4.75,1){$\sqcup$}; \draw[->, thick] (4,0) -- ++(0,-1) node[midway, right]{};
\fill[blue!20](240:0.5)++(5.5,1) arc(240:120:0.5) -- ++(0.5,0) arc(60:-60:0.5) -- ++(-0.5,0);
\draw[thick] (240:0.5)++(5.5,1) arc(240:120:0.5);
\draw[thick] (-60:0.5)++(5.5,1) arc(-60:60:0.5);
\begin{scope}[yshift = -3cm, xshift = 1cm]
\fill[gray!10] (4.5,1) .. controls (4.5,0.5) and (4,0.2) .. (3.5,0.2)	.. controls (2.5,0.2) and (1.5,0 ) .. (1,0)	.. controls (0.5,0 ) and (0,0.5 ) .. (0,1)	.. controls (0,1.5 ) and (0.5,2 ) .. (1,2)	.. controls (1.5,2 ) and (2.5,1.8) .. (3.5,1.8) .. controls (4,1.8) and (4.5,1.5) .. (4.5,1) -- cycle ;
\fill[white] (0.6,0.9) .. controls (0.7,0.7) and (1.8,0.7) .. (1.9,0.9) .. controls (1.7,1.2) and (0.8,1.2) .. (0.6,0.9);
\fill[white] (3.5,1) circle(0.5);
\fill[blue!20] (240:0.5)++(3.5,1) arc(240:120:0.5) arc(120:240:1 and 0.5);
\draw[dashed] (3.5,1) circle(0.5);
\draw[thick] (240:0.5)++(3.5,1) arc(240:120:0.5);
\draw[thick] (240:0.5)++(3.5,1) arc(240:120:1 and 0.5);
\draw[thick, red] (240:0.5)++(3.5,1) arc(270:90:0.7 and 0.43);
\draw (0.5,1) .. controls (0.7,0.7) and (1.8,0.7) .. (2,1);
\draw (0.6,0.9) .. controls (0.8,1.2) and (1.7,1.2) .. (1.9,0.9) node[midway, above right]{};
\draw (4.5,1) .. controls (4.5,0.5) and (4,0.2) .. (3.5,0.2)	.. controls (2.5,0.2) and (1.5,0 ) .. (1,0)	.. controls (0.5,0 ) and (0,0.5 ) .. (0,1)	.. controls (0,1.5 ) and (0.5,2 ) .. (1,2)	.. controls (1.5,2 ) and (2.5,1.8) .. (3.5,1.8) .. controls (4,1.8) and (4.5,1.5) .. (4.5,1);
\end{scope}
\end{tikzpicture}
$$
To simplify notation, we write ${\rm SK}_\mathcal{V}(-) := \Free({\rm Sk}_\mathcal{V}(-))$ and still denote the action functor on free cocompletions by $\rhd\colon \mathcal{E} \boxtimes {\rm SK}_\mathcal{V}(\Sigma)\to {\rm SK}_\mathcal{V}(\Sigma)$. For $M_1$ and $M_2$ two objects of ${\rm Sk}_\mathcal{V}(\Sigma)$, one has an internal $\Hom$ object $\undHom(M_1,M_2) \in \mathcal{E}$.

\begin{Definition}
Let $\Sigma$ be a surface with boundary with a chosen thickened arc on its boundary.
The internal skein algebra $A_\Sigma := \undHom(\varnothing,\varnothing)$ is the endomorphism algebra of $\varnothing \in {\rm Sk}_\mathcal{V}(\Sigma) \subseteq {\rm SK}_\mathcal{V}(\Sigma)$ with respect to the $\mathcal{E}$-module structure. Its algebra structure is recalled in Definition~\ref{algStrOnIntHom}.
Explicitly, $A_\Sigma$ comes equipped with a natural isomorphism $\sigma \colon \Hom_{{\rm SK}_\mathcal{V}(\Sigma)}(-\rhd \varnothing, \varnothing) \tilde{\Rightarrow} \Hom_{\mathcal{E}}(-, A_\Sigma)$ between (contravariant) functors $\mathcal{E}\to {\rm Vect}_k$.
If $\mathcal{V} = \Oq\text{-}{\rm comod}^{\rm fin}$ with $\mathcal{E}=\Oq\text{-}{\rm comod}$, then $A_\Sigma$ is an $\Oq$-comodule algebra.
\end{Definition}

\begin{Remark}\label{rmkTopPtsAndASigma}
The defining properties of $A_\Sigma$ enables one to describe morphisms $V \rhd \varnothing \to \varnothing$ in ${\rm Sk}_\mathcal{V}(\Sigma)$, so where the boundary points of tangles are at the bottom, as morphisms $V \to \mathscr{S}(\Sigma)$ in $\mathcal{E}$. We would also like to describe morphisms $W\rhd \varnothing \to V \rhd \varnothing$ in ${\rm Sk}_\mathcal{V}(\Sigma)$. By duality in~$\mathcal{V}$, the functor $V\rhd-$ is right adjoint to $V^*\rhd-$, see \cite[Proposition~7.1.6]{EGNO}. So one has natural isomorphisms:
\begin{align*}
\Hom_{{\rm Sk}_\mathcal{V}(\Sigma)}(W\rhd\varnothing, V \rhd\varnothing) &\,\tilde{\to} \, \Hom_{{\rm Sk}_\mathcal{V}(\Sigma)}(V^*\rhd( W \rhd \varnothing), \varnothing) \overset{\sigma_{V^*\otimes W}}{\to} \Hom_{\mathcal{E}}(V^*\otimes W, A_\Sigma)
\\
&\, \tilde{\to} \, \Hom_{\mathcal{E}}(W,V\otimes A_\Sigma).
\end{align*}
When $\Sigma$ is connected, every object of ${\rm Sk}_\mathcal{V}(\Sigma)$ is isomorphic to one of the form $V\rhd\varnothing$, and the above natural isomorphism suggests that the algebra $A_\Sigma\in\mathcal{E}$ is enough to fully describe ${\rm SK}_\mathcal{V}(\Sigma)$.
\end{Remark}
\begin{Theorem}[{\cite[Theorem~5.14]{BBJ}}] \label{thASigmaMods}
Suppose that $\Sigma$ is connected, then there is an equivalence of categories
\begin{align*}{\rm SK}_\mathcal{V}(\Sigma) &\; \tilde{\to}\; {\rm mod}_\mathcal{E}-A_\Sigma,
\\
M &\mapsto  \undHom(\varnothing,M)
\end{align*}
between the free cocompletion of ${\rm Sk}_\mathcal{V}(\Sigma)$ and the category of right modules over $A_\Sigma$ in $\mathcal{E}$. For $M$ of the form $V \rhd \varnothing$, $V\in\mathcal{E}$, which is always the case for $M \in {\rm Sk}_\mathcal{V}(\Sigma)$, this functor is given by $V\rhd \varnothing \mapsto V \otimes A_\Sigma$.
\end{Theorem}
\begin{proof}
For the last statement, one has $\undHom(\varnothing,V\rhd\varnothing) \simeq V \otimes A_\Sigma$ by Remark \ref{rmkTopPtsAndASigma}. This is the general idea of the proof, as morphisms of $A_\Sigma$-modules from $W \otimes A_\Sigma$ to $V \otimes A_\Sigma$ are equivalent to morphisms in $\mathcal{E}$ from $W$ to $V\otimes A_\Sigma$, which are equivalent by the above Remark to morphisms from $W\rhd\varnothing$ to $V\rhd \varnothing$ in ${\rm SK}_\mathcal{V}(\Sigma)$.

For the details, one uses Barr--Beck reconstruction, or more precisely its reformulation in \cite[Theorem~4.6]{BBJ}. One has to check that $\varnothing \in {\rm SK}_\mathcal{V}(\Sigma)$ is a progenerator.
It is projective:
\[
{\rm act}_\varnothing^R= \undHom(\varnothing,-) \colon\
\begin{cases}
{\rm SK}_\mathcal{V}(\Sigma)  \to  [\mathcal{V}^{\rm op},{\rm Vect}_k]\simeq \mathcal{E},
\\
M  \mapsto  \Hom_{{\rm SK}_\mathcal{V}(\Sigma)}(-\rhd\varnothing,M)
\end{cases}
\]
is cocontinuous because $-\rhd \varnothing \colon \mathcal{V} \to {\rm Sk}_\mathcal{V}(\Sigma) \subseteq {\rm SK}_\mathcal{V}(\Sigma)$ takes values in compact projective objects.
It~is a~generator: ${\rm act}_\varnothing^R$ being faithful is equivalent to $-\rhd \varnothing\colon \mathcal{E} \to {\rm SK}_\mathcal{V}(\Sigma)$ being dominant by \cite[Remark~4.9]{BBJ}, which is the case because $-\rhd \varnothing$ is essentially surjective on ${\rm Sk}_\mathcal{V}(\Sigma)$ which generate ${\rm SK}_\mathcal{V}(\Sigma)$ under colimits.
One gets right modules over $A_\Sigma$ because we considered left module categories, see \cite[Remark~4.7]{BBJ}.
\end{proof}

\begin{Remark} We are in the same context as \cite{BBJ} and \cite{Gunningham}. In \cite[Definition~2.18]{Gunningham}, the internal skein algebra is defined similarly as $\Hom_{{\rm Sk}_\mathcal{V}(\Sigma)}( \varnothing\lhd -,\varnothing) \in \Free(\mathcal{V})$ for $\mathcal{V}$ a $k$-linear ribbon category whose unit $1_\mathcal{V}$ is simple, which is the case for $\mathcal{V}=\Oq\text{-}{\rm comod}^{\rm fin}$.

In \cite[Definition~5.3]{BBJ}, the moduli algebra $A_\Sigma$ is defined to be the endomorphism algebra of the distinguished object $\mathcal{O}_{\mathcal{E},\Sigma}$ of the factorization homology over $\Sigma$ of a presentable abelian balanced $k$-linear category $\mathcal{E}$ generated under filtered colimits by rigid objects, with respect to the $\mathcal{E}$-module structure. The factorization homology of $\mathcal{V}$ is computed by ${\rm Sk}_\mathcal{V}$, see \cite[Theorem~2.28]{Cooke}, and the factorization homology of $\mathcal{E} = \Free(\mathcal{V})$ by ${\rm SK}_\mathcal{V} = \Free({\rm Sk}_\mathcal{V})$, see \cite[Theorem~3.27]{Cooke}. The distinguished object is $\varnothing \in {\rm Sk}_\mathcal{V}(\Sigma) \subseteq {\rm SK}_\mathcal{V}(\Sigma)$, and a $k$-linear free cocompletion is always abelian and generated under filtered colimit by rigid objects.

There is one difference with our paper though: we consider a left $\mathcal{E}$-action, and \cite{BBJ,Gunningham} consider a~right $\mathcal{E}$-action, with adapted definitions of internal $\Hom$ objects. In our description, it would mean seeing the red arc on the right instead of on the left, thus ${\rm Sk}_\mathcal{V}(\mathbb{R}^2)$ acting from the right. This gives a~braided opposite product, and we need left actions to have the same product than the one on $\mathscr{S}(\Sigma)$. Right actions and more generally multiple right/left actions and how they interact will be studied in Section~\ref{SectMultiEdges}.
\end{Remark}

\section{The relation}\label{SectRelation}
We show in this section that the stated skein algebra of a surface with a single boundary edge is isomorphic to $A_\Sigma$. We consider the algebra $A_\Sigma$ from last section with $\mathcal{V} = \Oq\text{-}{\rm comod}^{\rm fin}$ and $\mathcal{E} = \Oq\text{-}{\rm comod}$ at generic $q$, and prove that $\mathscr{S}(\Sigma)\simeq A_\Sigma$ as $\Oq$-comodule algebras. Actually since $A_\Sigma$ is only defined up to canonical isomorphism one may take an equality here, so we prove that $\mathscr{S}(\Sigma)$ satisfies the defining properties of $A_\Sigma$, namely that it is the internal endomorphism algebra of the empty set in ${\rm SK}_\mathcal{V}(\Sigma)$ with respect to the $\mathcal{E}$-module structure. This result is not new and was stated in a weaker form in \cite[Theorem~4.4]{LeYu}, \cite[Theorem~9.1]{LeSikora} and \cite[Remark 2.21]{Gunningham}, namely as algebras in $\rm Vect$. However, in these references one considers right $\mathcal{E}$-actions and therefore the internal skein algebra is isomorphic to the braided opposite of the stated skein algebra. The full result can still be recovered using \cite[Theorem~5.3]{Faitg} or \cite{Korinman}, which give an isomorphism between the opposite of the stated skein algebras and Alekseev--Grosse--Schomerus-algebras, which are themselves isomorphic to internal skein algebras by \cite{BBJ}. Our approach here is more direct and uses only skein theory.

We need a natural isomorphism $\St_W\colon \Hom_{{\rm SK}_\mathcal{V}(\Sigma)}(W \rhd \varnothing,\varnothing) \tilde{\Rightarrow} \Hom_{\mathcal{E}}(W, \mathscr{S}(\Sigma))$, for $W \in \mathcal{E} = \Oq\text{-}{\rm comod}$. In the case where $W = V^{\otimes n}\in {\rm TL}$ is a tensor product of standard corepresentations, and the element on the left hand side is a tangle $\alpha$ with $n$ boundary points, we want a morphism $V^{\otimes n}\to \mathscr{S}(\Sigma)$ associated to it. This is done by setting the entries, elements of $V^{\otimes n}$, as states of the tangle $\alpha$. Recall that $V$ has basis $v_+$, $v_-$ and we identify the states + and $-$ with these elements. As~in~Remark~\ref{rmkFormActOnSS}, we allow formal linear combination of states for stated tangles. In this context, the defining relations of stated skein algebras correspond exactly to naturality conditions, see~Proposition~\ref{propStNat}.
This gives the idea of how to deal with the objects of the full subcategory ${\rm TL} \subseteq \mathcal{V}$ of objects of the form $V^{\otimes n}$, and this extends to $\mathcal{E} \simeq \Free({\rm TL})$ by Proposition~\ref{propIntHomFree}. Note that one still has an action $\rhd \colon {\rm TL} \otimes {\rm Sk}_{{\rm TL}}(\Sigma) \to {\rm Sk}_{{\rm TL}}(\Sigma)$ which is the restriction of $\rhd \colon \mathcal{V} \otimes {\rm Sk}_\mathcal{V}(\Sigma) \to {\rm Sk}_\mathcal{V}(\Sigma)$.
\begin{Definition}
Let $W = V^{\otimes n}\in {\rm TL}$ and $\alpha \in \Hom_{{\rm Sk}_{{\rm TL}}(\Sigma)}(W \rhd \varnothing,\varnothing)$. By Theorem~\ref{thmSkAsTangles}, $\alpha$~can be represented by a (linear combination of) tangle(s), still denoted by $\alpha$, with $n$ ordered boundary points near the boundary edge, which is well defined up to isotopy and Kauffman-bracket relations.

Graphically, we set
\[
\St_W\Bigg(
\begin{tikzpicture}[scale = 1, baseline = 8pt]
\node [rightymissy, minimum width = 1cm] (A) at (1.5,1) {$\alpha$};
\draw (0,0) -- ++(2,0) node[pos = 0,below left =-2pt]{\footnotesize $\partial \Sigma$} node[pos = 0.5,below]{\footnotesize $\Sigma$};
\draw (0,0) -- ++ (0,1) node[pos = 0.5,left]{\footnotesize $[0,1]$};
\draw[blue] (0.5,0) node[above left = -2pt]{\tiny 1} .. controls (0.5,1) .. (A);
\draw[olive] (1,0) node[above left = -2pt]{\tiny $\cdots$} .. controls (1,0.7) .. (A);
\draw[red] (1.5,0) node[above left = -2pt]{\tiny $n$} -- (A);
\end{tikzpicture} \Bigg) \ (v_{\varepsilon_1}\otimes \cdots \otimes v_{\varepsilon_n}) := \begin{tikzpicture}[scale = 1, baseline = 8pt]
\node [rightymissy, minimum width = 1cm] (A) at (1.5,1) {$\alpha$};
\draw (0,0) -- ++(2,0) ;
\draw[->] (0,0) -- ++ (0,1);
\draw[blue] (0,0.75) node[ left = -2pt]{\footnotesize $\varepsilon_1$} -- (A);
\draw[olive] (0,0.5) node[ above , rotate = 90]{\tiny $\cdots$} .. controls (1,0.5) .. (A);
\draw[red] (0,0.25) node[ left = -2pt]{\footnotesize $\varepsilon_n$} .. controls (1.5,0.25) .. (A);
\end{tikzpicture}\,,\qquad \varepsilon_i \in \{\pm\}.
\]
For brevity we denote this last stated tangle by $_{\vec{\varepsilon}}\,\alpha$ and $v_{\vec{\varepsilon}} = v_{\varepsilon_1}\otimes \cdots \otimes v_{\varepsilon_n}$. Note that in this figure there are two implicit sums, $\alpha$ is a linear combination of tangles and an element of $V^{\otimes n}$ is a linear combination of $v_{\vec{\varepsilon}}$'s. They will remain implicit in the following.

In the context of stated skein algebras one needs the boundary points of the tangle $\alpha$ to be above the boundary arc though they are on the bottom at its right in the context of morphisms in ${\rm Sk}_{{\rm TL}}(\Sigma)$. One wants to simply push the boundary points through the boundary edge and up above it, but without braiding the strands with one another. Therefore we use a global automorphism of $\Sigma \times [0,1]$. Consider an isotopy of the identity on $\Sigma \times [0,1]$ which is trivial far from the corner and in it pushes the bottom boundary through and up above the edge. It results in the global automorphism $\psi_{hv}$ of $\Sigma \times [0,1]$ which maps a tangle ``from skein categories'' to one ``from stated skein algebras'' and preserves the order as desired, namely the bottom points on the leftmost will end up above. We will still denote the modified tangle $\psi_{hv}(\alpha)$ by $\alpha$. Now, it only needs states to give an element of $\mathscr{S}(\Sigma)$.

We set
\[
\St_W(\alpha) := \begin{cases}
V^{\otimes n} \to \mathscr{S}(\Sigma),
\\
v_{\vec{\varepsilon}} \mapsto {}_{\vec{\varepsilon}}\,\alpha,
\end{cases}
\]
so where ${}_{\vec{\varepsilon}}\,\alpha$ is the tangle $\psi_{hv}(\alpha)$ with states $\varepsilon_1,\dots,\varepsilon_n$ from top to bottom. It is well defined because the tangle representing $\alpha$ is well defined up to isotopy and Kauffman-bracket relations, which are quotiented out in $\mathscr{S}(\Sigma)$.
\end{Definition}

\begin{Proposition}\label{propStOqmorph}
Given a tangle $\alpha$, the map $\St_W(\alpha)\colon V^{\otimes n}\to\mathscr{S}(\Sigma)$ is an $\Oq$-comodule morphism.
\end{Proposition}
\begin{proof}
Note that we still see $\mathscr{S}(\Sigma)$ as a right $\Oq$-comodule even though we draw the edge at the left. Its comodule structure is given by $\Delta(_{\vec{\varepsilon}}\,\alpha) = \sum_{\vec{\eta} \in \{\pm\}^n} {}_{\vec{\eta}}\alpha \otimes {}_{\vec{\eta}}\beta_{\vec{\varepsilon}}$, and here~${}_{\vec{\eta}}\beta_{\vec{\varepsilon}}$ is a~product $\prod_i {}_{\eta_i}\beta_{\varepsilon_i}$ in $\mathscr{S}(B)$. The comodule structure on $V^{\otimes n}$ is given by $\Delta(v_{\vec{\varepsilon}}) = \sum_{\vec{\eta}} v_{\vec{\eta}} \otimes \prod_i x_{\eta_i,\varepsilon_i}$, where $x_{\eta,\varepsilon}$~is the $v_\eta$ part of $\Delta(v_\varepsilon)$. Namely, $x_{+,+} = a$, $x_{+,-}=b$, $x_{-,+}=c$ and $x_{-,-}=d$. Under the isomorphism $\Oq \simeq \mathscr{S}(B)$, $x_{\eta,\varepsilon}\mapsto {}_\eta\beta_{\varepsilon}$.
 Thus $\Delta(\St_W(\alpha)(v_{\vec{\varepsilon}})) = \Delta(_{\vec{\varepsilon}}\,\alpha) = \sum_{\vec{\eta}} \St_W(\alpha)(v_{\vec{\eta}}) \otimes {}_{\vec{\eta}}\beta_{\vec{\varepsilon}} = (\St_W(\alpha)\otimes {\rm Id}_{\mathscr{S}(B)})(\Delta(v_{\vec{\varepsilon}}))$.
\end{proof}

\begin{Proposition}\label{propStNat}
The maps $\St_W \colon \Hom_{{\rm Sk}_{{\rm TL}}(\Sigma)}(W\rhd \varnothing,\varnothing) \to \Hom_{\Oq\text{-}{\rm comod}}(W, \mathscr{S}(\Sigma))$, $W \in {\rm TL}$, define a natural transformation $\St\colon \Hom_{{\rm Sk}_{{\rm TL}}(\Sigma)}(-\rhd \varnothing,\varnothing) \Rightarrow \Hom_{\mathcal{E}}(-, \mathscr{S}(\Sigma))$ between functors ${\rm TL}^{\rm op}\to {\rm Vect}_k$.
\end{Proposition}
\begin{proof}
For $g \in \Hom_{{\rm TL}}(V^{\otimes n},V^{\otimes m})$, $\alpha \in \Hom_{{\rm Sk}_{{\rm TL}}(\Sigma)}(V^{\otimes m}\rhd \varnothing,\varnothing)$ and $v_{\vec{\varepsilon}} \in V^{\otimes n}$
, one needs to check that $\St_{V^{\otimes m}}(\alpha)(g(v_{\vec{\varepsilon}})) = \St_{V^{\otimes n}}(\alpha \circ (g\rhd {\rm Id}_\varnothing))(v_{\vec{\varepsilon}})$. Morphisms of ${\rm TL}$ are generated under composition and juxtaposition by identities, caps and cups, so one only needs to prove the result for these last two.
We will need some explicit computations of these caps and cups. Recall from Remark \ref{rmkUnorTanOq}, or \cite[Theorem~4.2]{Tingley}, that to do so one uses the isomorphism
\[
\varphi \colon\ \begin{cases}
V \to V^*,
\\
v_+ \mapsto -q^{\frac{5}{2}} v_-^*,
\\ v_- \mapsto q^{\frac{1}{2}}v_+^*\end{cases}
\]
 of Definition~\ref{defOqcomods} and sets $\cap = \lcap \circ (\varphi \otimes {\rm Id}_V)$ and $\cup = ({\rm Id}_V \otimes \varphi^{-1})\circ \lcup$, where $\lcap$ and $\lcup$ are the usual ${\rm ev}$ and ${\rm coev}$ in ${\rm Vect}_k^{\rm fin}$.

Let $g =\ $\begin{tikzpicture}[xscale = 0.18, yscale = 0.25, baseline = 0pt]
\draw (0,0) -- ++(0,1) ;
\draw (1,0) -- ++(0,1) ;
\draw (2,0) -- ++(0,1) ;
\draw (3,0) arc(180:0:0.5 and 0.75);
\draw (5,0) -- ++(0,1) ;
\draw (6,0)-- ++(0,1) ;
\draw (7,0) -- ++(0,1) ;
\end{tikzpicture}$\ = {\rm Id}_V^{\otimes k} \otimes \cap \otimes {\rm Id}_V^{\otimes n-k}\colon V^{\otimes n+2} \to V^{\otimes n}$, $\alpha \in \Hom_{{\rm Sk}_{{\rm TL}}(\Sigma)}(V^{\otimes n}\rhd \varnothing,\varnothing)$ and $v=(v_{\varepsilon_1}\otimes \cdots \otimes v_{\varepsilon_k})\otimes v_{\mu}\otimes v_{\nu} \otimes (v_{\varepsilon_{k+1}}\otimes \cdots\otimes v_{\varepsilon_{n}}) \in V^{\otimes n+2}$. We want to compare
\begin{gather*}
\St_{V^{\otimes n}}(\alpha)(g(v)) = \St_{V^{\otimes n}}(\alpha)(v_{\vec{\varepsilon}}\ .\cap(v_{\mu}\otimes v_{\nu})) =\begin{tikzpicture}[scale = 1, baseline = 8pt]
\node [rightymissy, minimum width = 1cm] (A) at (1.5,1) {$\alpha$};
\draw (0,0) -- ++(2,0) ;
\draw (0,0) -- ++ (0,1);
\draw (0,0.75) node[ left = -2pt]{\footnotesize $\varepsilon_1$} -- (A);
\draw (0,0.5) node[ above , rotate = 90]{\tiny $\cdots$} .. controls (1,0.5) .. (A);
\draw (0,0.25) node[ left = -2pt]{\footnotesize $\varepsilon_n$} .. controls (1.5,0.25) .. (A);
\end{tikzpicture}\ .\cap(v_{\mu}\otimes v_{\nu})
\end{gather*}
and
\begin{gather*}
\St_{V^{\otimes n+2}}(\alpha \circ (g\rhd {\rm Id}_\varnothing))(v)=\St_{V^{\otimes n+2}}\Bigg(\ \begin{tikzpicture}[scale = 0.7, baseline = 8pt]
\node [rightymissy, minimum width = 1cm] (A) at (2.5,1.5) {$\alpha$};
\draw (0,0) -- ++(3,0);
\draw (0,0) -- ++ (0,1.5);
\draw(0.5,0) .. controls (0.5,1.5) .. (A);
\draw(1,0) .. controls (1,1.2) .. (A);
\draw (1.5,0) arc(180:0:0.25 and 0.5);
\draw(2.5,0) -- (A);
\end{tikzpicture}
\Bigg)(v) = \begin{tikzpicture}[scale = 0.8, baseline = 8pt]
\node [rightymissy, minimum width = 1cm] (A) at (2.5,1.5) {$\alpha$};
\draw (0,0) -- ++(3,0) ;
\draw (0,0) -- ++ (0,1.5);
\draw (0,1.25) node[ left = -2pt]{\footnotesize $\varepsilon_1$} -- (A);
\draw (0,1) node[ above , rotate = 90]{\tiny $\cdots$} .. controls (1,1) .. (A);
\draw (0,0.75) node[ left = -2pt]{\footnotesize $\mu$} arc(90:-90:0.5 and 0.125) node[ left = -2pt]{\footnotesize $\nu$} ;
\draw (0,0.25) node[ left = -2pt]{\footnotesize $\varepsilon_n$} .. controls (1.5,0.25) .. (A);
\end{tikzpicture}
\\ \hphantom{\St_{V^{\otimes n+2}}(\alpha \circ (g\rhd {\rm Id}_\varnothing))(v)}
{}=\begin{tikzpicture}[scale = 1, baseline = 8pt]
\node [rightymissy, minimum width = 1cm] (A) at (1.5,1) {$\alpha$};
\draw (0,0) -- ++(2,0) ;
\draw (0,0) -- ++ (0,1);
\draw (0,0.75) node[ left = -2pt]{\footnotesize $\varepsilon_1$} -- (A);
\draw (0,0.5) node[ above , rotate = 90]{\tiny $\cdots$} .. controls (1,0.5) .. (A);
\draw (0,0.25) node[ left = -2pt]{\footnotesize $\varepsilon_n$} .. controls (1.5,0.25) .. (A);
\end{tikzpicture}. {}^{\mu}_{\nu}\text{\reflectbox{$C$}}.
\end{gather*}
One simply needs to check that the coefficient ${}^{\mu}_{\nu}\reflectbox{$C$}$ from the left boundary skein relations of Proposition~\ref{propLeftStSkRel} coincides with $\cap(v_{\mu}\otimes v_{\nu})$ and indeed $\cap(v_+ \otimes v_+) = {\rm ev}\big({-}q^{\frac{5}{2}}v_-^* \otimes v_+\big) = 0={}^+_+\reflectbox{$C$}$, $\cap(v_- \otimes v_-) = {\rm ev}\big(q^{\frac{1}{2}}v_+^* \otimes v_-\big) = 0={}^-_-\reflectbox{$C$}$, $\cap(v_+ \otimes v_-) = {\rm ev}\big({-}q^{\frac{5}{2}}v_-^* \otimes v_-\big) = -q^{\frac{5}{2}}={}^+_-\reflectbox{$C$}$ and $\cap(v_- \otimes v_+) = {\rm ev}\big(q^{\frac{1}{2}}v_+^* \otimes v_+\big) = q^{\frac{1}{2}}={}^-_+\reflectbox{$C$}$.

Now let $g =\ $\begin{tikzpicture}[xscale = 0.18, yscale = -0.25, baseline =-7pt]
\draw (0,0) -- ++(0,1) ;
\draw (1,0) -- ++(0,1) ;
\draw (2,0) -- ++(0,1) ;
\draw (3,0) arc(180:0:0.5 and 0.75);
\draw (5,0) -- ++(0,1) ;
\draw (6,0)-- ++(0,1) ;
\draw (7,0) -- ++(0,1) ;
\end{tikzpicture}$\ = {\rm Id}_V^{\otimes k} \otimes \cup \otimes {\rm Id}_V^{\otimes n-k}\colon V^{\otimes n} \to V^{\otimes n+2}$, $\alpha \in \Hom_{{\rm Sk}_{{\rm TL}}(\Sigma)}\big(V^{\otimes n+2}\rhd \varnothing,\varnothing\big)$ and~$v=v_{\varepsilon_1}\otimes \cdots \otimes v_{\varepsilon_{n}} \in V^{\otimes n}$. One can directly compute $\cup (1) = \big({\rm Id}_V \otimes \varphi^{-1}\big)\circ
\mathop{\rm coev}(1) = \big({\rm Id}_V \otimes \varphi^{-1}\big) (v_- \otimes v_-^* + v_+ \otimes v_+^*) = -q^{-\frac{5}{2}} v_-\otimes v_+ + q^{-\frac{1}{2}} v_+ \otimes v_-$.
 We want to compare
\begin{gather*}
{\rm St}_{V^{\otimes n+2}}(\alpha)(g(v)) = {\rm St}_{V^{\otimes n+2}}(\alpha)((v_{\varepsilon_1} \otimes\cdots v_{\varepsilon_k}) \otimes  \big({-}q^{-\frac{5}{2}} v_- \otimes  v_+  +  q^{-\frac{1}{2}} v_+  \otimes   v_-\big) \\
\hphantom{\St_{V^{\otimes n+2}}(\alpha)(g(v))=}
\otimes (v_{\varepsilon_{k+1}} \otimes  \cdots v_{\varepsilon_{n}}))
\\  \hphantom{\St_{V^{\otimes n+2}}(\alpha)(g(v))}
{}=-q^{-\frac{5}{2}}\begin{tikzpicture}[scale = 1, baseline = 8pt]
\node [rightymissy, minimum width = 1cm] (A) at (2.5,1.5) {$\alpha$};
\draw (0,0) -- ++(3,0) ;
\draw (0,0) -- ++ (0,1.5);
\draw (0,1.25) node[ left = -2pt]{\footnotesize $\varepsilon_1$} -- (A);
\draw (0,1) node[ above , rotate = 90]{\tiny $\cdots$} .. controls (1,1) .. (A);
\draw (0,0.75) node[ left = -2pt]{\footnotesize $-$}.. controls (1,0.75) .. (A);
\draw (0,0.5) node[ left = -2pt]{\footnotesize $+$} .. controls (1,0.5) .. (A);
\draw (0,0.25) node[ left = -2pt]{\footnotesize $\varepsilon_n$} .. controls (1.5,0.25) .. (A);
\end{tikzpicture}+q^{-\frac{1}{2}}\begin{tikzpicture}[scale = 1, baseline = 8pt]
\node [rightymissy, minimum width = 1cm] (A) at (2.5,1.5) {$\alpha$};
\draw (0,0) -- ++(3,0) ;
\draw (0,0) -- ++ (0,1.5);
\draw (0,1.25) node[ left = -2pt]{\footnotesize $\varepsilon_1$} -- (A);
\draw (0,1) node[ above , rotate = 90]{\tiny $\cdots$} .. controls (1,1) .. (A);
\draw (0,0.75) node[ left = -2pt]{\footnotesize $+$}.. controls (1,0.75) .. (A);
\draw (0,0.5) node[ left = -2pt]{\footnotesize $-$} .. controls (1,0.5) .. (A);
\draw (0,0.25) node[ left = -2pt]{\footnotesize $\varepsilon_n$} .. controls (1.5,0.25) .. (A);
\end{tikzpicture}
\end{gather*}
and
\begin{gather*}
\St_{V^{\otimes n}}(\alpha \circ (g\rhd {\rm Id}_\varnothing))(v)=\St_{V^{\otimes n}}\Bigg(\ \begin{tikzpicture}[yscale = 0.7, baseline = 8pt]
\node [rightymissy, minimum width = 1cm] (A) at (2.5,1.5) {$\alpha$};
\draw (0,0) -- ++(3,0);
\draw (0,0) -- ++ (0,1.5);
\draw(0.5,0) .. controls (0.5,1.5) .. (A);
\draw(1,0) .. controls (1,1.2) .. (A);
\draw(1.5,0.75) .. controls (1.5,1.1) .. (A);
\draw (1.5,0.75) arc(180:0:0.25 and -0.5) .. controls (2,0.9)..(A);
\draw(2.5,0) -- (A);
\end{tikzpicture}
\Bigg)(v) = \begin{tikzpicture}[scale = 1, baseline = 8pt]
\node [rightymissy, minimum width = 1cm] (A) at (2.5,1.5) {$\alpha$};
\draw (0,0) -- ++(3,0) ;
\draw (0,0) -- ++ (0,1.5);
\draw (0,1.25) node[ left = -2pt]{\footnotesize $\varepsilon_1$} -- (A);
\draw (0,1) node[ above , rotate = 90]{\tiny $\cdots$} .. controls (1,1) .. (A);
\draw (0.75,0.75) .. controls (1,0.75) .. (A);
\draw (0.75,0.75) arc(90:270:0.25 and 0.125) .. controls (1,0.5) .. (A);
\draw (0,0.25) node[ left = -2pt]{\footnotesize $\varepsilon_n$} .. controls (1.5,0.25) .. (A);
\end{tikzpicture}.
\end{gather*}
They are equal by the last left boundary skein relation of Proposition~\ref{propLeftStSkRel}.
\end{proof}

Recall from Proposition~\ref{propIntHomFree} and the discussion below that the internal endomorphism object of~$\varnothing$ with respect to the $\mathcal{E}$-action is an object $X$ equipped with an isomorphism $\Hom_{{\rm Sk}_{\rm TL}(\Sigma)}(-\rhd \varnothing,\varnothing)\allowbreak \tilde{\Rightarrow} \Hom_{\mathcal{E}}(-, X)$ in $[{\rm TL}^{\rm op},{\rm Vect}_k]$

\begin{Theorem}\label{thmRelation}
The natural transformation $\St$ is a natural isomorphism and exhibits $\mathscr{S}(\Sigma)$ as the internal endomorphism object of the empty set. Namely one can take $A_\Sigma = \mathscr{S}(\Sigma)$ as $\Oq$-comodule.
\end{Theorem}

\begin{proof}
We exhibit an inverse to $\St$. Let $W=V^{\otimes n}\in {\rm TL}$, which decomposes as a direct sum of simples $W_i$. Let $f\colon W\to \mathscr{S}(\Sigma)$ be a morphism in $\mathcal{E}$, and denote by $f_i$ its restriction on $W_i$. We~want a morphism $\St_W^{-1}(f) \in \Hom_{{\rm Sk}_{{\rm TL}}(\Sigma)}(W\rhd \varnothing,\varnothing)$, which is equivalent to a collection of morphisms ${\rm St}_W^{-1}(f_i) \in \Hom_{{\rm Sk}_{\mathcal{V}}(\Sigma)}(W_i\rhd \varnothing,\varnothing)$. Each $f_i$ is determined by its value on a single element $w_i\in W_i$ (pick a highest weight element for example), it extends to all $W_i$ by applying the $\Uq$-action (recall from Proposition~\ref{propOqcomodUqmod} that $\Oq$-comodules correspond to $\Uq$-modules). Choose any stated tangle~$a_i$ representing the element $f_i(w_i) \in \mathscr{S}(\Sigma)$. Denote by $\alpha_i$ its underlying tangle, $n_i=\abss{\partial \alpha_i}$ its number of boundary points and $\vec{\varepsilon}_i$ its states. This representative is well defined up to the boundary skein relations, the usual skein relations and isotopy. The assignment $w_i \mapsto v_{\vec{\varepsilon}_i}$ extends to a unique $ \Oq$-morphism $g_i\colon W_i \to V^{\otimes n_i}$ by applying the $\Uq$-action. We then set $\St_W^{-1}(f_i) = \alpha_i\circ (g_i \rhd {\rm Id}_\varnothing)$. Note that here~$\alpha_i$ denotes the tangle~$\alpha_i$ seen as a morphism in ${\rm Sk}_\mathcal{V}(\Sigma)$, so we actually mean $\psi_{hv}^{-1}(\alpha_i)$, the same tangle but with boundary points at the bottom instead of the left. We have to check that this definition does not depend on the representative $a_i$. Usual skein relations and isotopy do not change~$\alpha_i$ seen as a~morphism in the skein category. The boundary skein relations are equivalent to naturality using Proposition~\ref{propStNat}. Namely, another representative $a'_i$ has to be of the form $\alpha'_i= \alpha_i \circ (g \rhd {\rm Id}_\varnothing)$, for some $g \in {\rm TL}$, with states ${\vec{\varepsilon}_i}^{\ '}$ such that $g({\vec{\varepsilon}_i}^{\ '})=\vec{\varepsilon}_i$. Therefore the $g'_i$ for $a'_i$ will be such that $g_i = g \circ g'_i$. Then we simply check that $\alpha_i\circ (g_i \rhd {\rm Id}_\varnothing) = \alpha_i\circ (g \rhd {\rm Id}_\varnothing) \circ (g'_i \rhd {\rm Id}_\varnothing) = \alpha'_i\circ (g'_i \rhd {\rm Id}_\varnothing)$, and ${\rm St}_W^{-1}(f_i)$ is well defined.

For simplicity we assume that $\St_W^{-1}(f_i)$ and $g_i$ are actually defined on all $W$ but are 0 except on~$W_i$, namely we precompose by the projection $W \twoheadrightarrow W_i$, and thus we set $\St_W^{-1}(f) = \sum_i \St_W^{-1}(f_i)$.

It is now easy to check that $\St_W^{-1}$ is the inverse of $\St_W$. For $v_{\vec{\varepsilon}} \in W$, suppose $v_{\vec{\varepsilon}}\in W_{i_0} \subseteq W$ lies in a simple, or decompose it in the $W_i$'s, and write $v_{\vec{\varepsilon}} = X\cdot w_{i_0}$, $X \in \Uq$. One has $g_{i_0}(v_{\vec{\varepsilon}})= X\cdot v_{\vec{\varepsilon}_{i_0}}$ and for $j\neq {i_0},\ g_j(v_{\vec{\varepsilon}})=0$. Thus:
\begin{align*}
\St_W\big(\St_W^{-1}(f)\big)(v_{\vec{\varepsilon}}) &= \St_W\bigg(\sum_i \alpha_i\circ (g_i \rhd {\rm Id}_\varnothing)\bigg)(v_{\vec{\varepsilon}}) = \sum_i \St_W(\alpha_i\circ (g_i \rhd {\rm Id}_\varnothing))(v_{\vec{\varepsilon}})
\\
&\overset{{\rm Proposition~\ref{propStNat}}}{=} \sum_i \St_W(\alpha_i)(g_i(v_{\vec{\varepsilon}}))
 = \St_W(\alpha_{i_0})(X\cdot v_{\vec{\varepsilon}_{i_0}})
 \\&
\overset{{\rm Proposition~\ref{propStOqmorph}}}{=} X\cdot \St_W(\alpha_{i_0})(v_{\vec{\varepsilon}_{i_0}})= X\cdot a_{i_0}= X \cdot f_{i_0}(w_{i_0})\\
& = f_{i_0}(v_{\vec{\varepsilon}})= f(v_{\vec{\varepsilon}}).
\end{align*}
Symmetrically, let $\alpha \in \Hom_{{\rm Sk}_{{\rm TL}}(\Sigma)}(W \rhd \varnothing,\varnothing)$ and $v_{\vec{\varepsilon}} \in W$. Set $f=\St_W(\alpha)\colon V^{\otimes \abss{\partial \alpha}} \to \mathscr{S}(\Sigma)$, in the definition of $\St_W^{-1}(f)$ one has $\alpha_i = \alpha$ and $v_{\vec{\varepsilon}_i}=w_i$, so $g_i$ is the inclusion $W_i\hookrightarrow W$. If $v_{\vec{\varepsilon}} \in W_i \subseteq W$, one has $\St_W^{-1}(\St_W(\alpha))=\sum_i \alpha\circ (g_i\rhd {\rm Id}_\varnothing) = \alpha\circ ({\rm Id}_W \rhd {\rm Id}_\varnothing) = \alpha$.
\end{proof}

\begin{Proposition}\label{propProdsOnSS}
The algebra structure inherited from the internal endomorphism object structure on~$\mathscr{S}(\Sigma)$ coincides with its usual algebra structure. Namely $A_\Sigma =\mathscr{S}(\Sigma)$ as $\Oq$-comodule algebras.
\end{Proposition}

\begin{proof}
Recall that the product on $\mathscr{S}(\Sigma)$ is given by stacking the left tangle $a$ above the right one~$b$, and the product on $A_\Sigma$ is defined by evaluation and composition maps on internal $\Hom$ objects, in~Definition~\ref{algStrOnIntHom}. Graphically, the stated tangles $a$ and $b$, which we see as morphisms $\alpha$ and $\beta$ in~${\rm Sk}_\mathcal{V}(\Sigma)$ with prescribed inputs $v_{\vec{\varepsilon}}$ and $v_{\vec{\eta}}$, map to the morphism $\alpha \circ ({\rm Id}_{V^{\otimes \abss{\partial \alpha}}}\rhd \beta)$ in~${\rm Sk}_\mathcal{V}(\Sigma)$ with prescribed input $v_{\vec{\varepsilon}} \otimes v_{\vec{\eta}}$, which we see as the stated tangle $a$ above $b$:
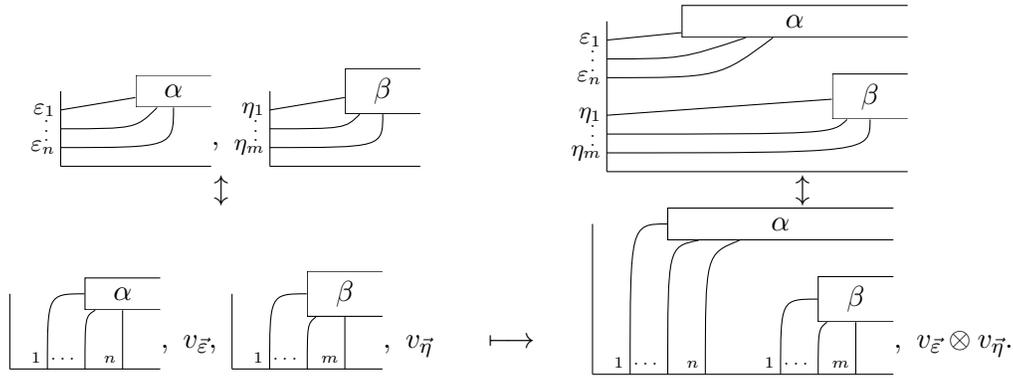
\[
\begin{array}{ccc}
\begin{tikzpicture}[scale = 1, baseline = 8pt]
\node [rightymissy, minimum width = 1cm] (A) at (1.5,1) {$\alpha$};
\draw (0,0) -- ++(2,0) ;
\draw (0,0) -- ++ (0,1);
\draw (0,0.75) node[ left = -2pt]{\footnotesize $\varepsilon_1$} -- (A);
\draw (0,0.5) node[ above , rotate = 90]{\tiny $\cdots$} .. controls (1,0.5) .. (A);
\draw (0,0.25) node[ left = -2pt]{\footnotesize $\varepsilon_n$} .. controls (1.5,0.25) .. (A);
\end{tikzpicture} , \begin{tikzpicture}[scale = 1, baseline = 8pt]
\node [rightymissy, minimum width = 1cm] (A) at (1.5,1) {$\beta$};
\draw (0,0) -- ++(2,0) ;
\draw (0,0) -- ++ (0,1);
\draw (0,0.75) node[ left = -2pt]{\footnotesize $\eta_1$} -- (A);
\draw (0,0.5) node[ above , rotate = 90]{\tiny $\cdots$} .. controls (1,0.5) .. (A);
\draw (0,0.25) node[ left = -2pt]{\footnotesize $\eta_m$} .. controls (1.5,0.25) .. (A);
\end{tikzpicture}
&
&
\begin{tikzpicture}[baseline = 10pt]
\draw (0,0) -- ++(4,0) ;
\draw (0,0) -- ++ (0,2);
\begin{scope}[yshift = 1cm]
\node [rightymissy, minimum width = 3cm] (A) at (2.5,1) {$\alpha$};
\draw (0,0.75) node[ left = -2pt]{\footnotesize $\varepsilon_1$} -- (A);
\draw (0,0.5) node[ above , rotate = 90]{\tiny $\cdots$} .. controls (1,0.5) .. (A);
\draw (0,0.25) node[ left = -2pt]{\footnotesize $\varepsilon_n$} .. controls (1.5,0.25) .. (A);
\end{scope}
\begin{scope}[xshift = 2cm]
\node [rightymissy, minimum width = 1cm] (A) at (1.5,1) {$\beta$};
\draw (-2,0.75) node[ left = -2pt]{\footnotesize $\eta_1$} -- (A);
\draw (-2,0.5) node[ above , rotate = 90]{\tiny $\cdots$} .. controls (1,0.5) .. (A);
\draw (-2,0.25) node[ left = -2pt]{\footnotesize $\eta_m$} .. controls (1.5,0.25) .. (A);
\end{scope}
\end{tikzpicture} \hspace{1.8cm}
\\
\updownarrow & & \updownarrow \\
\begin{tikzpicture}[scale = 1, baseline = 8pt]
\node [rightymissy, minimum width = 1cm] (A) at (1.5,1) {$\alpha$};
\draw (0,0) -- ++(2,0);
\draw (0,0) -- ++ (0,1);
\draw (0.5,0) node[above left = -2pt]{\tiny 1} .. controls (0.5,1) .. (A);
\draw (1,0) node[above left = -2pt]{\tiny $\cdots$} .. controls (1,0.7) .. (A);
\draw (1.5,0) node[above left = -2pt]{\tiny $n$} -- (A);
\end{tikzpicture},\ v_{\vec{\varepsilon}} , \
\begin{tikzpicture}[scale = 1, baseline = 8pt]
\node [rightymissy, minimum width = 1cm] (A) at (1.5,1) {$\beta$};
\draw (0,0) -- ++(2,0);
\draw (0,0) -- ++ (0,1);
\draw (0.5,0) node[above left = -2pt]{\tiny 1} .. controls (0.5,1) .. (A);
\draw (1,0) node[above left = -2pt]{\tiny $\cdots$} .. controls (1,0.6) .. (A);
\draw (1.5,0) node[above left = -2pt]{\tiny $m$} -- (A);
\end{tikzpicture},\ v_{\vec{\eta}}
& \quad \longmapsto
&
\begin{tikzpicture}[baseline = 10pt]
\draw (0,0) -- ++(4,0) ;
\draw (0,0) -- ++ (0,2);
\begin{scope}[yshift = 1cm]
\node [rightymissy, minimum width = 3cm] (A) at (2.5,1) {$\alpha$};
\draw (0.5,-1) node[above left = -2pt]{\tiny 1} .. controls (0.5,1) .. (A);
\draw (1,-1) node[above left = -2pt]{\tiny $\cdots$} .. controls (1,0.7) .. (A);
\draw (1.5,-1) node[above left = -2pt]{\tiny $n$} .. controls (1.5,0.6) .. (A);
\end{scope}
\begin{scope}[xshift = 2cm]
\node [rightymissy, minimum width = 1cm] (A) at (1.5,1) {$\beta$};
\draw (0.5,0) node[above left = -2pt]{\tiny 1} .. controls (0.5,1) .. (A);
\draw (1,0) node[above left = -2pt]{\tiny $\cdots$} .. controls (1,0.6) .. (A);
\draw (1.5,0) node[above left = -2pt]{\tiny $m$} -- (A);
\end{scope}
\end{tikzpicture} ,\ v_{\vec{\varepsilon}} \otimes v_{\vec{\eta}}.
\end{array}
\]
The evaluation map ${\rm ev}_{\varnothing,\varnothing}\colon \mathscr{S}(\Sigma)\rhd \varnothing\to\varnothing$ is the image under $\St^{-1}$ of ${\rm Id}_{\mathscr{S}(\Sigma)}$. We have not constructed~$\St^{-1}$ on all $\mathcal{E}$ above, but only on ${\rm TL}$, and it extends by cocontinuity in Proposition~ \ref{propIntHomFree}. The comodule~$\mathscr{S}(\Sigma)$ decomposes as simples as $\mathscr{S}(\Sigma) = \bigoplus_{\alpha \in \mathcal{B}^+} \Uq\cdot \alpha$ where $\mathcal{B}^+$ is the set of $\mathfrak{o}$-ordered simple stated tangles with only + states, see \cite[Theorem~4.6(b)]{Cos}. These stated tangles with only + states are simply a way to represent canonically a tangle without state information, and again in the following we will see $\alpha$ as a morphism in $\Hom_{{\rm Sk}_\mathcal{V}(\Sigma)}\big(V^{\otimes \abss{\partial \alpha}},\varnothing\big)$, which is actually $\psi_{hv}^{-1}(\alpha)$.

 We denote by $W_\alpha = \Uq\cdot \alpha$ and $g_\alpha\colon W_\alpha \to V^{\otimes \abss{\partial \alpha}}$ the inclusion mapping $\alpha$ to its states $v_{\overrightarrow{+\cdots +}}$. Then:
 \[
 {\rm ev}_{\varnothing,\varnothing} =\St^{-1}({\rm Id}_{\mathscr{S}(\Sigma)}) = \oplus_{\alpha \in \mathcal{B}^+} \St^{-1}(W_\alpha \hookrightarrow \mathscr{S}(\Sigma)) = \oplus_{\alpha \in \mathcal{B}^+}\alpha \circ (g_\alpha \rhd {\rm Id}_\varnothing).
 \]
The composition map $c \colon \mathscr{S}(\Sigma) \otimes \mathscr{S}(\Sigma) \to \mathscr{S}(\Sigma)$ is the image under $St$ of the morphism ${\rm ev}_{\varnothing,\varnothing} \circ ({\rm Id}_{\mathscr{S}(\Sigma)} \rhd {\rm ev}_{\varnothing,\varnothing}) \colon (\mathscr{S}(\Sigma) \otimes \mathscr{S}(\Sigma))\rhd \varnothing \to \mathscr{S}(\Sigma) \rhd \varnothing\to\varnothing$. This morphism is the double sum:
\begin{gather*}
\oplus_{\alpha \in \mathcal{B}^+}\oplus_{\beta \in \mathcal{B}^+}(\alpha \circ (g_\alpha \rhd {\rm Id}_\varnothing))\circ ({\rm Id}_{\mathscr{S}(\Sigma)} \rhd (\beta \circ (g_\beta \rhd {\rm Id}_\varnothing)))
\\ \qquad
{}=\oplus_{\alpha \in \mathcal{B}^+}\oplus_{\beta \in \mathcal{B}^+}\alpha \circ ({\rm Id}_{V^{\otimes \abss{\partial \alpha}}}\rhd \beta) \circ (g_\alpha \rhd g_\beta \rhd {\rm Id}_\varnothing) .
\end{gather*}
The product is obtained by applying $\St$ to this morphism. For $a,b \in \mathscr{S}(\Sigma)$, write $a=X\cdot \alpha$ and $b=Y\cdot \beta$ with $X,Y \in \Uq$ and $\alpha,\beta \in \mathcal{B}^+$. Thus $a$ has states $v_{\vec{\varepsilon}} = X \cdot v_{\overrightarrow{+\cdots +}}$ and $b$ has states $v_{\vec{\eta}}=Y \cdot v_{\overrightarrow{+\cdots +}}$. By naturality:
\begin{align*}
c(a\otimes b) &:= \St_{\mathscr{S}(\Sigma)\otimes \mathscr{S}(\Sigma)}({\rm ev}_{\varnothing,\varnothing} \circ ({\rm Id}_{\mathscr{S}(\Sigma)} \rhd {\rm ev}_{\varnothing,\varnothing}))(a\otimes b)
\\
&= \oplus_{\alpha' \in \mathcal{B}^+}\oplus_{\beta' \in \mathcal{B}^+} \St(\alpha' \circ ({\rm Id}_{V^{\otimes \abss{\partial \alpha'}}}\rhd \beta'))((g_{\alpha'} \otimes g_{\beta'})(a \otimes b))
\\
& = \St(\alpha \circ ({\rm Id}_{V^{\otimes \abss{\partial \alpha}}}\rhd \beta))((g_\alpha \otimes g_\beta)(a \otimes b)) = \St(\alpha \circ ({\rm Id}_{V^{\otimes \abss{\partial \alpha}}}\rhd \beta))(v_{\vec{\varepsilon}} \otimes v_{\vec{\eta}})
\end{align*}
which is precisely the usual product of $a$ and $b$ in $\mathscr{S}(\Sigma)$.
\end{proof}

\section{Multi-edges}\label{SectMultiEdges}
We define internal skein algebras for surfaces with more than one boundary, and possibly left or right boundary edges. We show they are isomorphic to stated skein algebras when $\mathcal{V} = \Oqcomfin$, and re-prove their excision properties using excision properties of skein categories.

\subsection{Right internal skein algebras}
In order to extend the definition of internal skein algebras to surfaces with multiple boundary edges, we would need a notion of left and right action to be able to glue surfaces together, such that internal skein algebras satisfy excision properties, just like stated skein algebras. One subtlety though is that one is only allowed to talk about right $\Oq$-comodules in the context of internal skein algebras, to stay in the category $\mathcal{E}$ (as opposed to the stated skein algebra context).

\begin{Definition}[{\cite[Section~3.2]{Cooke}}]
Let $\Sigma$ be a surface with a chosen boundary interval, which we see at the right of the surface. One can make a construction similar to Definition~\ref{defLeftActSkR2} to have a right action functor $\lhd \colon {\rm Sk}_\mathcal{V}(\Sigma) \otimes {\rm Sk}_\mathcal{V}\big(\mathbb{R}^2\big)\to {\rm Sk}_\mathcal{V}(\Sigma)$. It differs the one $\rhd$ of Definition~\ref{defLeftActSkR2} only by rotating the disk by 180 degrees.

The right moduli algebra $A_\Sigma^R$ of \cite[Section~5.2]{BBJ}, or right internal skein algebra of \cite{Gunningham}, is the internal endomorphism algebra of the empty set in ${\rm Sk}_\mathcal{V}(\Sigma)$ with respect to this ${\rm Sk}_\mathcal{V}\big(\mathbb{R}^2\big)^{\otimes\text{-}{\rm op}}$-module structure.
\end{Definition}
$$
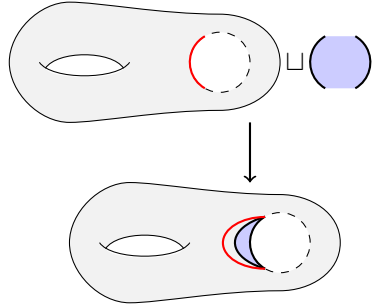
\begin{tikzpicture}[yscale = 0.8, xscale = 0.8]
\fill[gray!10] (4.5,1) .. controls (4.5,0.5) and (4,0.2) .. (3.5,0.2)	.. controls (2.5,0.2) and (1.5,0 ) .. (1,0)	.. controls (0.5,0 ) and (0,0.5 ) .. (0,1)	.. controls (0,1.5 ) and (0.5,2 ) .. (1,2)	.. controls (1.5,2 ) and (2.5,1.8) .. (3.5,1.8) .. controls (4,1.8) and (4.5,1.5) .. (4.5,1) -- cycle ;
\fill[white] (0.6,0.9) .. controls (0.7,0.7) and (1.8,0.7) .. (1.9,0.9) .. controls (1.7,1.2) and (0.8,1.2) .. (0.6,0.9);
\fill[white] (3.5,1) circle(0.5);
\draw[dashed] (3.5,1) circle(0.5);
\draw[thick, red] (240:0.5)++(3.5,1) arc(240:120:0.5);
\draw (0.5,1) .. controls (0.7,0.7) and (1.8,0.7) .. (2,1);
\draw (0.6,0.9) .. controls (0.8,1.2) and (1.7,1.2) .. (1.9,0.9) node[midway, above right]{};
\draw (4.5,1) .. controls (4.5,0.5) and (4,0.2) .. (3.5,0.2)	.. controls (2.5,0.2) and (1.5,0 ) .. (1,0)	.. controls (0.5,0) and (0,0.5 ) .. (0,1)	.. controls (0,1.5 ) and (0.5,2 ) .. (1,2)	.. controls (1.5,2) and (2.5,1.8) .. (3.5,1.8) .. controls (4,1.8) and (4.5,1.5) .. (4.5,1);
\node (sq) at (4.75,1){$\sqcup$}; \draw[->, thick] (4,0) -- ++(0,-1) node[midway, right]{};
\fill[blue!20](240:0.5)++(5.5,1) arc(240:120:0.5) -- ++(0.5,0) arc(60:-60:0.5) -- ++(-0.5,0);
\draw[thick] (240:0.5)++(5.5,1) arc(240:120:0.5);
\draw[thick] (-60:0.5)++(5.5,1) arc(-60:60:0.5);
\begin{scope}[yshift = -3cm, xshift = 1cm]
\fill[gray!10] (4.5,1) .. controls (4.5,0.5) and (4,0.2) .. (3.5,0.2)	.. controls (2.5,0.2) and (1.5,0 ) .. (1,0)	.. controls (0.5,0 ) and (0,0.5 ) .. (0,1)	.. controls (0,1.5 ) and (0.5,2 ) .. (1,2)	.. controls (1.5,2 ) and (2.5,1.8) .. (3.5,1.8) .. controls (4,1.8) and (4.5,1.5) .. (4.5,1) -- cycle ;
\fill[white] (0.6,0.9) .. controls (0.7,0.7) and (1.8,0.7) .. (1.9,0.9) .. controls (1.7,1.2) and (0.8,1.2) .. (0.6,0.9);
\fill[white] (3.5,1) circle(0.5);
\fill[blue!20] (240:0.5)++(3.5,1) arc(240:120:0.5) arc(120:240:1 and 0.5);
\draw[dashed] (3.5,1) circle(0.5);
\draw[thick] (240:0.5)++(3.5,1) arc(240:120:0.5);
\draw[thick] (240:0.5)++(3.5,1) arc(240:120:1 and 0.5);
\draw[thick, red] (240:0.5)++(3.5,1) arc(270:90:0.7 and 0.43);
\draw (0.5,1) .. controls (0.7,0.7) and (1.8,0.7) .. (2,1);
\draw (0.6,0.9) .. controls (0.8,1.2) and (1.7,1.2) .. (1.9,0.9) node[midway, above right]{};
\draw (4.5,1) .. controls (4.5,0.5) and (4,0.2) .. (3.5,0.2)	.. controls (2.5,0.2) and (1.5,0 ) .. (1,0)	.. controls (0.5,0 ) and (0,0.5 ) .. (0,1)	.. controls (0,1.5 ) and (0.5,2 ) .. (1,2)	.. controls (1.5,2 ) and (2.5,1.8) .. (3.5,1.8) .. controls (4,1.8) and (4.5,1.5) .. (4.5,1);
\end{scope}
\end{tikzpicture}
$$
\begin{Definition}
Denote by $rot$ the diffeomorphism of the disk given by $180^\circ$ rotation. By Remark \ref{rmkFunctEmbIsotOfSkCat} it induces an automorphism $(-)^\halft := {\rm Sk}_\mathcal{V}({\rm rot})\colon {\rm Sk}_\mathcal{V}\big(\mathbb{R}^2\big) \to {\rm Sk}_\mathcal{V}\big(\mathbb{R}^2\big)$ which squares to the identity. For $X \in {\rm Sk}_\mathcal{V}\big(\mathbb{R}^2\big)$, we call $X^\halft := {\rm Sk}_\mathcal{V}({\rm rot})(X)$ the half-twisted $X$. One easily checks that $(-)^\halft$ is anti-monoidal, namely $(X \otimes Y)^\halft = Y^\halft \otimes X^\halft$, because $rot$ reverses left-right order.
The diffeomorphism $rot$ is isotopic to the identity by rotating from 0 to $180^\circ$. This isotopy induces a natural isomorphism $\halft\colon {\rm Id}_{{\rm Sk}_\mathcal{V}(\mathbb{R}^2)} \tilde{\Rightarrow} (-)^\halft$ called the half twist, which squares to the twist (the $360^\circ$ rotation).
\end{Definition}

Remember from Remark \ref{rmkFunctEmbIsotOfSkCat} that $\halft$ is given on $n$ blackboard framed points on the real axis by~$n$ parallelly half-twisted vertical strands, namely drawn on the half twisted ribbon \halftwist. Naturality, namely $\halft_W \circ f \circ \halft_V^{-1} = f^\halft$, expresses the fact that one can untwist a top half twist and a bottom anti-half-twist by half twisting the middle. For $f= \halft_V$ one gets $\halft_{V^\halft} = \halft_V^\halft$. Note too that
\[
\halft_{V\otimes W} = (\halft_W \otimes \halft_V) \circ c_{V,W}\qquad\text{by}\quad
\begin{tikzpicture}[scale = 0.5, baseline = 10pt]
\draw[thick] (-0.1,0) .. controls (-0.1,1) and (0.1,1) .. (0.1,2);\draw[thick] (1,0) .. controls (1,1) and (-1,1) .. (-1,2) node[fill, color= white, pos = 0.5, circle, scale = 0.6]{};
\draw[thick] (0.1,0) .. controls (0.1,1) and (-0.1,1) .. (-0.1,2) node[fill, color= white, pos = 0.5, circle, scale = 0.3]{};
\fill[gray!20] (-1,0) .. controls (-1,1) and (1,1) ..(1,2)-- (0.1,2) .. controls (0.1,1) and (-0.1,1) .. (-0.1,0) -- (-1,0);
\fill[gray!70] (1,0) .. controls (1,1) and (-1,1) ..(-1,2)-- (-0.1,2) .. controls (-0.1,1) and (0.1,1) .. (0.1,0) -- (1,0);
\draw (-1,0) .. controls (-1,1) and (1,1) .. (1,2) ;
\end{tikzpicture} = \begin{tikzpicture}[scale = 0.5, baseline = 10pt]
\begin{scope}[yscale = 0.5]
\draw[thick] (1,0) .. controls (1,1) and (-0.1,1) ..(-0.1,2)node[fill, color= white, pos = 0.5, circle, scale = 0.3]{};
\draw[thick] (-1,2).. controls (-1,1) and (0.1,1) .. (0.1,0)node[fill, color= white, pos = 0.5, circle, scale = 0.3]{};
\fill[gray!70] (1,0) .. controls (1,1) and (-0.1,1) ..(-0.1,2)-- (-1,2) .. controls (-1,1) and (0.1,1) .. (0.1,0) -- (1,0);
\draw[thick] (-1,0) .. controls (-1,1) and (0.1,1) .. (0.1,2);
\draw[thick] (1,2) .. controls (1,1) and (-0.1,1) .. (-0.1,0);
\fill[gray!20] (-1,0) .. controls (-1,1) and (0.1,1) .. (0.1,2)-- (1,2) .. controls (1,1) and (-0.1,1) .. (-0.1,0) -- (-1,0);
\end{scope}
\begin{scope}[yshift = 0.99cm, xshift = 0.1cm, xscale = 0.9, yscale = 0.5]
\draw[thick] (1,0) .. controls (1,1) and (0,1) .. (0,2) node[fill, color= white, pos = 0.5, circle, scale = 0.5]{};
\fill[gray!20] (0,0) .. controls (0,1) and (1,1) ..(1,2)-- (0,2) .. controls (0,1) and (1,1) .. (1,0) -- (0,0);
\draw (0,0) .. controls (0,1) and (1,1) .. (1,2);
\end{scope}
\begin{scope}[yshift = 0.99cm, xshift = -1cm, xscale = 0.9, yscale = 0.5]
\draw[thick] (1,0) .. controls (1,1) and (0,1) .. (0,2) node[fill, color= white, pos = 0.5, circle, scale = 0.5]{};
\fill[gray!70] (0,0) .. controls (0,1) and (1,1) ..(1,2)-- (0,2) .. controls (0,1) and (1,1) .. (1,0) -- (0,0);
\draw (0,0) .. controls (0,1) and (1,1) .. (1,2);
\end{scope}
\end{tikzpicture}\,.
\]
Now, one simply has $\rhd = \lhd \circ \operatorname{fl} \circ ((-)^\halft\otimes {\rm Id}_{{\rm Sk}_\mathcal{V}(\Sigma)})$ and $\lhd \circ \operatorname{fl} = \rhd \circ ((-)^\halft\otimes {\rm Id}_{{\rm Sk}_\mathcal{V}(\Sigma)})$, where~$\operatorname{fl}$ is the flip of tensors. And indeed a left action turns into a right action under an anti-monoidal functor. Moreover, the natural isomorphism $\halft \colon {\rm Id}_{{\rm Sk}_\mathcal{V}(\mathbb{R}^2)} \tilde{\Rightarrow} (-)^\halft$ gives a natural isomorphism $\halft \rhd - := \rhd \circ (\halft\otimes {\rm Id}_{{\rm Sk}_\mathcal{V}(\Sigma)})\colon \rhd \tilde{\Rightarrow} \lhd\circ \operatorname{fl}$.

\begin{Remark}\label{rmkASigmaRnaive}
One can give an explicit relation between left and right internal skein algebras. Internal skein algebras are only defined up to isomorphism, so this description is just one choice. Actually, we will give another one below. Consider the internal skein algebra~$A_\Sigma$ defined as in Section~\ref{SectASigma} by seeing the red arc at the left, with a natural isomorphism $\sigma \colon \Hom_{{\rm Sk}_\mathcal{V}(\Sigma)}(-\rhd \varnothing, \varnothing) \tilde{\Rightarrow} \Hom_{\mathcal{E}}(-, A_\Sigma)$. Then one has natural isomorphisms
\begin{gather*}
\Hom_{{\rm Sk}_\mathcal{V}(\Sigma)}(\varnothing \lhd V,\varnothing) = \Hom_{{\rm Sk}_\mathcal{V}(\Sigma)}\!\big(V^\halft\rhd\varnothing,\varnothing\big) \overset{\sigma_{V^\halft}}\to \Hom_\mathcal{E}\!\big(V^\halft,A_\Sigma\big) \overset{(-)^\halft}\to \Hom_\mathcal{E}\!\big(V,A_\Sigma^\halft\big).
\end{gather*}
Namely, $A_\Sigma^R \simeq A_\Sigma^\halft$ as object of ${\rm SK}_\mathcal{V}\big(\mathbb{R}^2\big)$, but with natural isomorphism $(\sigma_{(-)^\halft})^\halft$.
\end{Remark}

\begin{Remark}
Note that the half twist is defined on ${\rm Sk}_\mathcal{V}\big(\mathbb{R}^2\big)$ and not on $\mathcal{V}$, which is its full subcategory of objects with only one coloured point, with blackboard framing. The half twist will map such an object to a point with anti-blackboard framing, so it does not stabilise $\mathcal{V}$. The equivalence of categories ${\rm Sk}_\mathcal{V}\big(\mathbb{R}^2\big)\simeq \mathcal{V}$ preserves properties of the half twist only up to natural isomorphism, and depends on the choice of a quasi-inverse of the inclusion. The one described in Remark \ref{rmkEqCatSkR2VV} will map the half twist on ${\rm Sk}_\mathcal{V}\big(\mathbb{R}^2\big)$ to the identity on $\mathcal{V}$, but if we had chosen to restore the framing counter-clockwise it would map it to the full twist. In $\mathcal{V} = \Oq\text{-}{\rm comod}^{\rm fin}$ an actual half twist exists, see~\cite{SnyderTingley}, and we will study it below. However, in general we will prefer a construction that uses the half twist only on the disk inserted in ${\rm Sk}_\mathcal{V}(\Sigma)$, where it is well defined, and not on $\mathcal{V}$. In particular, the above description $A_\Sigma^R \simeq A_\Sigma^\halft$ holds in ${\rm SK}_\mathcal{V}\big(\mathbb{R}^2\big)$ but has an unclear meaning in $\mathcal{E}$ (it depends on the choice of some quasi-inverses).
\end{Remark}
\begin{Definition}
There is another explicit relation between left and right internal skein algebras using $\halft\rhd -$ to relate the right action to the left one. Set $\sigma^R =\sigma \circ (\halft \rhd \varnothing)\colon \Hom_{{\rm Sk}_\mathcal{V}(\Sigma)}(\varnothing\lhd-, \varnothing) \allowbreak\tilde{\Rightarrow} \Hom_{\mathcal{E}}(-, A_\Sigma)$. What we mean is that for $V \in \mathcal{V}$ and $\alpha \in \Hom_{{\rm Sk}_\mathcal{V}(\Sigma)}(\varnothing \lhd V, \varnothing)$, we have $\alpha \circ \allowbreak (\halft_V\rhd {\rm Id}_\varnothing)\in \Hom_{{\rm Sk}_\mathcal{V}(\Sigma)}(V\rhd \varnothing, \varnothing)$, and $\sigma^R_V(\alpha) := \sigma_V(\alpha \circ (\halft_V \rhd {\rm Id}_\varnothing))$.
\end{Definition}
\begin{Proposition}\label{propASigRightBrOpp}
The natural isomorphism $\sigma^R$ exhibits $A_\Sigma$ as the right internal skein algebra. The algebra structure $m^R\colon A_\Sigma \otimes A_\Sigma \to A_\Sigma$ inherited from this right internal endomorphism object structure differs from the left one $m\colon A_\Sigma \otimes A_\Sigma \to A_\Sigma$ by a braiding: $m^R = m \circ c_{A_\Sigma \otimes A_\Sigma}$.

Namely, the right internal skein algebras introduced in {\rm \cite{Gunningham}} and {\rm \cite{BBJ}} are the braided opposites of the left ones introduced in Section~$\ref{SectASigma}$.
\end{Proposition}

\begin{proof}
The $\sigma^R_V$'s form a natural isomorphism: let $f\in\Hom_\mathcal{V}(V,W)$ and $\alpha\in \Hom_{{\rm Sk}_\mathcal{V}(\Sigma)}(\varnothing \lhd W, \varnothing)$, one checks:
\begin{align*}
\sigma_V^R(\alpha \circ ({\rm Id}_\varnothing \lhd f)) &= \sigma_V^R\big(\alpha \circ \big(f^\halft \rhd {\rm Id}_\varnothing\big)\big) := \sigma_V\big(\alpha \circ \big(f^\halft \rhd {\rm Id}_\varnothing\big) \circ (\halft_V \rhd {\rm Id}_\varnothing)\big)
\\
&= \sigma_V(\alpha \circ (\halft_W \rhd {\rm Id}_\varnothing) \circ (f \rhd {\rm Id}_\varnothing)) = \sigma_W^R(\alpha) \circ f
\end{align*}
by using naturality of $\sigma$ and of $\halft\rhd -$. For $f \in \Hom_\mathcal{E}(V,A_\Sigma)$ one has $\big(\sigma^R_V\big)^{-1}(f) = \sigma_V^{-1}(f) \circ (\halft_V^{-1} \rhd {\rm Id}_\varnothing)$.
Therefore $\big(A_\Sigma,\sigma^R\big)$ is the internal endomorphism object of the empty set in ${\rm SK}_\mathcal{V}(\Sigma)$ with respect to the right ${\rm SK}_\mathcal{V}\big(\mathbb{R}^2\big)$-action by
 Proposition~\ref{propIntHomFree}. We now study its algebra structure.
The evaluation map is
\begin{align*}
{\rm ev}^R :&= \big(\sigma^R_{A_\Sigma}\big)^{-1}({\rm Id}_{A_\Sigma}) = \sigma_{A_\Sigma}^{-1}({\rm Id}_{A_\Sigma}) \circ \big(\halft_{A_\Sigma}^{-1} \rhd {\rm Id}_\varnothing\big)
\\
&={\rm ev} \circ \big(\halft_{A_\Sigma}^{-1} \rhd {\rm Id}_\varnothing\big) \in \Hom_{{\rm SK}_\mathcal{V}(\Sigma)}(\varnothing \lhd A_\Sigma, \varnothing).
\end{align*}
The product, or composition map, is:
\begin{align*}
m^R :&= \sigma^R_{A_\Sigma \otimes A_\Sigma}\big({\rm ev}^R\circ \big({\rm ev}^R \lhd {\rm Id}_{A_\Sigma}\big)\big)
\\
&= \sigma_{A_\Sigma \otimes A_\Sigma}\big({\rm ev} \circ \big(\halft_{A_\Sigma}^{-1} \rhd {\rm Id}_\varnothing\big) \circ {\rm Id}_{A_\Sigma^\halft}\rhd\big({\rm ev} \circ \big(\halft_{A_\Sigma}^{-1} \rhd {\rm Id}_\varnothing\big)\big)\circ (\halft_{A_\Sigma \otimes A_\Sigma}\rhd {\rm Id}_\varnothing)\big)
\\
 &=\sigma_{A_\Sigma \otimes A_\Sigma}\big({\rm ev} \circ ({\rm Id}_{A_\Sigma}\rhd {\rm ev}) \circ \big(\halft_{A_\Sigma}^{-1}\otimes\halft_{A_\Sigma}^{-1} \rhd {\rm Id}_\varnothing\big)\circ (\halft_{A_\Sigma \otimes A_\Sigma}\rhd {\rm Id}_\varnothing)\big)
\\
&=\sigma_{A_\Sigma \otimes A_\Sigma}({\rm ev} \circ ({\rm Id}_{A_\Sigma}\rhd {\rm ev}) \circ (c_{A_\Sigma \otimes A_\Sigma} \rhd {\rm Id}_\varnothing))
= m \circ c_{A_\Sigma \otimes A_\Sigma}.
\end{align*}
The units $\sigma^R_{1_\mathcal{V}}({\rm Id}_\varnothing) := \sigma_{1_\mathcal{V}}({\rm Id}_\varnothing \circ (\halft_{1_\mathcal{V}} \rhd {\rm Id}_\varnothing))= \sigma_{1_\mathcal{V}}({\rm Id}_\varnothing)$ coincide.
\end{proof}

\begin{Remark}\label{rmkLinkTwoAsigmaR}
In ${\rm SK}_\mathcal{V}\big(\mathbb{R}^2\big)$, one has $\sigma^R_V(\alpha) := \sigma_V(\alpha \circ (\halft_V \rhd {\rm Id}_\varnothing)) = \sigma_{V^\halft}(\alpha) \circ \halft_V$. If one post-composes with $\halft_{A_\Sigma^\halft}^{-1}$ this is exactly $(\sigma_{V^\halft})^\halft$. Hence the two descriptions of right internal skein algebras we gave, $\big(A_\Sigma^\halft,{\sigma_{(-)^\halft}}^\halft\big)$ in Remark \ref{rmkASigmaRnaive} and $(A_\Sigma,\sigma \circ (\halft\rhd-))$ in Proposition~\ref{propASigRightBrOpp}, are isomorphic (as they should) by $\halft_{A_\Sigma^\halft}^{-1} \colon A_\Sigma \to A_\Sigma^\halft$. The product on $A_\Sigma^\halft$ is given by
\begin{align*}
\halft_{A_\Sigma^\halft}^{-1} \circ m^R \circ \halft_{A_\Sigma^\halft} \otimes \halft_{A_\Sigma^\halft} & = \halft_{A_\Sigma^\halft}^{-1} \circ m \circ c_{A_\Sigma \otimes A_\Sigma} \circ \halft_{A_\Sigma^\halft} \otimes \halft_{A_\Sigma^\halft} \\
& = \halft_{A_\Sigma^\halft}^{-1} \circ m \circ \halft_{A_\Sigma \otimes A_\Sigma} = m^\halft.
\end{align*}
\end{Remark}

\subsection{Multi-edges internal skein algebras} We extend the definition of internal skein algebras to the multi-edge context, and define them as internal endomorphism algebras of the empty set in the skein category with multiple boundary actions, as expected. We check that they still describe skein categories well-enough.
\begin{Definition}\label{defMultiEdgeASigma}
Let $\mathfrak{S}$ be a marked surface with $n$ boundary edges labelled either as left (numbe\-red~1 to $k$) or as right (numbered $k+1$ to $n$) edges. Each left (resp. right) boundary edge induces a left ${\rm Sk}_\mathcal{V}\big(\mathbb{R}^2\big)$-action (resp.~left ${\rm Sk}_\mathcal{V}\big(\mathbb{R}^2\big)^{\otimes\text{-}{\rm op}}$-action) on ${\rm Sk}_\mathcal{V}(\mathfrak{S})$, which all commute, so one has a left ${\rm Sk}_\mathcal{V}\big(\mathbb{R}^2\big)^{\otimes k}\otimes \big({\rm Sk}_\mathcal{V}\big(\mathbb{R}^2\big)^{\otimes\text{-}{\rm op}}\big)^{\otimes n-k}$-action $\bmrhd$ on ${\rm Sk}_\mathcal{V}(\mathfrak{S})$. We denote its components by $\rhd_i$ or $\lhd_i\circ \operatorname{fl}\colon {\rm Sk}_\mathcal{V}\big(\mathbb{R}^2\big) \otimes {\rm Sk}_\mathcal{V}(\mathfrak{S})\to {\rm Sk}_\mathcal{V}(\mathfrak{S})$, $1\leq i\leq n$, though we forget the indices when they are understood. When there are missing components they will be implicitly filled by $1_\mathcal{V}$. We may also write $(V_1 , \dots , V_k){\rhd} \varnothing \lhd (V_{k+1},\dots,V_n)$ instead of $(V_1 , \dots , V_n)\bmrhd \varnothing$.

The internal skein algebra $A_\mathfrak{S}$ is the internal endomorphism object of the empty set in ${\rm Sk}_\mathcal{V}(\mathfrak{S})$ with respect to the ${\rm Sk}_\mathcal{V}\big(\mathbb{R}^2\big)^{\otimes k}\otimes ({\rm Sk}_\mathcal{V}\big(\mathbb{R}^2\big)^{\otimes\text{-}{\rm op}})^{\otimes n-k}$-action. It is an algebra object in $\mathcal{E}^{\boxtimes n} \simeq \Free\big(\mathcal{V}^{\otimes n}\big) \simeq \Free\big({\rm Sk}_\mathcal{V}\big(\mathbb{R}^2\big)^{\otimes n}\big)$, where $\mathcal{E}^{\boxtimes n}$ has opposite tensor products on the last $n-k$ components. We denote this monoidal structure by $\bmotimes$. We denote by $\otimes_i$ the tensor product on coordinate $i$, and adopt the same convention as with $\bmrhd$ filling with missing $1_\mathcal{V}$'s and eventually writing $(V_1 , \dots , V_k)\otimes W \otimes (V_{k+1},\dots,V_n)$ instead of $(V_1 , \dots , V_n)\bmotimes W$.
Explicitly, $A_\mathfrak{S}$ comes equipped with natural isomorphisms $\sigma_{\vec{V}}\colon \Hom_{{\rm Sk}_\mathcal{V}(\mathfrak{S})}(\vec{V} \bmrhd \varnothing, \varnothing) \tilde{\Rightarrow} \Hom_{\mathcal{E}^{\boxtimes n}}(\vec{V} , A_\mathfrak{S})$ for \mbox{$\vec{V} = (V_1,\dots,V_n) \in \mathcal{V}^{\otimes n}$}, between functors $(\mathcal{V}^{\rm op})^{\otimes n}\allowbreak\to {\rm Vect}_k$.
\end{Definition}

\begin{Remark}
Objects and morphisms of tensor product of categories (e.g., $\mathcal{V}^{\otimes n}$) are sometimes denoted by tensor product of objects and morphisms (e.g., $V_1 \otimes \cdots \otimes V_n$ and $f_1 \otimes \cdots \otimes f_n$). To avoid confusion with the monoidal structures on the categories (e.g., $\otimes$ on $\mathcal{V}$), we prefer to use commas (e.g., $(V_1,\dots,V_n)$ and $(f_1,\dots,f_n)$).
\end{Remark}

\begin{Remark}
A legitimate worry about this extended definition of internal skein algebras is that when one has multiple boundary actions on a same connected component one cannot keep track of where did an object come from. Namely for $V \in \mathcal{V}$ and $c_1$, $c_2$ two boundary edges on a same connected component of $\mathfrak{S}$, one has an isomorphism $V\rhd_1\varnothing \to V\rhd_2\varnothing$ which sounds surprising because $(V,1_\mathcal{V})$ and $(1_\mathcal{V},V)$ are hardly isomorphic in $\mathcal{V}^{\otimes 2}$. The internal skein algebra actually keeps track of such identifications, and one has an isomorphism $(V,1_\mathcal{V})\bmotimes A_\mathfrak{S} \to(1_\mathcal{V},V)\bmotimes A_\mathfrak{S}$.
\end{Remark}

Note that the above definition makes sense for $n=0$, where we want endomorphisms of the empty set in ${\rm Sk}_\mathcal{V}(\mathfrak{S})$ to be described by morphisms $k\to A_\mathfrak{S}$ in $\mathcal{E}^{\boxtimes 0} = {\rm Vect}_k$. For $\mathcal{V}=\Oqcomfin$ one gets $A_\mathfrak{S} = \mathring{\mathscr{S}}(\mathfrak{S})$ is the usual skein algebra. It is no longer true, however, that all objects of ${\rm Sk}_\mathcal{V}(\mathfrak{S})$ are described as modules over $A_\mathfrak{S}$, because the (trivial) action of $\mathcal{E}^{\boxtimes 0}$ on $\varnothing$ is no longer dominant.

\begin{Definition}
Let $\mathfrak{S}$ be a marked surface and $\mathcal{V}$ a ribbon category. The reduced skein category ${\rm Sk}_\mathcal{V}^{\rm red}(\mathfrak{S})$ is the full subcategory of ${\rm Sk}_\mathcal{V}(\mathfrak{S})$ spanned by objects of the form $\vec{V}\bmrhd \varnothing$, namely in the image of the action of ${\rm Sk}_\mathcal{V}\big(\mathbb{R}^2\big)^{\otimes n}$ on the empty set.
It is equivalent to ${\rm Sk}_\mathcal{V}(\mathfrak{S})$ if~$\mathfrak{S}$ has at least one boundary edge per connected component.
\end{Definition}

\begin{Remark}\label{rmkTopPtsMultiEdge}
One can still apply Remark \ref{rmkTopPtsAndASigma}, slightly modified because for right edges the left adjoint of $-\lhd V$ is given by acting by the left dual $-\lhd {}^*V$. For $\vec{V}=(V_1 ,\dots , V_n) \in \mathcal{V}^{\otimes n}$ write $\vec{V}^* = ({V_i}^*$ or $^*V_i)_{1\leq i\leq n}$ with right duals for left edges and left duals for right edges. Then, as the notation suggests, $\vec{V}^*$ is the left dual of $\vec{V}$ in $\mathcal{V}^{\otimes n}$ for the monoidal structure $\bmotimes$, and $\vec{V}\bmrhd -$ has left adjoint $\vec{V}^*\bmrhd -$. For $\vec{W}\in\mathcal{V}^{\otimes n}$, one has natural isomorphisms:
\begin{align*}
\Hom_{{\rm Sk}_\mathcal{V}(\mathfrak{S})}\big(\vec{W}\bmrhd \varnothing, \vec{V} \bmrhd \varnothing\big)
&\quad \simeq \quad\Hom_{{\rm Sk}_\mathcal{V}(\mathfrak{S})}\big(\vec{V}^*\bmrhd (\vec{W}\bmrhd \varnothing), \varnothing\big)
\\
& \overset{\sigma_{\vec{V}^* \bmotimes \vec{W}}}\to \Hom_{\mathcal{E}^{\boxtimes n}}\big(\vec{V}^* \bmotimes \vec{W} , A_\mathfrak{S}\big)\ \simeq\ \Hom_{\mathcal{E}^{\boxtimes n}}\big( \vec{W} ,\vec{V} \bmotimes A_\mathfrak{S}\big).
\end{align*}
\end{Remark}

\begin{Theorem}\label{thmReducedSkModASig}
The free cocompletion of the reduced skein category ${\rm SK}_\mathcal{V}^{\rm red}(\mathfrak{S})$ is equivalent to the category of right $A_\mathfrak{S}$-modules in $\mathcal{E}^{\boxtimes n}$, with monoidal structure $\bmotimes$, by
\begin{align*}
{\rm SK}_\mathcal{V}^{\rm red}(\mathfrak{S}) &\ \tilde{\to}\ {\rm mod}_\mathcal{E}-A_\mathfrak{S},
\\
M &\mapsto \undHom(\varnothing,M).
\end{align*}
For $M$ of the form $\vec{V}\bmrhd\varnothing$, one has $\undHom (\varnothing,\vec{V}\bmrhd \varnothing)\simeq\vec{V}\bmotimes A_\mathfrak{S}$.
\end{Theorem}

\begin{proof}
We follow the proof of Theorem~\ref{thASigmaMods}, namely we use \cite[Theorem~4.6]{BBJ} on $\varnothing \in {\rm SK}_\mathcal{V}^{\rm red}(\mathfrak{S})$. It~is projective by the same arguments and $-\rhd \varnothing\colon \mathcal{E}^{\boxtimes n} \to {\rm SK}_\mathcal{V}^{\rm red}(\mathfrak{S})$ is dominant by construction so ${\rm act}_\varnothing^R$ is faithful and $\varnothing$ is a generator. For the last statement, one has $\undHom(\varnothing,\vec{V}\bmrhd\varnothing) \simeq \vec{V} \bmotimes A_\mathfrak{S}$ by Remark~\ref{rmkTopPtsMultiEdge}.
\end{proof}

\subsection{Relation for multiple left edges} Let $\mathcal{V} = \Oq\text{-}{\rm comod}^{\rm fin}$, $\mathcal{E} = \Oq\text{-}{\rm comod} \simeq \Free(\mathcal{V})$ and $\mathfrak{S}$ be a marked surface with all boundary edges labelled left. We show that $A_\mathfrak{S} \simeq \mathscr{S}(\mathfrak{S})$ as $\Oq^{\otimes n}$-comodule-algebras.
\begin{Proposition} There is an equivalence of categories $\mathcal{E}^{\boxtimes n}\simeq \Oq^{\otimes n}\text{-}{\rm comod}$.
\end{Proposition}
\begin{proof}
The category $\Oq^{\otimes n}\text{-}{\rm comod}$ is semi-simple with simples tensor products of simples $\Oq$-comodules. It implies that the cocontinuous extension of $\otimes^n\colon \mathcal{V}^{\otimes n}\to \Oq^{\otimes n}$-comod to $\mathcal{E}^{\boxtimes n}\to \Oq^{\otimes n}\text{-}{\rm comod}$ is an equivalence.
\end{proof}

In particular $\Free\big({\rm TL}^{\otimes n}\big)\simeq \Free({\rm TL})^{\boxtimes n}\simeq \mathcal{E}^{\boxtimes n} \simeq \Oq^{\otimes n}\text{-}{\rm comod}$.
\begin{Theorem}\label{thmRelationMulti}
Let $\mathfrak{S}$ be a marked surface with all boundary edges labelled left, then $A_\mathfrak{S}\simeq \mathscr{S}(\mathfrak{S})$ as $\Oq^{\otimes n}$-comodule-algebras.
\end{Theorem}
\begin{proof}
We give a natural isomorphism $\St$ exhibiting $\mathscr{S}(\mathfrak{S})$ as the internal endomorphism object of $\varnothing\in {\rm SK}_\mathcal{V}(\mathfrak{S})$ with respect to the $\Oq^{\otimes n}\text{-}{\rm comod}$-module structure. For $X \in \Oq^{\otimes n}\text{-}{\rm comod}$, we want ${\rm St}_X\colon \Hom_{{\rm SK}_\mathcal{V}(\mathfrak{S})}(X \rhd \varnothing,\varnothing) \to \Hom_{\Oq^{\otimes n}\text{-}{\rm comod}}(X,\mathscr{S}(\mathfrak{S}))$. Let $X \in {\rm TL}^{\otimes n}$ and $\alpha \in \Hom_{{\rm Sk}_\mathcal{V}(\mathfrak{S})}(X \rhd \varnothing,\varnothing)$ represented by a tangle, we set
\[
\St_X(\alpha)\colon \begin{cases} X \to \mathscr{S}(\mathfrak{S}),
\\
v_{\vec{\varepsilon}_1}^1 \otimes \cdots \otimes v_{\vec{\varepsilon}_n}^n \mapsto {}_{\vec{\varepsilon}_1\cdots \vec{\varepsilon}_n}\alpha, \end{cases}
\]
where ${}_{\vec{\varepsilon}_1\cdots \vec{\varepsilon}_n}\alpha$ is the tangle $\alpha$ with endpoints pushed over the boundary edges and states the~$\varepsilon_j^i$ over the $i$-th edge and in $j$-th position from top to bottom, as in Section~\ref{SectRelation} but with more than one edge.
It~is an~$\Oq^{\otimes n}$-comodule morphism because it is an $\Oq$-comodule morphism on each coordinate by the same calculations as in Proposition~\ref{propStOqmorph}. It is natural in~${\rm TL}^{\otimes n}$ because it is a natural in each coordinate by the same calculations as in Proposition~\ref{propStNat}.

It is an isomorphism by the same arguments as in Theorem~\ref{thmRelation}. Namely, let $W = V_1\otimes \cdots \otimes V_n \in {\rm TL}^{\otimes n}$ and $f\colon W \to \mathscr{S}(\mathfrak{S})$, split $W = \oplus_i W_i$ and $f = \oplus f_i$ with $W_i = V_{i,1}\otimes\cdots\otimes V_{i,n}$ simple and choose $w_i \in W_i\smallsetminus\{0\}$. Choose $a_i$ representing $f_i(w_i)$ and denote by $\alpha_i$ its underlying tangle and $\vec{\varepsilon}_{i,1},\dots,\vec{\varepsilon}_{i,n}$ its states. Include $W_i \overset{g_i}\hookrightarrow V^{\otimes n_{i,1}}\otimes\cdots\otimes V^{\otimes n_{i,n}}$ by mapping $w_i$ to $v_{\vec{\varepsilon}_{i,1}}\otimes\cdots\otimes v_{\vec{\varepsilon}_{i,n}}$, and set ${\rm St}_{W_i}^{-1}(f_i) = \alpha_i \circ (g_i \bmrhd {\rm Id}_\varnothing) \in \Hom_{{\rm Sk}_\mathcal{V}(\mathfrak{S})}(W_i\bmrhd \varnothing,\varnothing)$. Then the inverse of $\St$ is given by $\St_W^{-1}(f) = \oplus_i \St_{W_i}^{-1}(f_i) \in \Hom_{{\rm Sk}_\mathcal{V}(\mathfrak{S})}(W\bmrhd \varnothing,\varnothing)$ and does not depend on the choice of representative.

As in Proposition~\ref{propProdsOnSS}, because every boundary edge of $\mathfrak{S}$ is labelled left, the product inherited from the internal endomorphism object structure is still given by $\alpha$ with prescribed inputs $v_{\vec{\varepsilon}_{i,1}}\otimes\cdots\otimes v_{\vec{\varepsilon}_{i,n}}$ times $\beta$ with prescribed inputs $v_{\vec{\eta}_{i,1}}\otimes\cdots\otimes v_{\vec{\eta}_{i,n}}$ equals $\alpha \circ ({\rm Id} \bmrhd \beta)$ with prescribed inputs $(v_{\vec{\varepsilon}_{i,1}}\otimes v_{\vec{\eta}_{i,1}})\otimes\cdots\otimes (v_{\vec{\varepsilon}_{i,n}}\otimes v_{\vec{\eta}_{i,n}})$ which is the usual product on $\mathscr{S}(\mathfrak{S})$.
\end{proof}

\subsection[The half twist on O\_\{q\textasciicircum{}2\}(SL\_2)-comod]{The half twist on $\boldsymbol{\Oq\text{-}{\rm comod}}$}

In the last subsection, we only allowed left ${\rm Sk}_\mathcal{V}\big(\mathbb{R}^2\big)$-actions. We study here how to change from left to right actions using the half twist in the case $\mathcal{V} = \Oqcomfin$.
\begin{Remark}
As we saw, a half twist on $\mathcal{V}$ is usually not necessary to the general study of internal skein algebras, but it is needed to relate them to stated skein algebras when there are right edges, and to mirror their excision properties.
When one sees a boundary edge at the right instead of the left, it has very different effects on both sides. For stated skein algebras, it does not change the vector space, nor the algebra structure, but switches the right comodule structure to a left one using $\operatorname{rot}_*$. For internal skein algebras, it does not change the vector space, one keeps right comodules ($A_\Sigma^R$ is still an object of $\mathcal{E}$) though slightly changed: it is half-twisted, and the algebra structure is opposed. To~make both sides agree, one needs to switch the comodule structure of the internal skein algebra while taking the opposite of its algebra structure. This is done very naturally by using $S$. Using the other known comparisons, one gets that the half twist on $\Oqcom$ should be the difference between switching the comodule structure using $\operatorname{rot}_*$ and switching it using $S$. This is Proposition~\ref{prophtDiffLl}, but we give a more complete and algebraic definition below.
\end{Remark}

\begin{Definition}[{\cite[Section~4.1]{SnyderTingley}}, for categories of comodules]
A half-coribbon Hopf algebra is a~coribbon Hopf algebra $H$ equipped with a half-coribbon functional, i.e., a map $t \colon H \to k$ such that:
\begin{enumerate}\itemsep=0pt
\item[$(1)$] $t$ is invertible by convolution: $\exists\ t^{-1} \colon H \to k$ such that $t(a_{(1)})t^{-1}(a_{(2)}) = t^{-1}(a_{(1)})t(a_{(2)}) = \varepsilon(a)$,

\item[$(2)$] $t$ squares to the twist: $t(a_{(1)})t(a_{(2)}) = \theta(a)$,

\item[$(3)$] compatibility with product: $t(a.b) = t(b_{(1)}) t(a_{(1)})R(a_{(2)}\otimes b_{(2)})$.
\end{enumerate}
\end{Definition}
\begin{Definition}
The half-coribbon functional induces a half twist $\halft$ on the category $H\text{-}{\rm comod}$ by $\halft_V\colon V \overset{\Delta}{\to} V \otimes H \overset{{\rm Id} \otimes t}{\longrightarrow} V$. It is an isomorphism of vector space with $\halft_V^{-1} \colon V \overset{\Delta}{\to} V \otimes H \overset{{\rm Id} \otimes t^{-1}}{\longrightarrow} V$. The half-coribbon functional is not supposed to be central though, and this means that $\halft_V$ is not an $H$-comodule morphism. There is a unique comodule structure on (the target) $V$ that makes it a~comodule morphism, namely $\Delta^\halft := (\halft_V \otimes {\rm Id}_H) \circ \Delta \circ \halft_V^{-1}$. We denote by $V^\halft$ the vector space $V$ equipped with the coaction $\Delta^\halft$. Now, $\halft_V \colon V \to V^\halft$ is an isomorphism of $H$-comodules.

One has a functor $(-)^\halft\colon H\text{-}{\rm comod} \to H\text{-}{\rm comod}$ which sends an object $V$ to the half-twisted~$V^\halft$ and a morphism $f\colon V\to W$ to $f^\halft = \halft_W \circ f \circ \halft_V^{-1}\colon V^\halft \to W^\halft$. It is defined so that $\halft\colon {\rm Id} \Rightarrow (-)^\halft$ is a natural isomorphism.
\end{Definition}
As maps of vector spaces, one simply has
\begin{align*}
f^\halft &= ({\rm Id}_W \otimes t) \circ \Delta_W \circ f \circ \big({\rm Id}_V \otimes t^{-1}\big) \circ \Delta_V \\
&= ({\rm Id}_W \otimes t) \circ \Delta_W \circ\big({\rm Id}_W \otimes t^{-1}\big) \circ (f \otimes {\rm Id}_H) \circ \Delta_V
\\
& = ({\rm Id}_W \otimes t) \circ \Delta_W \circ\big({\rm Id}_W \otimes t^{-1}\big) \circ \Delta_W \circ f= \halft_W \circ \halft_W^{-1} \circ f = f.
\end{align*}
In particular $\halft_{V^\halft}= \halft_V^\halft = \halft_V$. The square of $t$ is $\theta$ so $\halft_{V^\halft}\circ\halft_V = \theta_V$, and $\theta$ is central so $(V^\halft)^\halft = V$, and $(-)^\halft \circ (-)^\halft = {\rm Id}$.
Regarding the monoidal structure, let $V$ and $W$ be two $H$-comodules. Remember that the braiding is defined as
\[
c_{V,W} \colon V \otimes W \overset{R_{24}\circ (\Delta_V \otimes \Delta_W)}\longrightarrow V \otimes W \overset{\operatorname{fl}}\to W \otimes V.
\]
 The third condition gives that $\halft_{V \otimes W} = \halft_V \otimes \halft_W \circ (\operatorname{fl} \circ c_{V,W})$ so $\operatorname{fl}\circ\halft_{V \otimes W} = \halft_W \otimes \halft_V \circ c_{V,W}$. In particular, $\operatorname{fl}\colon (V\otimes W)^\halft \to W^\halft \otimes V^\halft$ is an $H$-comodule isomorphism.

\begin{Definition}\sloppy
For $\mathcal{V}$ a ribbon category, let $\halft\text{-}{\rm Ribbon}_\mathcal{V}$ be the full subcategory of ${\rm Ribbon}_\mathcal{V}\big(\mathbb{R}^2\big)$ spanned by objects of the form $[n] \subseteq \mathbb{R}^2$ but now allowing either blackboard or anti-blackboard framing for every point. This subcategory is still ribbon and is stable by the functor $(-)^\halft$, and equipped with $\halft\colon {\rm Id} \Rightarrow (-)^\halft$. It also contains the category ${\rm Ribbon}_\mathcal{V}$ of blackboard framed points.
\end{Definition}

\begin{Theorem}[{\cite[Theorem~4.11]{SnyderTingley}}] Let $H$ be a half-coribbon Hopf algebra and $V$ a finite-dimensional comodule. There is a unique monoidal functor $\halft\text{-}{\rm RT}_V\colon \halft\text{-}{\rm Ribbon}_{H\text{-}{\rm comod}^{\rm fin}} \to H\text{-}{\rm comod}^{\rm fin}$ extending ${\rm RT}_V$ and commuting with both $(-)^\halft$ and $\halft$, so preserving the ``half-ribbon structure''.
\end{Theorem}

\begin{Remark}
This half-twisted Reshetikhin--Turaev functor also gives an equivalence of categories ${\rm Sk}_\mathcal{V}\big(\mathbb{R}^2\big) \simeq \mathcal{V}$ but this time with much nicer properties regarding the half twist. In the usual Reshetikhin--Turaev functor one only prescribes where to send points with blackboard framing and well-placed on the real line. For framed points not of this form, one has to choose an isomorphism with one of these, like in Remark~\ref{rmkEqCatSkR2VV}, but these choices are quite arbitrary. Then the half twist sends blackboard framed points to anti-blackboard framed points which are re-identified with blackboard framed points via these arbitrary isomorphisms. With the half-twisted Reshetikhin--Turaev functor one also prescribes where to send anti-blackboard framed points, so one controls closely what happens with the half twist, namely the half twist on ${\rm Sk}_\mathcal{V}(\mathbb{R}^2)$ is mapped to the half twist on~$\mathcal{V}$.

Note however that unlike on ${\rm Sk}_\mathcal{V}\big(\mathbb{R}^2\big)$, the half twist on $\mathcal{V}$ is not strictly anti-monoidal (indeed it is the identity on underlying vector spaces) but $(X\otimes Y)^\halft \simeq Y^\halft \otimes X^\halft$ is simply given by the flip of tensors. A bit like the $R$-matrix, the half twist gives the difference between the monoidal structure on $H\text{-}{\rm comod}$ and the symmetric one on~${\rm Vect}_k$. Formally, this error lies in the fact that the inclusion of $\mathcal{V}$ in ${\rm Sk}_\mathcal{V}\big(\mathbb{R}^2\big)$ is only monoidal up to natural isomorphism, and this isomorphism, given in Remark~\ref{rmkInclVtoSkMonoid}, maps by the half twist to the flip of tensors.
\end{Remark}

In the case of $\Oq$, we can define a half-coribbon functional on the generators by $t \left( \begin{smallmatrix} a & b \\ c & d \end{smallmatrix} \right) = \left( \begin{smallmatrix} 0 & -q^{\frac{5}{2}} \\ q^{\frac{1}{2}} & 0 \end{smallmatrix}\right)$. This tells in particular how the half twist acts on the standard corepresentation $V$, namely $\halft_V(v_+) = q^{\frac{1}{2}} v_- = C(-)^{-1} v_-$ and $\halft_V(v_-) = -q^{\frac{5}{2}} v_+ = C(+)^{-1} v_+$. For states $\eta \in \{\pm\}$, we will write $\halft_V(\eta) = -\eta .C(-\eta)^{-1}$.
Note that \cite{SnyderTingley} introduce another half-coribbon element corresponding to the matrix $\left( \begin{smallmatrix} 0 & q^{\frac{3}{2}} \\ -q^{\frac{3}{2}} & 0 \end{smallmatrix} \right)$, but our choice is imposed by conventions from stated skein algebras.

To define it on all $\Oq$ we prefer a geometric description on $\mathscr{S}(B)$. We would like to give the same definition as for the coribbon functional $\theta$ with a half twist instead of a full twist, but in the definition of stated skein algebras one only allows upward-framed boundary points, which would map to downward-framed points after the half twist. Still, we know how to do the ``global'' half twist on many strands (without twisting the framing), and we only need to add a~``local'' half twist on each strand, which are implicitly coloured by~$V$, on which we know how the half twist acts.

\begin{Proposition}
The coribbon Hopf algebra $\mathscr{S}(B)$ is half-coribbon with half-coribbon functional
\[
t\raisebox{-3pt}{$\left(\rule{0pt}{24pt}\right.$}\!\!
\begin{tikzpicture}[scale = 0.7, baseline = 0pt]
\draw (0,-1) node{\small $\bullet$} arc(-90:90:1) node[near end, sloped]{\small $<$} node{\small $\bullet$};
\draw (0,-1) arc(-90:-270:1) node[near end, sloped]{\small $>$};
\draw[thick] (-0.94,-0.3)node[below left]{\small $\varepsilon_n$} -- (0.94,-0.3)node[below right]{\small $\eta_m$};
\draw[thick] (-1,0) node[yshift=2pt,left]{\footnotesize $\vdots$} -- (1,0)node[yshift=2pt,right]{\footnotesize $\vdots$};
\draw[thick] (-0.94,0.3) node[above left]{\small $\varepsilon_1$} -- (0.94,0.3)node[above right]{\small $\eta_1$};
\node[circle, fill=white, draw = black, scale = 1.2, inner sep = 1.5pt] (X) at (0,0){$\alpha$};
\end{tikzpicture}\!\!
\raisebox{-3pt}{$\left.\rule{0pt}{24pt}\right)$}
:= \varepsilon\raisebox{-3pt}{$\left(\rule{0pt}{24pt}\right.$}\!\!
\begin{tikzpicture}[scale = 0.7, baseline = 0pt]
\draw (0,-1) node{\small $\bullet$} arc(-90:90:1) node[near end, sloped]{\small $<$} node{\small $\bullet$};
\draw (0,-1) arc(-90:-270:1) node[near end, sloped]{\small $>$};
\draw[thick] (-0.94,-0.3)node[below left]{\small $\varepsilon_n$} -- (0.3,-0.3) .. controls (0.5,-0.3) and (0.74,0.3).. (0.94,0.3)node[above right]{\small $\halft_V(\eta_m)$} node[pos = 0.5, circle, fill = white, scale = 0.66]{};
\draw[thick] (-1,0) node[yshift=2pt, left]{\footnotesize $\vdots$} -- (1,0)node[yshift=2pt,right=5pt]{\footnotesize $\vdots$} node[pos = 0.81, circle, fill = white, scale = 0.46]{};
\draw[thick] (-0.94,0.3) node[above left]{\small $\varepsilon_1$} -- (0.3,0.3) .. controls (0.5,0.3) and (0.74,-0.3).. (0.94,-0.3)node[below right]{\small $\halft_V(\eta_1)$};
\node[circle, fill=white, draw = black, scale = 1.2, inner sep = 1.5pt] (X) at (-0.2,0){$\alpha$};
\end{tikzpicture}\!\!
\raisebox{-3pt}{$\left.\rule{0pt}{24pt}\right)$}
=
 \varepsilon\raisebox{-3pt}{$\left(\rule{0pt}{24pt}\right.$}\!\!
 \begin{tikzpicture}[scale = 0.7, baseline = 0pt, rotate = 180]
\draw (0,-1) node{\small $\bullet$} arc(-90:90:1) node[near start, sloped]{\small $>$} node{\small $\bullet$};
\draw (0,-1) arc(-90:-270:1) node[near start, sloped]{\small $<$};
\draw[thick] (-0.94,-0.3)node[above right]{\small $\eta_1$} -- (0.3,-0.3) .. controls (0.5,-0.3) and (0.74,0.3).. (0.94,0.3)node[below left]{\small $-\varepsilon_1.C(\varepsilon_1)^{-1}$} node[pos = 0.5, circle, fill = white, scale = 0.66]{};
\draw[thick] (-1,0) node[yshift=2pt, right]{\footnotesize $\vdots$} -- (1,0)node[yshift=2pt,left=10pt]{\footnotesize $\vdots$} node[pos = 0.81, circle, fill = white, scale = 0.46]{};
\draw[thick] (-0.94,0.3) node[below right]{\small $\eta_m$} -- (0.3,0.3) .. controls (0.5,0.3) and (0.74,-0.3).. (0.94,-0.3)node[above left]{\small $-\varepsilon_n.C(\varepsilon_n)^{-1}$};
\node[circle, fill=white, draw = black, scale = 1.2, inner sep = 1.5pt] (X) at (-0.2,0){$\alpha$};
\end{tikzpicture}\!\!
\raisebox{-3pt}{$\left.\rule{0pt}{24pt}\right)$}\!.
\]
By Remark $\ref{rmkFormActOnSS}$, $\halft_{\mathscr{S}(B)}(\alpha) = ({\rm Id} \otimes t)\circ\Delta(\alpha)$ is the stated tangle represented in the middle, and by a~left version of Remark $\ref{rmkFormActOnSS}$, $(t \otimes {\rm Id})\circ\Delta(\alpha)$ is the stated tangle represented in the right.
\end{Proposition}

\begin{proof} These two formulations prove that $t$ is well defined on $\mathscr{S}(B)$ as it respects the boundary relations on the left edge by the first and on the right one by the second. So we begin by proving that these two formulations actually coincide. Let $\beta \in \mathscr{S}(B)$, then
\begin{gather*}
 \varepsilon\raisebox{-3pt}{$\left(\rule{0pt}{24pt}\right.$}\!\!
 \begin{tikzpicture}[scale = 0.7, baseline = 0pt]
\draw (0,-1) node{\small $\bullet$} arc(-90:90:1) node[near end, sloped]{\small $<$} node{\small $\bullet$};
\draw (0,-1) arc(-90:-270:1) node[near end, sloped]{\small $>$};
\draw[thick] (-0.94,-0.3)node[below left]{\small $\varepsilon_n$} -- (0.3,-0.3) .. controls (0.5,-0.3) and (0.74,0.3).. (0.94,0.3)node[above right]{\small $-\eta_m.C(-\eta_m)^{-1}$} node[pos = 0.5, circle, fill = white, scale = 0.66]{};
\draw[thick] (-1,0) node[yshift=2pt, left]{\footnotesize $\vdots$} -- (1,0)node[yshift=2pt,right=5pt]{\footnotesize $\vdots$} node[pos = 0.81, circle, fill = white, scale = 0.46]{};
\draw[thick] (-0.94,0.3) node[above left]{\small $\varepsilon_1$} -- (0.3,0.3) .. controls (0.5,0.3) and (0.74,-0.3).. (0.94,-0.3)node[below right]{\small $-\eta_1.C(-\eta_1)^{-1}$};
\node[circle, fill=white, draw = black, scale = 1.2, inner sep = 1.5pt] (X) at (-0.2,0){$\beta$};
\end{tikzpicture}\!\!
\raisebox{-3pt}{$\left.\rule{0pt}{24pt}\right)$}
\overset{\varepsilon = \varepsilon \circ S}=
 \varepsilon\raisebox{-3pt}{$\left(\rule{0pt}{24pt}\right.$}\!\!
 \begin{tikzpicture}[scale = 0.7, baseline = 0pt, rotate = 180]
\draw (0,-1) node{\small $\bullet$} arc(-90:90:1) node[near start, sloped]{\small $>$} node{\small $\bullet$};
\draw (0,-1) arc(-90:-270:1) node[near start, sloped]{\small $<$};
\draw[thick] (-0.94,-0.3)node[above right]{\small $-\varepsilon_n.C(\varepsilon_n)^{-1}$} -- (0.3,-0.3) .. controls (0.5,-0.3) and (0.74,0.3).. (0.94,0.3)node[below left]{\small $\eta_m.C(-\eta_m)^{-1}.C(-\eta_m)$} node[pos = 0.5, circle, fill = white, scale = 0.66]{};
\draw[thick] (-1,0) node[yshift=2pt, right=10pt]{\footnotesize $\vdots$} -- (1,0)node[yshift=2pt,left=10pt]{\footnotesize $\vdots$} node[pos = 0.81, circle, fill = white, scale = 0.46]{};
\draw[thick] (-0.94,0.3) node[below right]{\small $-\varepsilon_1.C(\varepsilon_1)^{-1}$} -- (0.3,0.3) .. controls (0.5,0.3) and (0.74,-0.3).. (0.94,-0.3)node[above left]{\small $\eta_1.C(-\eta_1)^{-1}.C(-\eta_1)$};
\node[circle, fill=white, draw = black, scale = 1.2, inner sep = 1.5pt, rotate = 180] (X) at (-0.2,0){$\beta$};
\end{tikzpicture}\!\!
\raisebox{-3pt}{$\left.\rule{0pt}{24pt}\right)$}
\\ \qquad
{}=
 \varepsilon \raisebox{-3pt}{$\left(\rule{0pt}{24pt}\right.$}\!\!
 \begin{tikzpicture}[scale = 0.7, baseline = 0pt]
\draw (0,-1) node{\small $\bullet$} arc(-90:90:1) node[near start, sloped]{\small $<$} node{\small $\bullet$};
\draw (0,-1) arc(-90:-270:1) node[near start, sloped]{\small $>$};
\draw[thick] (-0.94,-0.3)node[below left]{\small $\eta_1$} -- (0.3,-0.3) .. controls (0.5,-0.3) and (0.74,0.3).. (0.94,0.3)node[above right]{\small $-\varepsilon_1.C(\varepsilon_1)^{-1}$} node[pos = 0.5, circle, fill = white, scale = 0.66]{};
\draw[thick] (-1,0) node[yshift=2pt, left]{\footnotesize $\vdots$} -- (1,0)node[yshift=2pt,right=5pt]{\footnotesize $\vdots$} node[pos = 0.81, circle, fill = white, scale = 0.46]{};
\draw[thick] (-0.94,0.3) node[above left]{\small $\eta_m$} -- (0.3,0.3) .. controls (0.5,0.3) and (0.74,-0.3).. (0.94,-0.3)node[below right]{\small $-\varepsilon_n.C(\varepsilon_n)^{-1}$};
\node[circle, fill=white, draw = black, scale = 1.2, inner sep = 1.5pt, rotate = 180] (X) at (-0.2,0){$\beta$};
\end{tikzpicture}\!\!
\raisebox{-3pt}{$\left.\rule{0pt}{24pt}\right)$}
\overset{\varepsilon = \varepsilon \circ \operatorname{rot}_*}
= \varepsilon\raisebox{-3pt}{$\left(\rule{0pt}{24pt}\right.$}\!\!
\begin{tikzpicture}[scale = 0.7, baseline = 0pt, rotate = 180]
\draw (0,-1) node{\small $\bullet$} arc(-90:90:1) node[near start, sloped]{\small $>$} node{\small $\bullet$};
\draw (0,-1) arc(-90:-270:1) node[near start, sloped]{\small $<$};
\draw[thick] (-0.94,-0.3)node[above right]{\small $\eta_1$} -- (0.3,-0.3) .. controls (0.5,-0.3) and (0.74,0.3).. (0.94,0.3)node[below left]{\small $-\varepsilon_1.C(\varepsilon_1)^{-1}$} node[pos = 0.5, circle, fill = white, scale = 0.66]{};
\draw[thick] (-1,0) node[yshift=2pt, right]{\footnotesize $\vdots$} -- (1,0)node[yshift=2pt,left=10pt]{\footnotesize $\vdots$} node[pos = 0.81, circle, fill = white, scale = 0.46]{};
\draw[thick] (-0.94,0.3) node[below right]{\small $\eta_m$} -- (0.3,0.3) .. controls (0.5,0.3) and (0.74,-0.3).. (0.94,-0.3)node[above left]{\small $-\varepsilon_n.C(\varepsilon_n)^{-1}$};
\node[circle, fill=white, draw = black, scale = 1.2, inner sep = 1.5pt] (X) at (-0.2,0){$\beta$};
\end{tikzpicture}\!\!
\raisebox{-3pt}{$\left.\rule{0pt}{24pt}\right)$}\!,
\end{gather*}
where the second equality is only a change of picture representation, not of stated tangles, coming from switching the orientation of the edges, see \cite[Section~3.5]{BonahonWong}.

The convolution inverse $t^{-1}$ of $t$ is obtained the same way as the middle term but with the inverse half twist and $\halft_V^{-1}$ on states. Indeed by Remark \ref{rmkFormActOnSS}, $\big({\rm Id} \otimes t^{-1}\big) \circ \Delta(\alpha)$ is $\alpha$ with an inverse half twist at the right and $\halft_V^{-1}$ on right states. Thus $\big(t\otimes t^{-1}\big)\circ \Delta$ is the counit of $\alpha$ with an inverse half twist and a half twist at the right, and $\halft_V\circ\halft_V^{-1}$ on right states, namely the counit of $\alpha$. Similarly, $\big(t^{-1}\otimes t\big)\circ\Delta=\varepsilon$.

One directly checks that $\halft_V \circ \halft_V = \theta_V = -q^3 {\rm Id}_V$ on the standard corepresentation. Then $({\rm Id}\otimes t)\circ \Delta (\alpha)$ is $\alpha$ with a half twist at the right and $\halft_V$ on right states, and $(t\otimes t)\circ \Delta (\alpha)$ is the counit of $\alpha$ with a full twist (without framing twist) at the right and $\theta_V$ on right states. This is exactly the full twist by separating the unframed full twist and the full twists on framings:
\begin{align*}
\theta(\alpha) & =\varepsilon\raisebox{-3pt}{$\left(\rule{0pt}{24pt}\right.$}\!\! \begin{tikzpicture}[scale = 0.7, baseline = 0pt]
\draw (0,-1) node{\small $\bullet$} arc(-90:90:1) node[near end, sloped]{\small $<$} node{\small $\bullet$};
\draw (0,-1) arc(-90:-270:1) node[near end, sloped]{\small $>$};
\node[circle, draw = black, scale = 0.85, inner sep = 1.5pt] (X) at (-0.1,-0.3){$\alpha$};
\draw[thick, double] (-0.94,-0.3) -- (X) -- (0.2,-0.3) .. controls (0.5,-0.3) and (0.8,0.4).. (0.5,0.4) node[pos = 0.37, circle, fill = white, scale = 0.46]{};
\draw[thick, double] (0.5,0.4).. controls (0.2,0.4) and (0.5,-0.3) .. (0.8,-0.3) -- (0.94,-0.3);
\end{tikzpicture}\!\!
\raisebox{-3pt}{$\left.\rule{0pt}{24pt}\right)$}
 \overset{\varepsilon = (\varepsilon\otimes \varepsilon)\circ \Delta}=
\varepsilon\raisebox{-3pt}{$\left(\rule{0pt}{24pt}\right.$}\!\!
\begin{tikzpicture}[scale = 0.7, baseline = 0pt]
\draw (0,-1) node{\small $\bullet$} arc(-90:90:1) node[near end, sloped]{\small $<$} node{\small $\bullet$};
\draw (0,-1) arc(-90:-270:1) node[near end, sloped]{\small $>$};
\draw[thick] (-0.94,-0.3) -- (-0.1,-0.3) .. controls (0.15,-0.3) and (0.15,0.3).. (0.4,0.3);
\draw[thick] (0.4,-0.3) .. controls (0.65,-0.3) and (0.65,0.3).. (0.9,0.3) -- (0.94,0.3)node[above right]{\small $\nu_1$};
\draw[thick] (-1,0) -- (0,0) -- (1,0)node[yshift=2pt,right]{\footnotesize $\vdots$} node[pos = 0.15, circle, fill = white, scale = 0.46]{} node[pos = 0.65, circle, fill = white, scale = 0.46]{};
\draw[thick] (-0.94,0.3) -- (-0.1,0.3) .. controls (0.15,0.3) and (0.15,-0.3).. (0.4,-0.3);
\draw[thick] (0.4,0.3) .. controls (0.65,0.3) and (0.65,-0.3).. (0.9,-0.3) -- (0.94,-0.3)node[below right]{\small $\nu_m$};
\node[rectangle, fill=white, draw = black, minimum height = 0.7cm, inner sep = 1.5pt] (X) at (-0.5,0){$\alpha$};
\end{tikzpicture}\!\!
\raisebox{-3pt}{$\left.\rule{0pt}{24pt}\right)$}\, .
\varepsilon\raisebox{-3pt}{$\left(\rule{0pt}{24pt}\right.$}\!\!
\begin{tikzpicture}[scale = 0.7, baseline = 0pt]
\draw (0,-1) node{\small $\bullet$} arc(-90:90:1) node[near end, sloped]{\small $<$} node{\small $\bullet$};
\draw (0,-1) arc(-90:-270:1) node[near end, sloped]{\small $>$};
\draw (-0.94,-0.3)node[below left]{\small $\nu_m$} -- (0.94,-0.3) node[pos = 0.3, circle, fill = white, scale = 0.46]{} node[pos = 0.3, yshift = 2.2pt, scale = 0.5, rotate = -90]{$\unortwist$} node[below right]{\small $\eta_m$};
\draw (-1,0) node[yshift=2pt,left]{\footnotesize $\vdots$} -- (1,0) node[pos = 0.5, circle, fill = white, scale = 0.46]{} node[yshift=2pt,right]{\footnotesize $\vdots$} node[pos = 0.5, yshift = 2.2pt, scale = 0.5, rotate = -90]{$\unortwist$};
\draw (-0.94,0.3)node[above left]{\small $\nu_1$} -- (0.94,0.3) node[pos = 0.7, circle, fill = white, scale = 0.46]{} node[above right]{\small $\eta_1$} node[pos = 0.7, yshift = 2.2pt, scale = 0.5, rotate = -90]{$\unortwist$};
\end{tikzpicture}\!\!
\raisebox{-3pt}{$\left.\rule{0pt}{24pt}\right)$}\\
& = \varepsilon\raisebox{-3pt}{$\left(\rule{0pt}{24pt}\right.$}\!\!
\begin{tikzpicture}[scale = 0.7, baseline = 0pt]
\draw (0,-1) node{\small $\bullet$} arc(-90:90:1) node[near end, sloped]{\small $<$} node{\small $\bullet$};
\draw (0,-1) arc(-90:-270:1) node[near end, sloped]{\small $>$};
\draw[thick] (-0.94,-0.3) -- (-0.1,-0.3) .. controls (0.15,-0.3) and (0.15,0.3).. (0.4,0.3);
\draw[thick] (0.4,-0.3) .. controls (0.65,-0.3) and (0.65,0.3).. (0.9,0.3) -- (0.94,0.3)node[above right]{\small $-q^3\eta_1$};
\draw[thick] (-1,0) -- (0,0) -- (1,0)node[yshift=2pt,right]{\footnotesize $\vdots$} node[pos = 0.15, circle, fill = white, scale = 0.46]{} node[pos = 0.65, circle, fill = white, scale = 0.46]{};
\draw[thick] (-0.94,0.3) -- (-0.1,0.3) .. controls (0.15,0.3) and (0.15,-0.3).. (0.4,-0.3);
\draw[thick] (0.4,0.3) .. controls (0.65,0.3) and (0.65,-0.3).. (0.9,-0.3) -- (0.94,-0.3)node[below right]{\small $-q^3\eta_m$};
\node[rectangle, fill=white, draw = black, minimum height = 0.7cm, inner sep = 1.5pt] (X) at (-0.5,0){$\alpha$};
\end{tikzpicture}\!\!
\raisebox{-3pt}{$\left.\rule{0pt}{24pt}\right)$}\!.
\end{align*}
Finally,
\begin{align*}
t(\alpha.\beta) = \varepsilon\raisebox{-3pt}{$\left(\rule{0pt}{24pt}\right.$}\!\!
\begin{tikzpicture}[scale = 0.7, baseline = 0pt, rotate = 180]
\draw (0,-1) node{\small $\bullet$} arc(-90:90:1) node[near start, sloped]{\small $>$} node{\small $\bullet$};
\draw (0,-1) arc(-90:-270:1) node[near start, sloped]{\small $<$};
\draw[thick] (-0.9,-0.45) -- (-0.4,-0.45) .. controls (0.3,-0.45) and (0.3,0.45).. (0.9,0.45)node[left]{\small $-\overleftarrow{\varepsilon}.C(\vec{\varepsilon})^{-1}$};
\draw[line width = 3pt, white] (-0.94,-0.3) -- (-0.4,-0.3) .. controls (0.3,-0.3) and (0.3,0.3).. (0.94,0.3);
\draw[thick] (-0.94,-0.3) -- (-0.4,-0.3) .. controls (0.3,-0.3) and (0.3,0.3).. (0.94,0.3);
\draw[line width = 3pt, white] (-0.94,0.3) -- (-0.4,0.3) .. controls (0.3,0.3) and (0.3,-0.3).. (0.94,-0.3);
\draw[thick] (-0.94,0.3) -- (-0.4,0.3) .. controls (0.3,0.3) and (0.3,-0.3).. (0.94,-0.3);
\draw[line width = 3pt, white] (-0.9,0.45) -- (-0.4,0.45) .. controls (0.3,0.45) and (0.3,-0.45).. (0.9,-0.45);
\draw[thick] (-0.9,0.45) -- (-0.4,0.45) .. controls (0.3,0.45) and (0.3,-0.45).. (0.9,-0.45)node[left]{\small $-\overleftarrow{\eta}.C(\vec{\eta})^{-1}$};
\node[circle, fill=white, draw = black, scale = 0.8, inner sep = 1.5pt] (X) at (-0.5,0.37){$\beta$};
\node[circle, fill=white, draw = black, scale = 1, inner sep = 1.5pt] (X) at (-0.5,-0.37){$\alpha$};
\end{tikzpicture}\!\!
\raisebox{-3pt}{$\left.\rule{0pt}{24pt}\right)$}
&\overset{\varepsilon = (\varepsilon\otimes \varepsilon)\circ \Delta}
= \varepsilon\raisebox{-3pt}{$\left(\rule{0pt}{24pt}\right.$}\!\!
\begin{tikzpicture}[scale = 0.7, baseline = 0pt]
\draw (0,-1) node{\small $\bullet$} arc(-90:90:1) node[near end, sloped]{\small $<$} node{\small $\bullet$};
\draw (0,-1) arc(-90:-270:1) node[near end, sloped]{\small $>$};
\draw[thick] (-0.9,-0.45) node[left]{\small $-\overleftarrow{\varepsilon}.C(\vec{\varepsilon})^{-1}$} -- (-0.4,-0.45) .. controls (0.3,-0.45) and (0.3,-0.3).. (0.9,-0.3);
\draw[line width = 3pt, white] (-0.94,-0.3) -- (-0.4,-0.3) .. controls (0.3,-0.3) and (0.3,-0.45).. (0.94,-0.45);
\draw[thick] (-0.94,-0.3) -- (-0.4,-0.3) .. controls (0.3,-0.3) and (0.3,-0.45).. (0.94,-0.45)node[right]{\small $\vec{\nu}$};
\draw[thick] (-0.94,0.3) -- (-0.4,0.3) .. controls (0.3,0.3) and (0.3,0.45).. (0.94,0.45)node[right]{\small $\vec{\mu}$};
\draw[line width = 3pt, white] (-0.9,0.45) -- (-0.4,0.45) .. controls (0.3,0.45) and (0.3,0.3).. (0.9,0.3);
\draw[thick] (-0.9,0.45) node[left]{\small $-\overleftarrow{\eta}.C(\vec{\eta})^{-1}$} -- (-0.4,0.45) .. controls (0.3,0.45) and (0.3,0.3).. (0.9,0.3);
\end{tikzpicture}\!\!
\raisebox{-3pt}{$\left.\rule{0pt}{24pt}\right)$}
. \varepsilon \raisebox{-3pt}{$\left(\rule{0pt}{24pt}\right.$}\!\!
\begin{tikzpicture}[scale = 0.7, baseline = 0pt, rotate = 180]
\draw (0,-1) node{\small $\bullet$} arc(-90:90:1) node[near start, sloped]{\small $>$} node{\small $\bullet$};
\draw (0,-1) arc(-90:-270:1) node[near start, sloped]{\small $<$};
\draw[thick] (-0.9,-0.45) -- (-0.4,-0.45) .. controls (0.3,-0.45) and (0.3,0.3).. (0.9,0.3);
\draw[thick] (-0.94,-0.3) -- (-0.4,-0.3) .. controls (0.3,-0.3) and (0.3,0.45).. (0.94,0.45)node[left]{\small $\vec{\nu}$};
\draw[line width = 3pt, white] (-0.94,0.3) -- (-0.4,0.3) .. controls (0.3,0.3) and (0.3,-0.45).. (0.94,-0.45);
\draw[thick] (-0.94,0.3) -- (-0.4,0.3) .. controls (0.3,0.3) and (0.3,-0.45).. (0.94,-0.45)node[left]{\small $\vec{\mu}$};
\draw[line width = 3pt, white] (-0.9,0.45) -- (-0.4,0.45) .. controls (0.3,0.45) and (0.3,-0.3).. (0.9,-0.3);
\draw[thick] (-0.9,0.45) -- (-0.4,0.45) .. controls (0.3,0.45) and (0.3,-0.3).. (0.9,-0.3);
\node[circle, fill=white, draw = black, scale = 0.8, inner sep = 1.5pt] (X) at (-0.5,0.37){$\beta$};
\node[circle, fill=white, draw = black, scale = 1, inner sep = 1.5pt] (X) at (-0.5,-0.37){$\alpha$};
\end{tikzpicture}\!\!
\raisebox{-3pt}{$\left.\rule{0pt}{24pt}\right)$}
\\
& \overset{\varepsilon\circ m = \varepsilon\otimes \varepsilon}= t(\beta_{(1)}) .t(\alpha_{(1)}).R(\alpha_{(2)}\otimes \beta_{(2)}).
\tag*{\qed}
\end{align*}
\renewcommand{\qed}{}
\end{proof}
\begin{Definition}
Let $\mathfrak{S}$ be a marked surface and $e$ a boundary edge, with orientation induced by the one of $\mathfrak{S}$. The inversion along the edge $e$ is the morphism of $k$-vector-spaces ${\rm inv}_e\colon \mathscr{S}(\mathfrak{S})\to\mathscr{S}(\mathfrak{S})$ given on a stated tangle $\alpha$ by ordering the heights according to the orientation of $e$, then switching height order vertically, then taking opposite states and some coefficients, namely:
\[
{\rm inv}_e\raisebox{-3pt}{$\left(\rule{0pt}{24pt}\right.$}
\begin{tikzpicture}[scale = 0.7, baseline = 0pt]
\draw (0,-1) node{\small $\bullet$} arc(-90:90:1) node[near end, sloped]{\small $<$} node{\small $\bullet$};
\draw (0,-1) -- (-1,-1);
\draw (0,1) -- (-1,1);
\draw[thick] (-1,-0.3) -- (0.94,-0.3)node[below right]{\small $\eta_m$};
\draw[thick] (-1,0) -- (1,0)node[yshift=2pt,right]{\footnotesize $\vdots$};
\draw[thick] (-1,0.3) -- (0.94,0.3)node[above right]{\small $\eta_1$};
\node[fill=white, minimum width = 1cm, minimum height = 0.7cm, leftymissy, inner sep = 1.5pt] (X) at (-0.5,0){$\alpha$};
\end{tikzpicture}\!\!
\raisebox{-3pt}{$\left.\rule{0pt}{24pt}\right)$}
:= \begin{tikzpicture}[scale = 0.7, baseline = 0pt]
\draw (0,-1) node{\small $\bullet$} arc(-90:90:1) node[near end, sloped]{\small $>$} node{\small $\bullet$};
\draw (0,-1) -- (-1,-1);
\draw (0,1) -- (-1,1);
\draw[thick] (-1,-0.3) -- (0.94,-0.3)node[below right]{\small $-\eta_m.C(\eta_m)$};
\draw[thick] (-1,0) -- (1,0)node[yshift=2pt,right=10pt]{\footnotesize $\vdots$};
\draw[thick] (-1,0.3) -- (0.94,0.3)node[above right]{\small $-\eta_1.C(\eta_1)$};
\node[fill=white, minimum width = 1cm, minimum height = 0.7cm, leftymissy, inner sep = 1.5pt] (X) at (-0.5,0){$\alpha$};
\end{tikzpicture} = \begin{tikzpicture}[scale = 0.7, baseline = 0pt]
\draw (0,-1) node{\small $\bullet$} arc(-90:90:1) node[near end, sloped]{\small $<$} node{\small $\bullet$};
\draw (0,-1) -- (-1,-1);
\draw (0,1) -- (-1,1);
\draw[thick] (-0.94,0.3) -- (0.3,0.3) .. controls (0.5,0.3) and (0.74,-0.3).. (0.94,-0.3)node[below right]{\small $\halft_V^{-1}(\eta_1)$}node[pos = 0.5, circle, fill = white, scale = 0.66]{};
\draw[thick] (-1,0) -- (1,0)node[yshift=2pt,right=5pt]{\footnotesize $\vdots$} node[pos = 0.81, circle, fill = white, scale = 0.46]{};
\draw[thick] (-0.94,-0.3) -- (0.3,-0.3) .. controls (0.5,-0.3) and (0.74,0.3).. (0.94,0.3)node[above right]{\small $\halft_V^{-1}(\eta_m)$} ;
\node[fill=white, minimum width = 1cm, minimum height = 0.7cm, leftymissy, inner sep = 1.5pt] (X) at (-0.5,0){$\alpha$};
\end{tikzpicture}.
\]
It is well-defined by \cite[Proposition~2.7]{Cos}. Note that ${\rm inv}_e$ is neither an algebra morphism nor a~comodule morphism.
\end{Definition}

\begin{Proposition}
Let $e_r$ be the right edge of the bigon and $\alpha \in \mathscr{S}(B)$, then $t(\alpha) = \varepsilon \circ {\rm inv}^{-1}_{e_r}(\alpha)$.
In~particular, by Remark $\ref{rmkFormActOnSS}$, for $\mathfrak{S}$ a marked surface with a right edge~$e$, the half twist acts on~$\mathscr{S}(\mathfrak{S})$ as~${\rm inv}_{e}^{-1}$.
\end{Proposition}

\begin{proof}
Indeed,
\begin{gather*}
t\circ {\rm inv}_{e_r} (\alpha)=t
\raisebox{-3pt}{$\left(\rule{0pt}{24pt}\right.$}\!\!
\begin{tikzpicture}[scale = 0.7, baseline = 0pt]
\draw (0,-1) node{\small $\bullet$} arc(-90:90:1) node[near end, sloped]{\small $<$} node{\small $\bullet$};
\draw (0,-1) arc(-90:-270:1) node[near end, sloped]{\small $>$};
\draw[thick] (-0.94,0.3) -- (0.3,0.3) .. controls (0.5,0.3) and (0.74,-0.3).. (0.94,-0.3)node[below right]{\small $\halft_V^{-1}(\eta_1)$}node[pos = 0.5, circle, fill = white, scale = 0.66]{};
\draw[thick] (-1,0) -- (1,0)node[yshift=2pt,right=5pt]{\footnotesize $\vdots$} node[pos = 0.81, circle, fill = white, scale = 0.46]{};
\draw[thick] (-0.94,-0.3) -- (0.3,-0.3) .. controls (0.5,-0.3) and (0.74,0.3).. (0.94,0.3)node[above right]{\small $\halft_V^{-1}(\eta_m)$} ;
\node[fill=white, draw, circle, scale = 1.2, inner sep = 1.5pt] (X) at (-0.3,0){$\alpha$};
\end{tikzpicture}\!\!
\raisebox{-3pt}{$\left.\rule{0pt}{24pt}\right)$}
=\varepsilon \raisebox{-3pt}{$\left(\rule{0pt}{24pt}\right.$}\!\!
\begin{tikzpicture}[scale = 0.7, baseline = 0pt]
\draw (0,-1) node{\small $\bullet$} arc(-90:90:1) node[near end, sloped]{\small $<$} node{\small $\bullet$};
\draw (0,-1) arc(-90:-270:1) node[near end, sloped]{\small $>$};
\draw[thick] (-0.94,0.3) -- (-0.1,0.3) .. controls (0.15,0.3) and (0.15,-0.3).. (0.4,-0.3);
\draw[thick] (0.4,-0.3) .. controls (0.65,-0.3) and (0.65,0.3).. (0.9,0.3) -- (0.94,0.3)node[above right]{\small $\halft_V\circ\halft_V^{-1} (\eta_1)$};
\draw[thick] (-1,0) -- (0,0) -- (1,0)node[yshift=2pt,right]{\footnotesize $\vdots$} node[pos = 0.15, circle, fill = white, scale = 0.46]{} node[pos = 0.65, circle, fill = white, scale = 0.46]{};
\draw[thick] (-0.94,-0.3) -- (-0.1,-0.3) .. controls (0.15,-0.3) and (0.15,0.3).. (0.4,0.3);
\draw[thick] (0.4,0.3) .. controls (0.65,0.3) and (0.65,-0.3).. (0.9,-0.3) -- (0.94,-0.3)node[below right]{\small $\halft_V\circ\halft_V^{-1} (\eta_m)$};
\node[rectangle, fill=white, draw = black, minimum height = 0.7cm, inner sep = 1.5pt] (X) at (-0.5,0){$\alpha$};
\end{tikzpicture}\!\!
\raisebox{-3pt}{$\left.\rule{0pt}{24pt}\right)$} = \varepsilon(\alpha).
\tag*{\qed}
\end{gather*}
\renewcommand{\qed}{}
\end{proof}

\begin{Remark}
In \cite[Section~3.4]{Cos} the counit is defined as $\varepsilon = i_* \circ {\rm inv}_{e_r}$, where $i$ is the inclusion of the bigon in the monogon. Surprisingly enough, the half coribbon functional is then $t = i_*$ and is simpler to write. This suggests that there is a half twist built in the construction of stated skein algebras. We~claim this comes from the passage from right to left comodule structure.
\end{Remark}
If $A$ is a right comodule over a Hopf algebra $H$, it is naturally a left comodule with $\Delta^L = \operatorname{fl} \circ ({\rm Id}_A \otimes S) \circ \Delta$, and we will consider these two as the ``same'' comodule. Indeed, one has an isomorphism of categories $(-)^L\colon H\text{-}{\rm comod} \to {\rm comod}\text{-}H$ which is the identity on vector spaces, switches the action as above, and is the identity on morphisms. However, when passing from right to left edges -- and comodule structures -- on stated skein algebras, one uses another way to see a right comodule $A$ as a left, namely with $\Delta^l = \operatorname{fl} \circ ({\rm Id}_A \otimes \operatorname{rot}_*) \circ \Delta$. Again one has $(-)^l\colon H\text{-}{\rm comod} \to comod\text{-}H$ with the identity on morphisms. So we have two functors $(-)^L$ and $(-)^l$ and we claim that the difference between them is precisely a half twist:

\begin{Proposition}\label{prophtDiffLl}
One has $(-)^l = (-)^L \circ (-)^\halft$ and $(-)^L = (-)^l \circ (-)^\halft$. Equivalently the map $\halft_A^l \colon A^l \to A^L$ is an isomorphism of left $\Oq$-comodules.
\end{Proposition}

\begin{proof}
All these functors are the identity on morphisms and only change the comodule structure. The map $\halft_A^l\colon A^l \to \big(A^\halft\big)^l$ is just $\halft_A$ as a map of vector spaces and is an isomorphism of vector spaces. The comodule structure on $\big(A^\halft\big)^l$ is the unique so that $\halft_A^l$ is a~comodule morphism. We~show that it is a comodule morphism $A^l\to A^L$, and hence that $A^L = \big(A^\halft\big)^l$. Let $a \in A$, one compares $\Delta^L \circ \halft_A (a) = \Delta^L(a_{(1)} \otimes t(a_{(2)})) = Sa_{(2)}.t(a_{(3)}) \otimes a_{(1)}$ and $\big({\rm Id} \otimes \halft_A\big) \circ \Delta^l = ({\rm Id} \otimes \halft_A)(\operatorname{rot}_*(a_{(2)})\otimes a_{(1)}) = \operatorname{rot}_*(a_{(3)}).t(a_{(2)}) \otimes a_{(1)}$. We show directly that for $\beta \in \mathscr{S}(B)$, $S\beta_{(1)}.t(\beta_{(2)}) = \operatorname{rot}_*(\beta_{(2)}).t(\beta_{(1)})$:
\begin{gather*}
S \circ ({\rm Id} \otimes t) \circ \Delta (\beta)
\\ \qquad
{}=S\raisebox{-3pt}{$\left(\rule{0pt}{24pt}\right.$}\!\!
\begin{tikzpicture}[scale = 0.7, baseline = 0pt]
\draw (0,-1) node{\small $\bullet$} arc(-90:90:1) node[near end, sloped]{\small $<$} node{\small $\bullet$};
\draw (0,-1) arc(-90:-270:1) node[near end, sloped]{\small $>$};
\draw[thick] (-0.94,-0.3)node[below left]{\small $\varepsilon_n$} -- (0.3,-0.3) .. controls (0.5,-0.3) and (0.74,0.3).. (0.94,0.3)node[above right]{\small $-\eta_m.C(-\eta_m)^{-1}$} node[pos = 0.5, circle, fill = white, scale = 0.66]{};
\draw[thick] (-1,0) node[yshift=2pt, left]{\footnotesize $\vdots$} -- (1,0)node[yshift=2pt,right=5pt]{\footnotesize $\vdots$} node[pos = 0.81, circle, fill = white, scale = 0.46]{};
\draw[thick] (-0.94,0.3) node[above left]{\small $\varepsilon_1$} -- (0.3,0.3) .. controls (0.5,0.3) and (0.74,-0.3).. (0.94,-0.3)node[below right]{\small $-\eta_1.C(-\eta_1)^{-1}$};
\node[circle, fill=white, draw = black, scale = 1.2, inner sep = 1.5pt] (X) at (-0.2,0){$\beta$};
\end{tikzpicture}\!\!\!
\raisebox{-3pt}{$\left.\rule{0pt}{24pt}\right)$}
= \raisebox{-3pt}{$\left(\rule{0pt}{24pt}\right.$}\!\!
\begin{tikzpicture}[scale = 0.7, baseline = 0pt, rotate = 180]
\draw (0,-1) node{\small $\bullet$} arc(-90:90:1) node[near start, sloped]{\small $>$} node{\small $\bullet$};
\draw (0,-1) arc(-90:-270:1) node[near start, sloped]{\small $<$};
\draw[thick] (-0.94,-0.3)node[above right]{\small $-\varepsilon_n.C(\varepsilon_n)^{-1}$} -- (0.3,-0.3) .. controls (0.5,-0.3) and (0.74,0.3).. (0.94,0.3)node[below left]{\small $\eta_m.C(-\eta_m)^{-1}.C(-\eta_m)$} node[pos = 0.5, circle, fill = white, scale = 0.66]{};
\draw[thick] (-1,0) node[yshift=2pt, right=10pt]{\footnotesize $\vdots$} -- (1,0)node[yshift=2pt,left=10pt]{\footnotesize $\vdots$} node[pos = 0.81, circle, fill = white, scale = 0.46]{};
\draw[thick] (-0.94,0.3) node[below right]{\small $-\varepsilon_1.C(\varepsilon_1)^{-1}$} -- (0.3,0.3) .. controls (0.5,0.3) and (0.74,-0.3).. (0.94,-0.3)node[above left]{\small $\eta_1.C(-\eta_1)^{-1}.C(-\eta_1)$};
\node[circle, fill=white, draw = black, scale = 1.2, inner sep = 1.5pt, rotate = 180] (X) at (-0.2,0){$\beta$};
\end{tikzpicture}\!\!\!
\raisebox{-3pt}{$\left.\rule{0pt}{24pt}\right)$}
\\ \qquad
{}=\raisebox{-3pt}{$\left(\rule{0pt}{24pt}\right.$}\!\!
\begin{tikzpicture}[scale = 0.7, baseline = 0pt]
\draw (0,-1) node{\small $\bullet$} arc(-90:90:1) node[near start, sloped]{\small $<$} node{\small $\bullet$};
\draw (0,-1) arc(-90:-270:1) node[near start, sloped]{\small $>$};
\draw[thick] (-0.94,-0.3)node[below left]{\small $\eta_1$} -- (0.3,-0.3) .. controls (0.5,-0.3) and (0.74,0.3).. (0.94,0.3)node[above right]{\small $-\varepsilon_1.C(\varepsilon_1)^{-1}$} node[pos = 0.5, circle, fill = white, scale = 0.66]{};
\draw[thick] (-1,0) node[yshift=2pt, left]{\footnotesize $\vdots$} -- (1,0)node[yshift=2pt,right=5pt]{\footnotesize $\vdots$} node[pos = 0.81, circle, fill = white, scale = 0.46]{};
\draw[thick] (-0.94,0.3) node[above left]{\small $\eta_m$} -- (0.3,0.3) .. controls (0.5,0.3) and (0.74,-0.3).. (0.94,-0.3)node[below right]{\small $-\varepsilon_n.C(\varepsilon_n)^{-1}$};
\node[circle, fill=white, draw = black, scale = 1.2, inner sep = 1.5pt, rotate = 180] (X) at (-0.2,0){$\beta$};
\end{tikzpicture}\!\!
\raisebox{-3pt}{$\left.\rule{0pt}{24pt}\right)$}
\overset{\operatorname{rot}_*^2={\rm Id}}=
\operatorname{rot}_* \raisebox{-3pt}{$\left(\rule{0pt}{24pt}\right.$}\!\!
 \begin{tikzpicture}[scale = 0.7, baseline = 0pt, rotate = 180]
\draw (0,-1) node{\small $\bullet$} arc(-90:90:1) node[near start, sloped]{\small $>$} node{\small $\bullet$};
\draw (0,-1) arc(-90:-270:1) node[near start, sloped]{\small $<$};
\draw[thick] (-0.94,-0.3)node[above right]{\small $\eta_1$} -- (0.3,-0.3) .. controls (0.5,-0.3) and (0.74,0.3).. (0.94,0.3)node[below left]{\small $-\varepsilon_1.C(\varepsilon_1)^{-1}$} node[pos = 0.5, circle, fill = white, scale = 0.66]{};
\draw[thick] (-1,0) node[yshift=2pt, right]{\footnotesize $\vdots$} -- (1,0)node[yshift=2pt,left=10pt]{\footnotesize $\vdots$} node[pos = 0.81, circle, fill = white, scale = 0.46]{};
\draw[thick] (-0.94,0.3) node[below right]{\small $\eta_m$} -- (0.3,0.3) .. controls (0.5,0.3) and (0.74,-0.3).. (0.94,-0.3)node[above left]{\small $-\varepsilon_n.C(\varepsilon_n)^{-1}$};
\node[circle, fill=white, draw = black, scale = 1.2, inner sep = 1.5pt] (X) at (-0.2,0){$\beta$};
\end{tikzpicture}\!\!
\raisebox{-3pt}{$\left.\rule{0pt}{24pt}\right)$}
\\ \qquad
{}= \operatorname{rot}_* \circ (t \otimes {\rm Id}) \circ \Delta(\beta).
\tag*{\qed}
\end{gather*}
\renewcommand{\qed}{}
\end{proof}

\begin{Proposition}\label{propInvHHO}
Given two right $\Oq$-comodules $A$ and $B$ one has $ (A \otimes B)^{{\rm inv}} = HH^0\big(A^L \otimes B\big) \allowbreak= HH^0\big(\big(A^\halft\big)^l \otimes B\big)$.
\end{Proposition}

\begin{proof}
The first equality is true in any Hopf algebra by a direct computation, and the second is just the above proposition.
\end{proof}

\subsection{The general relation}
We can now express the full correspondence between stated skein algebras and internal skein algebras, with both right and left boundary edges.
On a single edge, by Remark~\ref{rmkASigmaRnaive} one gets $A_\Sigma^R = A_\Sigma^\halft = \mathscr{S}(\Sigma)^\halft$ equipped with the natural isomorphism $(\sigma_{(-)^\halft})^\halft$, using the half twist in $\Oqcom$. More precisely, for $\alpha \in \Hom_{{\rm Sk}_\mathcal{V}(\mathfrak{S})}(\varnothing\lhd V,\varnothing)$ one has
\[
\sigma_{V^\halft}(\alpha)^\halft = \halft_{A_\Sigma^\halft}^{-1} \circ \sigma_{V^\halft}(\alpha) \circ \halft_V = \halft_{A_\Sigma^\halft}^{-1} \circ \sigma_{V}(\alpha \circ (\halft_V \rhd\varnothing)) = \halft_{A_\Sigma^\halft}^{-1} \circ \sigma^R_{V}(\alpha).
\]
As in Remark \ref{rmkLinkTwoAsigmaR} its algebra structure is
\begin{align*}
\halft_{A_\Sigma^\halft}^{-1} \circ m^R \circ \halft_{A_\Sigma^\halft} \otimes \halft_{A_\Sigma^\halft} & = \halft_{A_\Sigma^\halft}^{-1} \circ m \circ c_{A_\Sigma \otimes A_\Sigma} \circ \halft_{A_\Sigma^\halft} \otimes \halft_{A_\Sigma^\halft} \\
& = \halft_{A_\Sigma^\halft}^{-1} \circ m \circ \operatorname{fl} \circ \halft_{A_\Sigma \otimes A_\Sigma} = m^\halft \circ \operatorname{fl},
\end{align*}
so $m^{\rm op}$ as maps of vector spaces.

Note that because $S$ is an anti-algebra morphism, the functor $(-)^L$ is (almost strictly) anti-monoidal (like the half twist, it is the identity on vector spaces but is anti-monoidal on the comodule structure) namely $\operatorname{fl}\colon (V \otimes W)^L \to W^L \otimes V^L$ is an isomorphism of left $H$-comodules. Thus a right $H$-comodule algebra $A$ induces a left $H$-comodule algebra $A^L$ with $A^L \otimes A^L \overset{\operatorname{fl}}\simeq (A\otimes A)^L \overset{m}\to A^L$, namely with product $m^{\rm op}$.
When one has multi-edges one can switch the $i$-th right $\Oq$-comodule structure to a left using either $S$ or $\operatorname{rot}_*$ and we denote the associated functors by $(-)^{L_i}$ and $(-)^{l_i}$, one can take opposite product on the $i$-th coordinate which we denote by $m^{{\rm op}_i}$, and there are half twists on each coordinates, which we denote by $\halft_i$.

\begin{Theorem}\label{thmRelationMultiAndRight}
Let $\mathfrak{S}$ be a marked surface with $n$ boundary edges labelled either as left $($num\-be\-red~$1$ to~$k)$ or as right $($numbered $k+1$ to $n)$ edges. There is an isomorphism of $\big(\Oq^{\otimes k},\allowbreak \Oq^{\otimes n-k}\big)$-bicomodules algebras $A_{\mathfrak{S}}^{L_{k+1},\dots,L_n} \simeq \mathscr{S}(\mathfrak{S})$.
\end{Theorem}
\begin{proof}
To avoid confusion we denote the stated skein algebra of the marked surface $\mathfrak{S}$ by $\mathscr{S}^R(\mathfrak{S})$ when it is seen as a right $\Oq^{\otimes n}$-comodule and by~$\mathscr{S}(\mathfrak{S})= \mathscr{S}^R(\mathfrak{S})^{l_{k+1},\dots,l_n}$ when it is seen as an $\big(\Oq^{\otimes k},\Oq^{\otimes n-k}\big)$-bicomodule. We denote by $m$ its product, which is the same in both cases.
By Theorem~\ref{thmRelationMulti}, $\mathscr{S}^R(\mathfrak{S})$ is the internal skein algebra of $\mathfrak{S}$ with every edge labelled as left, and then by Remark \ref{rmkASigmaRnaive} on coordinates $k+1,\dots,n$ one may take $A_\mathfrak{S} := \mathscr{S}^R(\mathfrak{S})^{\halft_{k+1},\dots,\halft_n}$ as algebra object in $\mathcal{E}^{\boxtimes n}$ with skew monoidal structure ${\bmotimes}$. The algebra structure on $\mathscr{S}^R(\mathfrak{S})^{\halft_{k+1},\dots,\halft_n}$ is $m^{{\rm op}_{k+1},\dots,{\rm op}_n}$.
Thus by Proposition~\ref{prophtDiffLl},
\[
A_\mathfrak{S}^{L_{k+1},\dots,L_n} := \big(\mathscr{S}^R(\mathfrak{S})^{\halft_{k+1},\dots,\halft_n}\big)^{L_{k+1},\dots,L_n} = \mathscr{S}^R(\mathfrak{S})^{l_{k+1},\dots,l_n}=\mathscr{S}(\mathfrak{S})
\]
 as $\big(\Oq^{\otimes k},\Oq^{\otimes n-k}\big)$-bi\-co\-modu\-les, and the algebra structure on \[
 \big(\mathscr{S}^R(\mathfrak{S})^{\halft_{k+1},\dots,\halft_n}\big)^{L_{k+1},\dots,L_n} \qquad \text{is} \qquad (m^{{\rm op}_{k+1},\dots,{\rm op}_n})^{{\rm op}_{k+1},\dots,{\rm op}_n}=m.\tag*{\qed}
 \]\renewcommand{\qed}{}
\end{proof}

\begin{Remark}
A nice miracle with stated skein algebras is that the quantum group $\Oq$, which is used to define the tangle invariants used to define stated skein algebras, is re-obtained as the stated skein algebra of the bigon. One can see why this should be true in internal skein algebras. By~Defi\-ni\-tion~\ref{defMultiEdgeASigma}, the internal skein algebra of the bigon is an object $A_B \in \Oq^{\otimes 2}\text{-}{\rm comod}$ together with a~natural isomorphism
\begin{align*}
\Hom_{{\rm SK}(\mathbb{R}^2)}((X,Y)\rhd\varnothing,\varnothing) & = \Hom_{\Oqcom}(X\otimes Y,k) \\
& \ \tilde{\to} \Hom_{\Oq^{\otimes 2}\text{-}{\rm comod}}(X \otimes Y,A_B)
\end{align*}
 for $X,Y\in\Oqcom$. We set $A_B = \Oq$ with usual first right comodule structure $\Delta_1=\Delta$ and with second comodule strucure its left one switched using $L_2^{-1}$ namely $\Delta_2=\operatorname{fl}\circ\big(S^{-1}\otimes {\rm Id}\big)\circ \Delta$. The demanded isomorphism is given by $f\mapsto \tilde{f}$ where \[
 \tilde{f}(x\otimes y)= x_{(2)}. f(x_{(1)}\otimes y) = S(y_{(2)}). f(x\otimes y_{(1)}).
 \]
  Its inverse is given by $\tilde{f}\mapsto \varepsilon \circ \tilde{f}$.
\end{Remark}

\begin{Remark}\label{rmkPbComRightActLeft}
Despite this theorem, it is still annoying that in the simplest case one wants to see the boundary at the right for stated skein algebras and at the left for internal skein algebras. This should be solvable by considering the category of left (instead of right) $\Oq$-comodules as coefficients, so it is a minor issue.
\end{Remark}

\subsection{Excision properties of multi-edges internal skein algebras} Let $\mathfrak{S}_1 \hookleftarrow C \hookrightarrow \mathfrak{S}_2$ be a right and a left thick embeddings in two marked surfaces and $\mathfrak{S}$ their collar gluing. Namely, $C$ embeds as a sequence of $k$ right boundary edges $\vec{c}_1$ of $\partial \mathfrak{S}_1$ and as $k$ left boundary edges $\vec{c}_2$ of $\partial \mathfrak{S}_2$, and $\mathfrak{S}$ is the gluing $\mathfrak{S}_1 \cup_{\vec{c}_1 = \vec{c}_2} \mathfrak{S}_2$. We show how to compute~$A_\mathfrak{S}$ from $A_{\mathfrak{S}_1}$ and $A_{\mathfrak{S}_2}$.
The general idea goes as follows. In the case where $\mathfrak{S}_1$ and $\mathfrak{S}_2$ both have a single boundary edge, so $k=1$, one wants to describe endomorphisms $\alpha$ of the empty set in ${\rm Sk}_\mathcal{V}(\mathfrak{S})$. By Corollary \ref{corrMorphSkRecoll} they are described by a morphism $\alpha_1\colon \varnothing \to \varnothing \lhd V$ in ${\rm Sk}_\mathcal{V}(\mathfrak{S}_1)$ and a morphism $\alpha_2\colon V \rhd \varnothing \to \varnothing$ in ${\rm Sk}_\mathcal{V}(\mathfrak{S}_2)$, linked by an isomorphism $\iota_V \colon \varnothing \lhd V \to V \rhd \varnothing$, which is just a slanted skein crossing over $C$ in ${\rm Sk}_\mathcal{V}(\mathfrak{S})$, see the idea of proof of Theorem~\ref{thExcSkCat}. One can reconstruct $\alpha$ as $\alpha = ({\rm Id}_\varnothing , \alpha_2) \circ \iota_V \circ (\alpha_1 , {\rm Id}_\varnothing)$. The morphisms $\alpha_{1,2}$ are well defined up to balancing, namely naturality of $\iota$. Now, by definition of $A_{\mathfrak{S}_1}$ and $A_{\mathfrak{S}_2}$, they are described by some $f_1 \in \Hom_{\mathcal{E}}(1_\mathcal{V},A_{\mathfrak{S}_1}\otimes V)$ and $f_2\in\Hom_\mathcal{E}(V,A_{\mathfrak{S}_2})$. Composing them mimicking the reconstruction of $\alpha$ gives a morphism $f=({\rm Id}_{A_{\mathfrak{S}_1}}\otimes f_2)\circ f_1\colon 1_\mathcal{V} \to A_{\mathfrak{S}_1}\otimes A_{\mathfrak{S}_2}$, namely an invariant inside $A_{\mathfrak{S}_1}\otimes A_{\mathfrak{S}_2}$. This suggests $A_\mathfrak{S} \simeq (A_{\mathfrak{S}_1}\otimes A_{\mathfrak{S}_2})^{{\rm inv}}$, which we will prove below.
Now we need to define what we mean by invariants of a tensor product in any ribbon category~$\mathcal{V}$.
\begin{Definition} \label{defInvRibCat}
Let $\mathcal{V}$ be a ribbon category, $\mathcal{E} = \Free(\mathcal{V})$ and $n \geq 2$. For $1\leq i <j \leq n-1$ we denote the tensor product of coordinates $i$ and $j$ by
\[
\otimes_{i,j}\colon\ \begin{cases}
\mathcal{V}^{\otimes n}  \to  \mathcal{V}^{\otimes n-1},
\\
(V_1,\dots,V_n)  \mapsto  (V_1,\dots, V_{i-1}, V_i \otimes V_j, V_{i+1},\dots, V_{j-1},V_{j+1},\dots,V_n).\end{cases}
\]
For $\vec{k}_1<\vec{k}_2$ two sequences of $k$ distinct indices we denote the tensor product of coordinates~$\vec{k}_1$ with coordinates~$\vec{k}_2$ by $\otimes_{\vec{k}_1,\vec{k}_2} \colon \mathcal{V}^{\otimes n}\to \mathcal{V}^{\otimes n-k}$. It extends to $\otimes_{\vec{k}_1,\vec{k}_2}\colon \mathcal{E}^{\boxtimes n} \to \mathcal{E}^{\boxtimes n-k}$ by cocontinuity.
We~denote the unit on $i$-th coordinate by
\[
\eta_i \colon \begin{cases}
\mathcal{V}^{\otimes n-2} \to \mathcal{V}^{\otimes n-1},
\\
\big(V_1,\dots, \overset{\vee}{V_i},\dots, \overset{\vee}{V_j},\dots,V_n\big) \mapsto\big(V_1,\dots,1_\mathcal{V},\dots,\overset{\vee}{V_j},\dots,V_n\big)
\end{cases}
\]
and the unit on $\vec{k}_1$-th coordinates as $\eta_{\vec{k}_1}\colon \mathcal{V}^{\otimes n-2k}\to \mathcal{V}^{\otimes n-k}$.

Let $X \in \mathcal{E}^{\boxtimes n}$, we want to define its $\big(\vec{k}_1,\vec{k}_2\big)$-invariants $X^{{\rm inv}_{\vec{k}_1,\vec{k}_2}} \in \mathcal{E}^{\boxtimes n-2k}$. One only needs to describe morphisms from any $\vec{V} \in \mathcal{V}^{\otimes n-2k}$ to it. We set
\[
X^{{\rm inv}_{\vec{k}_1,\vec{k}_2}}\big(\vec{V}\big) = \Hom_{\mathcal{E}^{\boxtimes n-2k}}\big(\vec{V},X^{{\rm inv}_{\vec{k}_1,\vec{k}_2}}\big) := \Hom_{\mathcal{E}^{\boxtimes n-k}}\big(\eta_{\vec{k}_1}(\vec{V}),\otimes_{\vec{k}_1,\vec{k}_2}(X)\big).
\]
 For $X \in \mathcal{E}^{\boxtimes n_1}\boxtimes \mathcal{E}^{\boxtimes \vec{k}_1}$ and $Y \in \mathcal{E}^{\boxtimes \vec{k}_2}\boxtimes \mathcal{E}^{\boxtimes n_2}$ we write $X \otimes_{\vec{k}_1,\vec{k}_2} Y := \otimes_{\vec{k}_1,\vec{k}_2}(X,Y)$. Then
 \begin{align*}
(X,Y)^{{\rm inv}_{\vec{k}_1,\vec{k}_2}}((V_{\vec{n}_1},V_{\vec{n}_2})) :&= \Hom_{\mathcal{E}^{\boxtimes n_1+n_2+k}}\big((V_{\vec{n}_1},1_{\mathcal{V}^{\otimes k}},V_{\vec{n}_2}),X\otimes_{\vec{k}_1,\vec{k}_2}Y\big)
\\
&= \Hom_{\mathcal{E}^{\boxtimes n_1+n_2+k}}\big((V_{\vec{n}_1},1_{\mathcal{V}^{\otimes k}})\otimes_{\vec{k}_1,\vec{k}_2}(1_{\mathcal{V}^{\otimes k}} ,V_{\vec{n}_2}),X\otimes_{\vec{k}_1,\vec{k}_2}Y\big).
 \end{align*}
\end{Definition}

For $\mathcal{V} = \Oqcomfin$ we get a notion of invariants for bicomodules $(X_i,X_j)\in \mathcal{V}^{\otimes 2}$ where we first ``merge'' the two comodule structures (by the product, in the definition of the tensor product) and then take invariants in the usual sense, namely maps $k \to X_i \otimes X_j$.
\begin{Theorem}\label{thmExcisionAFS}
Let $\mathfrak{S}_1$ be a marked surface with $n_1+k$ boundary edges with a sequence of $k$ right boundary edges $\vec{c}_1$ $($numbered $\vec{k}_1=\{n_1+1,\dots,n_1+k\})$ and $\mathfrak{S}_2$ a marked surface with $n_2+k$ boundary edges with a sequence of $k$ left boundary edges $\vec{c}_2$ $($numbered $\vec{k}_2=\{ n_1+k+1,\dots,n_1+2k\})$. We write $\vec{n}_1 = \{1,\dots,n_1\}$ and $\vec{n}_2 = \{n_1+2k+1,\dots,n_1+2k+n_2\}$ the indices of the other edges of $\mathfrak{S}_1$ and~$\mathfrak{S}_2$. Let $\mathfrak{S} = \mathfrak{S}_1 \cup_{\vec{c}_1=\vec{c}_2}\mathfrak{S}_2$, then one has an isomorphism $A_\mathfrak{S} \simeq (A_{\mathfrak{S}_1}, A_{\mathfrak{S}_2})^{{\rm inv}_{\vec{k}_1,\vec{k}_2}}$ in $\mathcal{E}^{\boxtimes n_1+n_2}$.
Note that one has two thick embeddings $\mathfrak{S}_1 \hookleftarrow C \hookrightarrow \mathfrak{S}_2$ where $C = \sqcup^k (0,1)$ and $\mathfrak{S}$ is their collar gluing.
\end{Theorem}

\begin{proof}
We describe a natural isomorphism between functors $\mathcal{V}^{\otimes n_1+n_2} \to {\rm Vect}_k$
\begin{gather*}
\begin{split}
\sigma\colon\ \Hom_{{\rm Sk}_\mathcal{V}(\mathfrak{S})}(-\bmrhd\varnothing,\varnothing) & \ \tilde{\Rightarrow}\ \Hom_{\mathcal{E}^{\boxtimes n_1+n_2}}\big({-},(A_{\mathfrak{S}_1}, A_{\mathfrak{S}_2})^{{\rm inv}_{\vec{k}_1,\vec{k}_2}}\big)
\\
:&=\Hom_{\mathcal{E}^{\boxtimes n_1+n_2+k}}\big(\eta_{\vec{k}_1}(-),A_{\mathfrak{S}_1} \otimes_{\vec{k}_1,\vec{k}_2} A_{\mathfrak{S}_2}\big).
\end{split}
\end{gather*}
We write $(\sigma_1)_{(V_{\vec{n}_1},V_{\vec{k}_1})}\colon \Hom_{{\rm Sk}_\mathcal{V}(\mathfrak{S}_1)}(V_{\vec{n}_1}\bmrhd \varnothing,\varnothing\lhd V_{\vec{k}_1}) \,\tilde{\Rightarrow}\, \Hom_{\mathcal{E}^{\boxtimes n_1+k}}((V_{\vec{n}_1},1_{\mathcal{V}^{\otimes k}}), A_{\mathfrak{S}_1}\otimes_{\vec{k}_1} V_{\vec{k}_1})$ for $V_{\vec{n}_1} \in \mathcal{V}^{\otimes n_1}$ and $V_{\vec{k}_1} \in \mathcal{V}^{\otimes k}$, obtained from the defining natural isomorphism of $A_{\mathfrak{S}_1}$ by Remark \ref{rmkTopPtsMultiEdge}. We~write $\sigma_2\colon \Hom_{{\rm Sk}_\mathcal{V}(\mathfrak{S}_2)}(-\bmrhd\varnothing,\varnothing) \,\tilde{\Rightarrow}\, \Hom_{\mathcal{E}^{\boxtimes n_2+k}}(-,A_{\mathfrak{S}_2})$ the defining natural isomorphism of $A_{\mathfrak{S}_2}$.

\medskip\noindent
\textbf{Step 1} (decomposition in $\mathfrak{S}$). Let $ \vec{V}=(V_{\vec{n}_1},V_{\vec{n}_2}) \in \mathcal{V}^{\otimes n_1+n_2}$ and $\alpha \in \Hom_{{\rm Sk}_\mathcal{V}(\mathfrak{S})}\big(\vec{V}\bmrhd\varnothing,\varnothing\big)$ a~morphism from $V_{\vec{n}_1}\bmrhd\varnothing$ in $\mathfrak{S}_1$ and~$V_{\vec{n}_2} \bmrhd \varnothing$ in~$\mathfrak{S}_2$ to the empty set in~$\mathfrak{S}$. By Corollary \ref{corrMorphSkRecoll}, $\alpha$ decomposes into a~pair $(\alpha_1,\alpha_2)$ as
\[
\alpha = ({\rm Id}_\varnothing , \alpha_2) \circ \iota_{\varnothing,V_{\vec{k}},V_{\vec{n}_2}\bmrhd\varnothing} \circ (\alpha_1 , {\rm Id}_{V_{\vec{n}_2}\bmrhd\varnothing})
\]
with $\alpha_1 \in \Hom_{{\rm Sk}_\mathcal{V}(\mathfrak{S}_1)}(V_{\vec{n}_1} \bmrhd \varnothing,\varnothing\lhd_{\vec{k}_1} V_{\vec{k}})$ and $\alpha_2\in\Hom_{{\rm Sk}_\mathcal{V}(\mathfrak{S}_2)}(V_{\vec{k}}\rhd_{\vec{k}_2}(V_{\vec{n}_2}\bmrhd\varnothing),\varnothing)$ for some $V_{\vec{k}} \in {\rm Sk}_\mathcal{V}(C \times (0,1)) \simeq \mathcal{V}^{\otimes k}$, with an implicit sum. This decomposition is unique up to balancing, namely if $\alpha_2$ can be written $\beta_2 \circ (\gamma\rhd {\rm Id}_{V_{\vec{n}_2}\rhd\varnothing})$, with $\beta_2\in\Hom_{{\rm Sk}_\mathcal{V}(\mathfrak{S}_2)}(W_{\vec{k}}\rhd (V_{\vec{n}_2} \bmrhd \varnothing),\varnothing)$ and $\gamma\in\Hom_{{\rm Sk}_\mathcal{V}(C\times (0,1))}(V_{\vec{k}},W_{\vec{k}})$ for some $W_{\vec{k}}\in {\rm Sk}_\mathcal{V}(C\times(0,1))$, then: \[(\alpha_1,\beta_2\circ (\gamma\rhd {\rm Id}_{V_{\vec{n}_2}\rhd\varnothing})) \sim (({\rm Id}_\varnothing\lhd\gamma) \circ \alpha_1, \beta_2).\]
$$
\begin{tikzpicture}[baseline = 2cm, xscale = 1.2, yscale = 3]
\fill[gray!20] (4.5,0.5) -- (3.5,0.5)-- (3.5,1.5) -- (4.5,1.5)--cycle;
\draw (4.5,0.5) -- (3.5,0.5)	.. controls (3,0.5 ) and (3,0.45 ) .. (2.5,0.4) coordinate[pos = 0.5] (EM);
\draw (4.5,0.5)	.. controls (5,0.5 ) and (5,0.45 ) .. (5.5,0.4)coordinate[pos = 0.5] (ENbel);
\draw (4.5,1.5) -- (3.5,1.5)	.. controls (3,1.5 ) and (3,1.55 ) .. (2.5,1.6);
\draw (5.5,1.6)	.. controls (5,1.55 ) and (5,1.5 ) .. (4.5,1.5);
\draw[dashed] (2.5,0.88) -- (5.5,0.88) coordinate[pos = 0.33] (LA) coordinate[pos = 0.833] (ENmidbel);
\draw[dashed] (2.5,1.12) -- (5.5,1.12) coordinate[pos = 0.66] (RA) coordinate[pos = 0.95] (ENmidtop);
\node[minimum width=1.2cm, leftymissy, outer sep = 0pt] (a1) at (3,0.65){$\alpha_1$};
\node[minimum width=1.2cm, rightymissy, outer sep = 0pt] (a2) at (5,1.35){$\alpha_2$};
\draw (EM) node{\small $\bullet$} node[below]{\small $V_{\vec{n}_1} \bmrhd\varnothing$} -- (a1);
\draw (ENbel) node{\small $\bullet$} node[below]{\small $V_{\vec{n}_2} \bmrhd\varnothing$} -- (ENmidbel) -- (ENmidtop) -- (a2);
\draw (a1)-- (LA) node{\tiny $\bullet$} node[below right = -1pt]{\small $\lhd V_{\vec{k}}$} -- (RA) node[midway, above left = -3pt]{\small $\iota$} node{\tiny $\bullet$} node[above left = -1pt]{\small $V_{\vec{k}} \rhd$} -- (a2);
\end{tikzpicture}
$$
\medskip\noindent
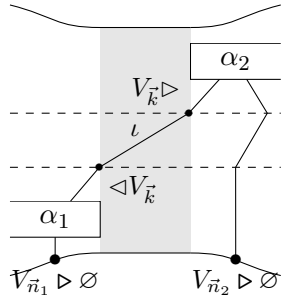
\textbf{Step 2} \big(re-composition in $\mathbb{R}^2$\big). The morphism $\alpha_1$ is described by a morphism \[
f_1 = (\sigma_1)_{(V_{\vec{n}_1},V_{\vec{k}})}(\alpha_1) \in \Hom_{\mathcal{E}^{\boxtimes n_1+k}}((V_{\vec{n}_1},1_{\mathcal{V}^{\otimes k}}), A_{\mathfrak{S}_1}\otimes_{\vec{k}_1} V_{\vec{k}})$ and $\alpha_2
\]
is described by \[
f_2=(\sigma_2)_{(V_{\vec{k}},V_{\vec{n}_2})}(\alpha_2)\in \Hom_{\mathcal{E}^{\boxtimes n_2+k}}((V_{\vec{k}},V_{\vec{n}_2}), A_{\mathfrak{S}_2}).
\]
These morphisms are well-defined (depend only on $\alpha$) up to balancing, namely if $\alpha_2=\beta_2 \circ (\gamma\rhd {\rm Id}_{V_{\vec{n}_2}\rhd\varnothing})$ with~$\beta_2$ described by $g_2 = \sigma_2(\beta_2)\in\Hom_{\mathcal{E}^{\boxtimes n_2+k}}((W_{\vec{k}},V_{\vec{n}_2}), A_{\mathfrak{S}_2})$, then by naturality of~$\sigma_1$ and $\sigma_2$, $f_2 = g_2 \circ (\gamma,{\rm Id}_{V_{\vec{n}_2}})$ and the above relation becomes $(f_1, g_2 \circ (\gamma,{\rm Id}_{V_{\vec{n}_2}})) \sim (({\rm Id}_{A_{\mathfrak{S}_1}}\otimes_{\vec{k}_1}\gamma) \circ f_1,g_2)$.
$$
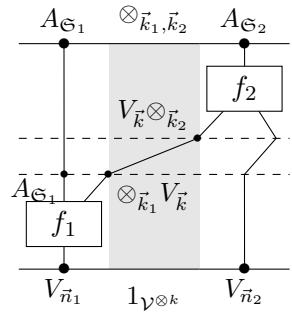
\begin{tikzpicture}[baseline = 2cm, xscale = 1.2, yscale = 3]
\fill[gray!20] (4.5,0.5) -- (3.5,0.5)-- (3.5,1.5) -- (4.5,1.5)--cycle;
\draw (4.5,0.5) -- (3.5,0.5)	.. controls (3,0.5 ) and (3,0.5 ) .. (2.5,0.5) coordinate[pos = 0.5] (EM);
\draw (4.5,0.5)	.. controls (5,0.5 ) and (5,0.5 ) .. (5.5,0.5)coordinate[pos = 0.5] (ENbel);
\draw (4.5,1.5) -- (3.5,1.5)	.. controls (3,1.5 ) and (3,1.5 ) .. (2.5,1.5) coordinate[pos = 0.5] (EMtop);
\draw (5.5,1.5)	.. controls (5,1.5 ) and (5,1.5 ) .. (4.5,1.5) coordinate[pos = 0.5] (ENtop);
\draw[dashed] (2.5,0.92) -- (5.5,0.92) coordinate[pos = 0.33] (LA) coordinate[pos = 0.833] (ENmidbel) coordinate[pos = 0.167] (EMmidbel);
\draw[dashed] (2.5,1.08) -- (5.5,1.08) coordinate[pos = 0.66] (RA) coordinate[pos = 0.95] (ENmidtop);
\node[minimum width=1cm, draw, rectangle, outer sep = 0pt] (a1) at (3,0.7){$f_1$};
\node[minimum width=1cm, draw, rectangle, outer sep = 0pt] (a2) at (5,1.3){$f_2$};
\draw (EM) node{\small $\bullet$} node[below]{\small $V_{\vec{n}_1}$} -- (a1);
\draw (ENbel) node{\small $\bullet$} node[below]{\small $V_{\vec{n}_2}$} -- (ENmidbel) -- (ENmidtop) -- (a2);
\draw (a1)-- (LA) node{\tiny $\bullet$} node[below right = -1pt]{\small $\otimes_{\vec{k}_1} V_{\vec{k}}$} -- (RA) node{\tiny $\bullet$} node[above left = -1pt]{\small $V_{\vec{k}} \otimes_{\vec{k}_2}$} -- (a2);
\draw (a1) -- (EMmidbel)node{\tiny $\bullet$} node[below left = -3pt]{\small $A_{\mathfrak{S}_1}$} -- (EMtop) node{\small $\bullet$} node[above]{\small $A_{\mathfrak{S}_1}$};
\draw (a2) -- (ENtop) node{\small $\bullet$} node[above]{\small $A_{\mathfrak{S}_2}$};
\node[above] at (4,1.5) {\small $\otimes_{\vec{k}_1,\vec{k}_2}$};
\node[below=3pt] at (4,0.5) {\small $1_{\mathcal{V}^{\otimes k}}$};
\end{tikzpicture}
$$
Thus the map{\samepage
\begin{gather*}
\sigma_{\vec{V}}(\alpha) := ({\rm Id}_{A_{\mathfrak{S}_1}} \otimes_{\vec{k}_1,\vec{k}_2} f_2) \circ (f_1 , {\rm Id}_{V_{\vec{n}_2}})\in\Hom_{\mathcal{E}^{\boxtimes n_1+n_2+k}}((V_{\vec{n}_1},1_{\mathcal{V}^{\otimes k}},V_{\vec{n}_2}),A_{\mathfrak{S}_1}\otimes_{\vec{k}_1,\vec{k}_2}A_{\mathfrak{S}_2})
\end{gather*}
 is well defined, because this relation is killed.}

\medskip \noindent
\textbf{Step 3} (naturalilty). Naturality is quite obvious from the picture: one can insert morphisms from below. For $g_1 \in \Hom_{\mathcal{E}^{\boxtimes n_1}}(W_{\vec{n}_1},V_{\vec{n}_1})$ and $g_2 \in \Hom_{\mathcal{E}^{\boxtimes n_2}}(W_{\vec{n}_2},V_{\vec{n}_2})$, by naturality of~$\sigma_1$ and~$\sigma_2$, one has $\sigma_1(\alpha_1 \circ (g_1 \bmrhd {\rm Id}_\varnothing)) = f_1 \circ g_1$ and $\sigma_2(\alpha_2 \circ (g_2\bmrhd {\rm Id}_\varnothing)) = f_2 \circ g_2$. Now $\alpha \circ ((g_1,g_2)\bmrhd {\rm Id}_\varnothing)$ splits in Step~1 as $({\rm Id}_\varnothing , \alpha_2) \circ \iota_{\varnothing,V_{\vec{k}},V_{\vec{n}_2}\bmrhd\varnothing} \circ (\alpha_1 , {\rm Id}_{V_{\vec{n}_2}\bmrhd\varnothing}) \circ (g_1 \bmrhd {\rm Id}_\varnothing,g_2\bmrhd {\rm Id}_\varnothing) = ({\rm Id}_\varnothing , \alpha_2\circ (g_2\bmrhd {\rm Id}_\varnothing)) \circ \iota_{\varnothing,V_{\vec{k}},W_{\vec{n}_2}\bmrhd\varnothing} \circ (\alpha_1\circ (g_1\bmrhd {\rm Id}_\varnothing), {\rm Id}_{W_{\vec{n}_2}\bmrhd\varnothing})$. Thus \[\sigma(\alpha \circ ((g_1,g_2)\bmrhd {\rm Id}_\varnothing)) = ({\rm Id}_{A_{\mathfrak{S}_1}} \otimes_{\vec{k}_1,\vec{k}_2} (f_2\circ g_2)) \circ ((f_1\circ g_1) , {\rm Id}_{W_{\vec{n}_2}}) = \sigma(\alpha) \circ (g_1,g_2).\]
We now construct an inverse to $\sigma$ by the same steps in reverse order.

\medskip \noindent
\textbf{Step $\boldsymbol{2^{-1}}$} \big(decomposition in $\mathbb{R}^2$\big). We want to decompose a morphism \[
f\in\Hom_{\mathcal{E}^{\boxtimes n_1+n_2+k}}\big(\eta_{\vec{k}_1}\big(\vec{V}\big),A_{\mathfrak{S}_1}\otimes_{\vec{k}_1,\vec{k}_2}A_{\mathfrak{S}_2}\big)\qquad \text{as}\quad
f = \big({\rm Id}_{A_{\mathfrak{S}_1}} \otimes_{\vec{k}_1,\vec{k}_2} f_2\big) \circ \big(f_1 , {\rm Id}_{V_{\vec{n}_2}}\big)
\]
with $f_1\in\Hom_{\mathcal{E}^{\boxtimes n_1+k}}((V_{\vec{n}_1},1_{\mathcal{V}^{\otimes k}}), A_{\mathfrak{S}_1}\otimes_{\vec{k}_1} V_{\vec{k}})$ and $f_2\in\Hom_{\mathcal{E}^{\boxtimes n_2+k}}((V_{\vec{k}},V_{\vec{n}_2}), A_{\mathfrak{S}_2})$.

This is easy in $\mathcal{V}^{\otimes n_1+n_2+k}$, as all maps split on each coordinates. For $\vec{A}_1=(A_{\vec{n}_1},A_{\vec{k}_1}) \in \mathcal{V}^{\otimes n_1+k}$ and $\vec{A}_2=(A_{\vec{k}_2},A_{\vec{n}_2}) \in \mathcal{V}^{\otimes n_2+k}$, a morphism $f \in \Hom_{\mathcal{V}^{\otimes n_1+n_2+k}}((V_{\vec{n}_1},1_{\mathcal{V}^{\otimes k}},V_{\vec{n}_2}),\allowbreak \vec{A}_1 \otimes_{\vec{k}_1,\vec{k}_2} \vec{A}_2)$ is, up to a linear combination, of the form $(g_{\vec{n}_1},g_{\vec{k}_1},g_{\vec{n}_2})$ with $g_{\vec{n}_1}\colon V_{\vec{n}_1} \to A_{\vec{n}_1}$, $g_{\vec{k}_1}\colon 1_{\mathcal{V}^{\otimes k}} \to A_{\vec{k}_1}\otimes_{\vec{k}_1,\vec{k}_2} A_{\vec{k_2}}$ and $g_{\vec{n}_2}\colon V_{\vec{n}_2}\to A_{\vec{n}_2}$.
Then, set $V_{\vec{k}} = A_{\vec{k}_2}$, $f_1 = g_{\vec{n}_1} \otimes g_{\vec{k}_1} \in \Hom_{\mathcal{V}^{\otimes n_1+k}}((V_{\vec{n}_1},1_{\mathcal{V}^{\otimes k}}), \allowbreak \vec{A}_1\otimes_{\vec{k}_1} V_{\vec{k}})$ and $f_2 = {\rm Id}_{V_{\vec{k}}} \otimes g_{\vec{n}_2} \in \Hom_{\mathcal{V}^{\otimes n_2+k}}((V_{\vec{k}},V_{\vec{n}_2}),\vec{A}_2)$, one has $f= ({\rm Id}_{\vec{A}_1} \otimes_{\vec{k}_1,\vec{k}_2} f_2) \circ (f_1 , {\rm Id}_{V_{\vec{n}_2}})$. This decomposition is unique up to balancing, if $f= ({\rm Id}_{\vec{A}_1} \otimes_{\vec{k}_1,\vec{k}_2} f_2') \circ (f_1' , {\rm Id}_{V_{\vec{n}_2}})$ one can split~$f_2'$, which has to coincide with~$f_2$ on~$\vec{n}_2$ coordinates, and is some $\gamma\colon W_{\vec{k}} \to A_{\vec{k}_2}$ on~$\vec{k}_2$ coordinates (which are now~$\vec{k}_1$ coordinates after the $\otimes_{\vec{k}_1,\vec{k}_2}$). Similarly, $f_1'$ coincides with~$f_1$ on~$\vec{n}_1$ coordinates, and is some $\delta\colon 1_\mathcal{V} \to A_{\vec{k}_1}\otimes W_{\vec{k}}$ on $\vec{k}_1$ coordinates. On~$\vec{k}_1$ coordinates one has $({\rm Id}_{A_{\vec{k}_1}}\otimes_{\vec{k}_1}\gamma) \circ \delta = g_{\vec{k}_1}$, so the only relation is $((-,({\rm Id}_{A_{\vec{k}_1}}\otimes_{\vec{k}_1}\gamma) \circ \delta),({\rm Id}_{V_{\vec{k}}},-))\sim ((-,\delta),(\gamma,-))$.

Now, $A_{\mathfrak{S}_1}$ and $A_{\mathfrak{S}_2}$ are not objects of $\mathcal{V}^{\otimes n_1+k}$ and $\mathcal{V}^{\otimes n_2+k}$, but are obtained as canonical colimits of such objects, $A_{\mathfrak{S}_1} = \colim_i \vec{A}_{1,i}$ and $A_{\mathfrak{S}_2} = \colim_j \vec{A}_{2,j}$, so $A_{\mathfrak{S}_1}\otimes_{\vec{k}_1,\vec{k}_2}A_{\mathfrak{S}_2} = \colim_{i,j} \vec{A}_{1,i} \otimes_{\vec{k}_1,\vec{k}_2} \vec{A}_{2,j}$ by cocontinuity.
The object $\eta_{\vec{k}_1}\big(\vec{V}\big)=\big(V_{\vec{n}_1},1_{\mathcal{V}^{\otimes k}},V_{\vec{n}_2}\big)$ is compact projective in $\mathcal{E}^{\boxtimes n_1+n_2+k}$ therefore
\begin{gather*}
\Hom_{\mathcal{E}^{\boxtimes n_1+n_2+k}}\big(\eta_{\vec{k}_1}\big(\vec{V}\big),A_{\mathfrak{S}_1}\otimes_{\vec{k}_1,\vec{k}_2}A_{\mathfrak{S}_2}\big) \\
\qquad{} = \colim_{i,j} \Hom_{\mathcal{V}^{\otimes n_1+n_2+k}}\big(\eta_{\vec{k}_1}\big(\vec{V}\big),\vec{A}_{1,i} \otimes_{\vec{k}_1,\vec{k}_2} \vec{A}_{2,j}\big).
\end{gather*}
So a~morphism $f \in \Hom_{\mathcal{E}^{\boxtimes n_1+n_2+k}}\big(\eta_{\vec{k}_1}\big(\vec{V}\big),A_{\mathfrak{S}_1}\otimes_{\vec{k}_1,\vec{k}_2}A_{\mathfrak{S}_2}\big)$ factorises through a single (actually, a linear combination of) $\vec{A}_{1,i} \otimes_{\vec{k}_1,\vec{k}_2} \vec{A}_{2,j}$ as
\[
f \colon\ \eta_{\vec{k}_1}\big(\vec{V}\big) \overset{f_{i,j}}\to \vec{A}_{1,i} \otimes_{\vec{k}_1,\vec{k}_2} \vec{A}_{2,j} \overset{{\rm can}_{1,i} \otimes_{\vec{k}_1,\vec{k}_2} {\rm can}_{2,j}}\longrightarrow A_{\mathfrak{S}_1}\otimes_{\vec{k}_1,\vec{k}_2}A_{\mathfrak{S}_2}.
\]
There it splits as $f_{i,j}= \big({\rm Id}_{\vec{A}_{1,i}} \otimes_{\vec{k}_1,\vec{k}_2} f_2\big) \circ (f_1 , {\rm Id}_{V_{\vec{n}_2}})$, and
\[
f = ({\rm Id}_{A_{\mathfrak{S}_1}} \otimes_{\vec{k}_1,\vec{k}_2} ({\rm can}_{2,j} \circ f_2)) \circ (({\rm can}_{1,i}\otimes_{\vec{k}_1} {\rm Id}_{V_{\vec{k}}}) \circ f_1 , {\rm Id}_{V_{\vec{n}_2}}).
\]
This $f_{i,j}$ is unique up to the relations in the above colimit, namely for $h_1 \otimes_{\vec{k}_1,\vec{k}_2} h_2 \colon \vec{A}_{1,i} \otimes_{\vec{k}_1,\vec{k}_2} \vec{A}_{2,j} \to \vec{A}_{1,i'} \otimes_{\vec{k}_1,\vec{k}_2} \vec{A}_{2,j'}$ over $A_{\mathfrak{S}_1}\otimes_{\vec{k}_1,\vec{k}_2}A_{\mathfrak{S}_2}$ one has $f_{i',j'} = (h_1 \otimes_{\vec{k}_1,\vec{k}_2} h_2) \circ f_{i,j}$ and $f ={\rm can}_{1,i'} \otimes_{\vec{k}_1,\vec{k}_2} {\rm can}_{2,j'} \circ f_{i',j'}$. Split $h_2$ as $(h_{\vec{k}_2}, h_{\vec{n}_2})$, then $f$ decomposes through $f_{i',j'}$ as
\begin{gather*}
\begin{split}
& f = ({\rm Id}_{A_{\mathfrak{S}_1}} \otimes_{\vec{k}_1,\vec{k}_2} ({\rm can}_{2,j'} \circ ({\rm Id}_{{V_{\vec{k}}}'} , h_{\vec{n}_2})\circ f_2)) \\
& \hphantom{f =}{}
\circ (({\rm can}_{1,i'}\otimes_{\vec{k}_1} {\rm Id}_{{V_{\vec{k}}}'}) \circ (h_1 \otimes_{\vec{k}_1} h_{\vec{k}_2})\circ f_1 , {\rm Id}_{V_{\vec{n}_2}}).
\end{split}
\end{gather*}
Set $F_1 = ({\rm can}_{1,i}\otimes_{\vec{k}_1} {\rm Id}_{V_{\vec{k}}}) \circ f_1 = (({\rm can}_{1,i'}\circ h_1)\otimes_{\vec{k}_1} {\rm Id}_{V_{\vec{k}}}) \circ f_1$ and $F_2 = {\rm can}_{2,j'}\circ ({\rm Id}_{{V_{\vec{k}}}'} , h_{\vec{n}_2}) \circ f_2 = {\rm can}_{2,j'}\circ ({\rm Id}_{{V_{\vec{k}}}'} , h_{\vec{n}_2} \circ g_{\vec{n}_2})$. The first decomposition was
\[
(F_1,{\rm can}_{2,j} \circ f_2)= (F_1,{\rm can}_{2,j'}\circ (h_{\vec{k}_2}, h_{\vec{n}_2}) \circ ({\rm Id}_{V_{\vec{k}}}, g_{\vec{n}_2})) = (F_1,F_2 \circ (h_{\vec{k}_2}, {\rm Id}_{V_{\vec{n}_2}}))
\]
and the second is
\[
(({\rm can}_{1,i'}\otimes_{\vec{k}_1} {\rm Id}_{{V_{\vec{k}}}'})\circ (h_1 \otimes_{\vec{k}_1} h_{\vec{k}_2}) \circ f_1,F_2) = (({\rm Id}_{A_{\mathfrak{S}_1}}\otimes_{\vec{k}_1} h_{\vec{k}_2})\circ F_1,F_2)
\]
so the only relation is
\[
(F_1,F_2\circ (h_{\vec{k}_2}, {\rm Id}_{V_{\vec{n}_2}}))\sim (({\rm Id}_{A_{\mathfrak{S}_1}}\otimes_{\vec{k}_1} h_{\vec{k}_2}) \circ F_1,F_2).
\]

\medskip\noindent
\textbf{Step $\boldsymbol{1^{-1}}$} (re-composition in $\mathfrak{S}$). We decompose $f$ using last step as $f = ({\rm Id}_{A_{\mathfrak{S}_1}} \otimes_{\vec{k}_1,\vec{k}_2} f_2) \circ (f_1 , {\rm Id}_{V_{\vec{n}_2}})$ with $f_1\in\Hom_{\mathcal{E}^{\boxtimes n_1+k}}((V_{\vec{n}_1},1_{\mathcal{V}^{\otimes k}}), A_{\mathfrak{S}_1}\otimes_{\vec{k}_1} V_{\vec{k}})$ and $f_2\in\Hom_{\mathcal{E}^{\boxtimes n_2+k}}((V_{\vec{k}},V_{\vec{n}_2}), A_{\mathfrak{S}_2})$. They are described by morphisms $\alpha_1 := (\sigma_1)^{-1}_{(V_{\vec{n}_1},V_{\vec{k}})}(f_1) \in \Hom_{{\rm Sk}_\mathcal{V}(\mathfrak{S}_1)}(V_{\vec{n}_1} \bmrhd \varnothing,\varnothing\lhd_{\vec{k}_1} V_{\vec{k}})$ and $\alpha_2 := (\sigma_2)_{(V_{\vec{k}},V_{\vec{n}_2})}^{-1}(f_2)\in\Hom_{{\rm Sk}_\mathcal{V}(\mathfrak{S}_2)}(V_{\vec{k}}\rhd_{\vec{k}_2}(V_{\vec{n}_2}\bmrhd\varnothing),\varnothing)$. The above relation becomes $(\alpha_1,\beta_2\circ (\gamma\rhd {\rm Id}_{V_{\vec{n}_2}\rhd\varnothing})) \sim (({\rm Id}_\varnothing\lhd\gamma) \circ \alpha_1, \beta_2)$. So the morphism
\[
\sigma_{\vec{V}}^{-1}(f) := ({\rm Id}_\varnothing , \alpha_2) \circ \iota_{\varnothing,V_{\vec{k}},V_{\vec{n}_2}\bmrhd\varnothing} \circ (\alpha_1 , {\rm Id}_{V_{\vec{n}_2}\bmrhd\varnothing})
\]
is well defined, because this relation is killed.

\medskip\noindent
\textbf{Step 4} (isomorphism). One easily checks that $\sigma^{-1}$ defined this way is an inverse to $\sigma$. Let $\alpha\in \Hom_{{\rm Sk}_\mathcal{V}(\mathfrak{S})}\big(\vec{V}\bmrhd\varnothing,\varnothing\big)$ that decomposes as $\alpha= ({\rm Id}_\varnothing , \alpha_2) \circ \iota_{\varnothing,V_{\vec{k}},V_{\vec{n}_2}\bmrhd\varnothing} \circ (\alpha_1 , {\rm Id}_{V_{\vec{n}_2}\bmrhd\varnothing})$, then $\sigma_{\vec{V}}(\alpha) := ({\rm Id}_{A_{\mathfrak{S}_1}} \otimes_{\vec{k}_1,\vec{k}_2} f_2) \circ (f_1 , {\rm Id}_{V_{\vec{n}_2}})$ is already decomposed with $\alpha_1 = (\sigma_1)^{-1}_{(V_{\vec{n}_1},V_{\vec{k}})}(f_1)$ and $\alpha_2 = (\sigma_2)_{(V_{\vec{k}},V_{\vec{n}_2})}^{-1}(f_2)$, so
\[
\sigma_{\vec{V}}^{-1}(\sigma_{\vec{V}}(\alpha)) := ({\rm Id}_\varnothing , \alpha_2) \circ \iota_{\varnothing,V_{\vec{k}},V_{\vec{n}_2}\bmrhd\varnothing} \circ (\alpha_1 , {\rm Id}_{V_{\vec{n}_2}\bmrhd\varnothing}) = \alpha.
\]
Similarly, let $f \in \Hom_{\mathcal{E}^{\boxtimes n_1+n_2+k}}\big(\eta_{\vec{k}_1}\big(\vec{V}\big),A_{\mathfrak{S}_1}\otimes_{\vec{k}_1,\vec{k}_2}A_{\mathfrak{S}_2}\big)$ that decomposes as $f=({\rm Id}_{A_{\mathfrak{S}_1}} \otimes_{\vec{k}_1,\vec{k}_2} f_2) \circ (f_1 , {\rm Id}_{V_{\vec{n}_2}})$ then $\sigma_{\vec{V}}^{-1}(f) := ({\rm Id}_\varnothing , \alpha_2) \circ \iota_{\varnothing,V_{\vec{k}},V_{\vec{n}_2}\bmrhd\varnothing} \circ (\alpha_1 , {\rm Id}_{V_{\vec{n}_2}\bmrhd\varnothing})$ is already decomposed with $f_1 = (\sigma_1)_{(V_{\vec{n}_1},V_{\vec{k}})}(\alpha_1)$ and $f_2 = (\sigma_2)_{(V_{\vec{k}},V_{\vec{n}_2})}(\alpha_2)$, so
\[
\sigma_{\vec{V}}\big(\sigma_{\vec{V}}^{-1}(f)\big) := ({\rm Id}_{A_{\mathfrak{S}_1}} \otimes_{\vec{k}_1,\vec{k}_2} f_2) \circ (f_1 , {\rm Id}_{V_{\vec{n}_2}})=f
\]
and $\sigma^{-1}$ is indeed an inverse to $\sigma$.
\end{proof}

\begin{Remark}
When $\mathcal{V} = \Oqcomfin$, with $A_\mathfrak{S} \simeq \mathscr{S}(\mathfrak{S})$, one obtains the same excision properties as in Theorem~\ref{thmExcOfSS}. One uses repeatedly Theorem~\ref{thmExcOfSS} on $\mathfrak{S}_1\sqcup\mathfrak{S}_2$ on each couple of boundary edges to glue. This gives $\mathscr{S}(\mathfrak{S}) \simeq HH^0_{\vec{k}_1,\vec{k_2}}(\mathscr{S}(\mathfrak{S}_1)\otimes \mathscr{S}(\mathfrak{S}_2)):= \{ x \in \mathscr{S}(\mathfrak{S}_1) \otimes \mathscr{S}(\mathfrak{S}_2)\ /\ \forall 1\leq i\leq k,\ \Delta_{n_1+i}(x) = \operatorname{fl}\circ \Delta^l_{n_1+n_2+k+i}(x)\}$. By Proposition~\ref{propInvHHO} on all couples of edges to glue one gets $HH^0_{\vec{k}_1,\vec{k_2}}((\mathscr{S}(\mathfrak{S}_1), \mathscr{S}(\mathfrak{S}_2)))= \big(\mathscr{S}^R(\mathfrak{S}_1)^{\halft_{\vec{k}_1}} , \mathscr{S}(\mathfrak{S}_2)\big)^{{\rm inv}_{\vec{k}_1,\vec{k_2}}}$. By Theorem~\ref{thmRelationMultiAndRight} $\mathscr{S}^R(\mathfrak{S}_1)^{\halft_{\vec{k}_1}}$ is the internal skein algebra of~$\mathfrak{S}_1$, and one obtains exactly the formulation of Theorem~\ref{thmExcisionAFS}.

Note that we described how to glue two surfaces along many edges at once and Theorem~\ref{thmExcOfSS} describes how to glue only two edges but possibly of the same surface. The two forms of excision are equivalent, in one way by applying it repeatedly as above and in the other way by gluing a~bigon to the two edges of the surface that one wants to glue together.
\end{Remark}

\begin{Remark}\label{rmkChoiceCutting}
This remark answers a natural question arising at the sight of the cutting property of stated skein algebras: why is it not a coevaluation one sees on newly created states when one cuts along an ideal arc? Indeed in the definition one uses $\sum_{\vec{\mu}} v_{\vec{\mu}}\otimes v_{\vec{\mu}}$ though the coevaluation would give $\mathop{\rm coev}(1)= \sum_{\vec{\mu}} v_{\vec{\mu}} \otimes v_{\vec{\mu}}^* \overset{{\rm Id} \otimes \varphi^{-1}}\longmapsto \sum_{\vec{\mu}} v_{\vec{\mu}} \otimes v_{-\overleftarrow{\mu}} C(-\vec{\mu})$ in particular matching~$+$ states to~$-$ states. The answer is that it is indeed given by a coevaluation, but the stated skein algebra of the surface at the right is not the good object: one must take its half-twisted version. Then the half twist re-exchanges~$+$ signs to~$-$ signs and kills the coefficients appearing. In particular we see that there has been a choice in the way the splitting morphism of stated skein algebras is defined, and that this choice seems to determine both the half twist and the identification~$V \simeq V^*$. This is to be put in light with the unicity of stated skein coefficients proved in \cite[Section~3.4]{TTQLe}.
\end{Remark}

\begin{Remark}
Internal skein algebras are defined for any ribbon category $\mathcal{V}$, and coincide with stated skein algebras when $\mathcal{V} = \Oqcomfin$.
Stated skein algebras for ${\rm SL}_n$ were very recently introduced in \cite{LeSikora}, and one can expect to prove they coincide with internal skein algebras for $\mathcal{V} = \mathcal{O}_{q^n}({\rm SL}_n)\text{-}{\rm comod}^{\rm fin}$ for generic $q$ with the very same proof. The authors actually showed it for surfaces with a single boundary interval using excision properties with respect to gluing patterns from both theories. The constructions and arguments of this paper work more generally with any semisimple coribbon Hopf algebras $H$, using the equivalence $H\text{-}{\rm comod} \simeq \Free(H\text{-}{\rm comod}^{\rm fin})$, and are actually \cite{Gunningham}'s candidate for the generalisation of stated skein algebras. The results of this section show that this generalisation extends to multiple markings, and that one obtains excision properties immediately.

Internal skein algebras are defined more generally in \cite{BBJ} for any $E_2$-algebra $\mathcal{A} \in Pr$ under the name moduli algebras, and the skein-theoretic description holds for $\mathcal{A} = \Free(\mathcal{V})$. As both moduli algebras and stated skein algebras can be defined integrally, or at roots of unity, it would be very interesting to understand how they compare in greater generality. So far, there is no skein-theoretic description of the factorization homology used in the construction of moduli algebras, but it seems credible that with extra work one could rewrite this whole theory in these integral or non-semisimple contexts.
\end{Remark}

\subsection*{Acknowledgements}
I would like to thank warmly my three advisors: Francesco Costantino for his guidance throughout the discovery of this subject and the editing of this article; Joan Bellier-Mill\`es for all the time he spent in explanations and David Jordan for very helpful conversations and comments. I~am grateful to Patrick Kinnear for his remarks and advice. I would also like to thank the anonymous referees for their exceptionally detailed feedback. This research took place in the Institut Math\'ematique de Toulouse and was supported by the \'Ecole Normale Sup\'erieure de Lyon.\looseness=-1


\pdfbookmark[1]{References}{ref}
\LastPageEnding

\end{document}